\documentclass[12pt, leqno]{amsart}

\usepackage{
amsmath,
amssymb,
amsthm,
mathrsfs,
amsfonts,
ytableau,
enumitem,
upgreek,
soul,
yhmath,
mathtools,
}

\usepackage{tikz,tikz-cd}
\usetikzlibrary{decorations.pathreplacing,decorations.pathmorphing,calligraphy}
\usetikzlibrary{arrows,arrows.meta}
\usetikzlibrary{shapes,shapes.geometric,shapes.symbols}
\usetikzlibrary{matrix,calc}

\usepackage[poly,all]{xy}
\usepackage[colorinlistoftodos, textwidth = 2.3cm]{todonotes}
\usepackage{graphicx}
\usepackage{multicol}
\usepackage{cancel}
\usepackage{pict2e,picture}

\usepackage[colorlinks=true, pdfstartview=FitV, linkcolor=blue,citecolor=blue, urlcolor=blue]{hyperref}
\usepackage[capitalise, nameinlink]{cleveref}

\ytableausetup{mathmode, boxsize=1.1em,aligntableaux=center}

\newtheorem{theorem}{Theorem}[section]

\newtheorem{lemma}[theorem]{Lemma}
\newtheorem{corollary}[theorem]{Corollary}

\newtheorem*{claim*}{Claim}

\theoremstyle{definition}

\newtheorem{example}[theorem]{Example}
\newtheorem{definition}[theorem]{Definition}
\newtheorem{remark}[theorem]{Remark}

\numberwithin{equation}{section}
\numberwithin{figure}{section}
\numberwithin{table}{section}

\def\Z{\mathbb Z}
\def\C{\mathbb C}

\newcommand{\nc}{\newcommand}

\nc{\calI}{\mathcal{I}}
\nc{\calF}{\mathcal{F}}
\nc{\SG}{\Sigma}
\nc{\frakS}{\mathfrak{S}}
\nc{\PCT}{\mathrm{PCT}}
\nc{\SPCT}{\mathrm{SPCT}}
\nc{\RT}{\mathrm{RT}}
\nc{\SRT}{\mathrm{SRT}}
\nc{\RCT}{\mathrm{RCT}}
\nc{\SRCT}{\mathrm{SRCT}}
\nc{\SYCT}{\mathrm{SYCT}}
\nc{\SPYCT}{\mathrm{SPYCT}}
\nc{\stan}{\mathrm{stan}}
\nc{\Span}{\mathrm{span}}
\nc{\comp}{\mathrm{comp}}
\nc{\rmst}{\mathrm{st}}
\nc{\bfw}{\mathbf{w}}
\nc{\bfv}{\mathbf{v}}
\nc{\Des}{\mathrm{Des}}
\nc{\set}{\mathrm{set}}
\nc{\wt}{\mathrm{wt}}
\nc{\ch}{\mathrm{ch}}
\nc{\id}{\mathrm{id}}
\nc{\Sym}{\mathrm{Sym}}
\nc{\Qsym}{\mathrm{QSym}}
\nc{\Nsym}{\mathrm{NSym}}
\nc{\sh}{\mathrm{sh}}
\nc{\bfS}{\mathbf{S}}
\nc{\bfm}{\mathbf{m}}
\nc{\hbfS}{\widehat{\mathbf{S}}}
\nc{\bfF}{\mathbf{F}}
\nc{\calS}{\mathcal{S}}
\nc{\hcalS}{\widehat{\mathcal{S}}}
\nc{\alphamax}{\alpha_{\rm max}}
\nc{\brho}{\overline{\rho}}
\nc{\calV}{\mathcal{V}}
\nc{\calX}{\mathcal{X}}
\nc{\calR}{\mathcal{R}}
\nc{\calG}{\mathcal{G}}
\nc{\tal}{\lambda(\alpha)}
\nc{\tbe}{\widetilde{\beta}}
\nc{\opi}{\overline{\pi}}
\nc{\rmtop}{\mathrm{top}}
\nc{\rad}{\mathrm{rad}}
\nc{\bfP}{\mathbf{P}}
\nc{\SET}{\mathrm{SET}}
\nc{\SIT}{\mathrm{SIT}}
\nc{\rev}{\mathrm{r}}
\nc{\Th}{\theta}
\nc{\htau}{\widehat{\tau}}
\nc{\hT}{\widehat{T}}
\nc{\mPhi}{\Phi}
\nc{\mPsi}{\Psi}
\nc{\hmPsi}{\widehat{\Psi}}
\nc{\mpsi}{\psi}
\nc{\mGam}{\Gamma}
\nc{\tcd}{\mathtt{cd}}
\nc{\trd}{\mathtt{rd}}
\nc{\trcd}{\mathtt{rcd}}
\nc{\rmr}{\mathrm{r}}
\nc{\rmc}{\mathrm{c}}
\nc{\rmt}{\mathrm{t}}
\nc{\bubact}{\,\scalebox{0.6}{$\bullet$}\,}
\nc{\hbubact}{\,\scalebox{0.6}{$\widehat{\bullet}$}\,}
\nc{\col}{\mathrm{col}}
\nc{\row}{\mathrm{row}}
\nc{\calE}{\mathcal{E}}
\nc{\calT}{\mathscr{T}}
\nc{\frakP}{\mathfrak{P}}
\nc{\sfT}{\mathsf{T}}
\nc{\calEsa}{\mathcal{E}^\sigma(\alpha)}
\nc{\tauC}{\tau_{\scalebox{0.5}{$C$}}}
\nc{\sytabC}{\sytab_{\scalebox{0.5}{$C$}}}
\nc{\bbfP}{\overline{\bfP}}
\nc{\pr}{\mathbf{pr}}
\nc{\Ups}{\Upsilon}
\nc{\pact}{\diamond}
\nc{\tauE}{\tau_{\scalebox{0.5}{$E$}}}
\nc{\tauF}{\tau_{\scalebox{0.5}{$F$}}}
\nc{\tauG}{\tau_{\scalebox{0.5}{$G$}}}
\nc{\rtE}{T_{\scalebox{0.5}{$E$}}}
\nc{\rtF}{T_{\scalebox{0.5}{$F$}}}
\nc{\rtG}{T_{\scalebox{0.5}{$G$}}}
\nc{\oPaE}{\overline{\Phi}_{\alpha_E}}
\nc{\oPaF}{\overline{\Phi}_{\alpha_F}}
\nc{\oPaG}{\overline{\Phi}_{\alpha_G}}
\nc{\tab}{\tau}
\nc{\sytab}{\widehat{\tau}}
\nc{\hatE}{\widehat{E}}
\nc{\hcalE}{\widehat{\calE}}
\nc{\hatC}{\widehat{C}}
\nc{\bal}{{\boldsymbol{\upalpha}}}
\nc{\bbe}{{\boldsymbol{\upbeta}}}
\nc{\bgam}{{\boldsymbol{\upgamma}}}
\nc{\bdel}{{\boldsymbol{\updelta}}}
\nc{\weakcon}{\odot}
\nc{\calB}{\mathcal{B}}
\nc{\calM}{\mathcal{M}}
\nc{\ocalM}{\overline{\mathcal{M}}}
\nc{\calN}{\mathcal{N}}
\nc{\g}{\mathfrak{g}}
\nc{\la}{\lambda}
\nc{\h}{\mathfrak{h}}
\nc{\te}{\widetilde{e}}
\nc{\tf}{\widetilde{f}}
\nc{\mf}{\mathfrak}
\nc{\cP}{\mathscr{P}}
\nc{\gl}{\mathfrak{gl}}
\nc{\ldalpha}{\lambda(\alpha)}
\nc{\SRIT}{\mathrm{SRIT}}
\nc{\re}{\mathrm{rev}}
\nc{\otau}{\overline{\tau}}
\nc{\rtop}{{\rm top}}
\nc{\sfc}{\mathsf{c}}
\nc{\sfr}{\mathsf{r}}
\nc{\tH}{\mathtt{H}}
\nc{\tS}{\mathtt{S}}
\nc{\tV}{\mathtt{V}}
\nc{\tA}{\mathtt{A}}
\nc{\otA}{\overline{\mathtt{A}}}
\nc{\frakG}{\mathfrak{G}}
\nc{\frakX}{\mathfrak{X}}
\nc{\sfcomp}{\mathsf{Comp}}
\nc{\sfp}{\mathsf{p}}
\nc{\calJ}{\mathcal{J}}
\nc{\calK}{\mathcal{K}}
\nc{\calL}{\mathcal{L}}
\nc{\Ext}{\mathrm{Ext}}
\nc{\soc}{\mathrm{soc}}
\nc{\calW}{\mathcal{W}}
\nc{\bfT}{\mathbf{T}}
\nc{\bfB}{\mathbf{B}}
\nc{\bfx}{\mathbf{x}}
\nc{\bfy}{\mathbf{y}}
\nc{\balalp}{\underline{{\boldsymbol{\upalpha}}}}
\nc{\balalpp}[1]{\underline{{\boldsymbol{\upalpha}}}_{(#1)}}
\nc{\TcalTt}{{T^{\calT}}}
\nc{\calTuT}{\calT^T}
\nc{\ovT}{\overline{T}}
\nc{\Hom}{\mathrm{Hom}}
\nc{\rmIm}{\mathrm{Im}}
\nc{\setA}[2]{\mathcal{P}(\tA_{#1;#2})}
\nc{\ebfw}[2]{\mathsf{end}(\bfw_{#1;#2})}
\nc{\bstau}[2]{{\boldsymbol{\tau}}_{#1;#2}}
\nc{\TLb}[1]{{T^\leftarrow_{#1}}}
\nc{\TLba}[2]{{T^\leftarrow_{#1;#2}}}
\nc{\TLbaa}[3]{{T^\leftarrow_{#1;#2;#3}}}
\nc{\tauf}{\boldsymbol{\tau}_{(i)}}
\nc{\taus}{\hat{\boldsymbol{\tau}}_{(i)}}
\nc{\SPCTsa}{\SPCT^\sigma(\alpha)}
\nc{\bfSsa}{\bfS_\alpha^\sigma}
\nc{\bfSsaC}{{\bfS}^\sigma_{\alpha,C}}
\nc{\hbfSsa}{\widehat{\bfS}_\alpha^\sigma}
\nc{\upineq}{\rotatebox{90}{$<$}}
\nc{\downineq}{\rotatebox{270}{$<$}}
\nc{\diagineq}{\rotatebox{135}{$<$}}

\nc{\taufj}{\boldsymbol{\tau}_{(j)}}
\nc{\tausj}{\hat{\boldsymbol{\tau}}_{(j)}}
\nc{\tauT}{\tau_{\bfT}}
\nc{\tauB}{\tau_{\bfB}}

\nc{\coker}{\mathrm{Coker}}
\nc{\sgn}{\mathrm{sgn}}
\nc{\supp}{\mathsf{supp}}

\nc{\rmw}{\mathrm{w}}
\nc{\urmw}{\underline{\mathrm{w}}}

\nc{\projco}{\Phi}
\nc{\injhull}{\epsilon}

\nc{\ra}{\rightarrow}

\nc{\opone}{\overline{\partial^1}}
\nc{\SGO}[2]{\SG_{#1;#2}}
\nc{\SGT}[2]{\SG'_{#1;#2}}
\nc{\longelto}[1]{w_0(#1)}
\nc{\longeltt}[2]{\boldsymbol{w_0}(#1;#2)}
\nc{\tJ}[2]{{\tt J}_{#1;#2}}
\nc{\htJ}[2]{\widehat{{\tt J}}_{#1;#2}}
\nc{\sigmatab}[3]{T_{#3;#2}(#1)}
\nc{\prodss}[2]{\boldsymbol{\Delta}_{(#1);#2}}
\nc{\prodsst}[3]{\boldsymbol{\Delta}^{(#1)}_{#2,#3}}
\nc{\prodsso}[1]{\Delta(#1)}
\nc{\mapu}[1]{\mathsf{m}(#1)}
\nc{\ssA}[2]{A^{1}_{#1;#2}}
\nc{\ssAz}[2]{A^{0}_{#1;#2}}
\nc{\mapeta}[1]{\eta_{#1}}
\nc{\sigmatabsh}[1]{T(#1)}
\nc{\mapsh}[2]{{#1}^{[#2]}}
\nc{\tphi}{\widetilde{\phi}}
\nc{\AZTDuj}{\left[u+2,u+|\tS'_{k_0}|\right]}
\nc{\AOTDuj}{A^1_{T(\deltau);j}}
\nc{\mapG}[1]{\kappa_{#1}}
\nc{\GuSG}{\mapG{u}(\SG^{(|\tS'_{k_0}|)}_{[u+2,\ell(\alpha)]})}
\nc{\GuSGn}[3]{\mapG{#1}\left((\SG_{#2})^{(#3)}_{[2,#2-1]}\right)}
\nc{\GuSGGn}[3]{\mapG{#1}\left((\SG_{#2+1})^{(#3)}_{[2,#2]}\right)}
\nc{\SGG}[2]{(\SG_{#1})_{#2}}
\nc{\SGGG}[3]{(\SG_{#1})_{#2}^{#3}}
\nc{\GuSGGG}[3]{\SG^{(#2)}_{#1}[#3]}
\nc{\subASG}[3]{\SG_{#1;#2;#3}}
\nc{\bdI}{\boldsymbol{I}}
\nc{\PBQ}[3]{\left(\SG_{#1}\right)_{[#2]}^{(#3)}}
\nc{\paralong}[1]{w_0(#1)}
\nc{\pqlong}[2]{w_0(#1;#2)}
\nc{\deltau}{\Delta_u}
\nc{\AZero}{A^0}
\nc{\shiftu}{{\theta}_u}
\nc{\DeltaOmega}[1]{\Delta(#1)}

\title[Homological properties of 0-Hecke modules $\calV_\alpha$]
{Homological properties of 0-Hecke modules for dual immaculate quasisymmetric functions}

\author{Seung-Il Choi$^{1}$, Young-Hun Kim${}^{1}$, Sun-Young Nam$^{2}$, Young-Tak Oh$^{2}$}

\thanks{${}^1$ Center for quantum structures in modules and spaces, Seoul National University, Republic of Korea; E-mail: ignatioschoi@snu.ac.kr, ykim.math@gmail.com.}

\thanks{${}^2$ Department of Mathematics, Sogang University, Republic of Korea; E-mail: synam.math@gmail.com, ytoh@sogang.ac.kr.}

\keywords{Quasisymmetric function, $0$-Hecke algebra, Projective presentation, Injective presentation, Ext-group}

\date{\today}
\subjclass[2020]{20C08, 05E05, 05E10}

\begin{document}

\begin{abstract}
Let $n$ be a nonnegative integer.
For each composition $\alpha$ of $n$, Berg, Bergeron, Saliola, Serrano, and Zabrocki introduced a cyclic indecomposable $H_n(0)$-module $\calV_\alpha$ with a dual immaculate quasisymmetric function as the image of the quasisymmetric characteristic.
In this paper, we study $\calV_\alpha$'s from the homological viewpoint.
To be precise, we construct a minimal projective presentation of $\calV_\alpha$ and a minimal injective presentation of $\calV_\alpha$ as well.
Using them, we compute $\Ext^1_{H_n(0)}(\calV_\alpha, \bfF_\beta)$ and $\Ext^1_{H_n(0)}( \bfF_\beta, \calV_\alpha)$, where $\bfF_\beta$ is the simple $H_n(0)$-module attached to a composition $\beta$ of $n$.
We also compute $\Ext_{H_n(0)}^i(\calV_\alpha,\calV_{\beta})$ when $i=0,1$ and $\beta \le_l \alpha$,
where $\le_l$ represents the lexicographic order on compositions.
\end{abstract}

\maketitle
\tableofcontents

\section{Introduction}
The first systematic work on the representation theory of the $0$-Hecke algebras was made by Norton~\cite{79Norton}, who completely classified all projective indecomposable modules and simple modules, up to isomorphism,
for all $0$-Hecke algebras of finite type.
In case of $H_n(0)$, the $0$-Hecke algebra of type $A_{n-1}$,
they are naturally parametrized by compositions of $n$.
For each composition $\alpha$ of $n$, let us denote by $\bfP_\alpha$ and $\bfF_\alpha$
the projective indecomposable module and the simple module corresponding $\alpha$, respectively
(see Subsection~\ref{subsec: PIM}).
These modules were again studied intensively in the 2000s (for instance, see~\cite{11Denton, 06HNT, 16Huang}).
In particular, Huang~\cite{16Huang} studied the induced modules $\bfP_{\bal}$ of projective indecomposable modules by using the combinatorial objects called {\it standard ribbon tableaux},
where $\bal$ in bold-face ranges over the set of generalized compositions.

In~\cite{96DKLT, 97KT}, it was shown that the representation theory of the 0-Hecke algebras of type $A$ has a deep connection to
the ring $\Qsym$ of quasisymmetric functions.
Letting $\calG_0(H_n(0))$ be the Grothendieck group of the category of finitely generated $H_n(0)$-modules,
their direct sum over all $n\ge 0$ endowed with the induction product is isomorphic to $\Qsym$ via the {\em quasisymmetric characteristic}
\begin{align*}
\ch : \bigoplus_{n \ge 0} \calG_0(H_n(0)) \ra \Qsym, \quad [\bfF_{\alpha}] \mapsto F_{\alpha}.
\end{align*}
Here, for a composition $\alpha$ of $n$, 
$[\bfF_\alpha]$ is the equivalence class of $\bfF_\alpha$ inside $\calG_0(H_n(0))$ and $F_{\alpha}$ is the fundamental quasisymmetric function attached to $\alpha$
(for more information, see Subsection~\ref{subsec: 0-Hecke alg}).

Suppose that $\alpha$ ranges over the set of all compositions of $n$.
In the mid-2010s, Berg, Bergeron, Saliola, Serrano, and Zabrocki~\cite{14BBSSZ} introduced
the {\it immaculate functions} $\frakS_\alpha$
by applying noncommutative Bernstein operators to the constant power series $1$,
the identity of the ring $\Nsym$ of noncommutative symmetric functions.
These functions form a basis of $\Nsym$.
Then, the authors defined the {\it dual immaculate function} $\frakS^\ast_\alpha$ as the quasisymmetric function dual to $\frakS_\alpha$ under the appropriate  pairing between $\Qsym$ and $\Nsym$,
thus $\frakS^\ast_\alpha$'s also form a basis of $\Qsym$.
Due to their nice properties, the immaculate and dual immaculate functions have since drawn the attention of many mathematicians
(see~\cite{17BBSSZ,16BSZ,16Camp,17Camp,16GY,17Grin,21MS}).
In a subsequent paper~\cite{15BBSSZ}, the same authors successfully construct a cyclic indecomposable $H_n(0)$-module $\calV_\alpha$ with $\ch(\calV_\alpha)=\frakS^\ast_\alpha$ by using the combinatorial objects called {\it standard immaculate tableaux}.
Although several notable properties have recently been revealed  in~\cite{20CKNO2, 21JKLO},
the structure of $\calV_\alpha$ is not yet well known, especially compared to $\frakS^\ast_\alpha$.

The studies of the $0$-Hecke algebras from the homological viewpoint can be found in~\cite{89Cab, 02DHT, 05Fayers}.
For type $A$, Duchamp, Hivert, and Thibon~\cite[Section 4]{02DHT} construct all nonisomorphic 2-dimensional indecomposable modules, and use this result to calculate $\Ext^1_{H_n(0)}(\bfF_{\alpha},\bfF_{\beta})$ for all compositions $\alpha, \beta$ of $n$.

Moreover, when $n \le 4$, they show that its {Poincar\'{e} series} is given by the $(\alpha,\beta)$ entry of the inverse of $(-q)$-Cartan matrix.
For all finite types, Fayers~\cite[Section 5]{05Fayers} shows that
$\dim \Ext^1_{\mathcal{\bullet}}(M,N) =1$ or $0$ for all simple modules $M$ and $N$.
He also classifies when the dimension equals $1$.
However, to the best knowledge of the authors,
little is known about Ext-groups other than simple (and projective) modules.

In this paper, we study homological properties of $\calV_\alpha$'s.
To be precise, we explicitly describe a minimal projective presentation and a minimal injective presentation of $\calV_\alpha$.
By employing these presentations, we calculate
\[
\Ext_{H_n(0)}^1(\calV_\alpha, \bfF_{\beta})
\quad \text{and} \quad
\Ext_{H_n(0)}^1(\bfF_{\beta}, \calV_\alpha).
\]
In addition, we calculate
\[
\Hom_{H_n(0)}(\calV_\alpha, \calV_\beta)
\quad \text{and} \quad
\Ext_{H_n(0)}^1(\calV_\alpha, \calV_\beta)
\]
for all $\beta \le_l \alpha$, where $\le_l$ represents the lexicographic order on compositions.
In the following, let us explain our results in more detail.

Let $\alpha = (\alpha_1,\alpha_2,\ldots, \alpha_{\ell(\alpha)})$ be a composition of $n$.
The first main result concerns a minimal projective presentation of $\calV_\alpha$.
The projective cover, $\projco : \bfP_{\alpha} \ra \calV_\alpha$, of $\calV_\alpha$ has already been provided in~\cite[Theorem 3.2]{20CKNO2}.
Let $\calI(\alpha) := \{1 \le i \le \ell(\alpha)-1 \mid \alpha_{i+1} \neq 1 \}$ and for each $i \in \calI(\alpha)$, let $\bal^{(i)}$ be the generalized composition
\[
(\alpha_1, \alpha_2, \ldots, \alpha_{i-1}, \alpha_i +1, \alpha_{i+1} - 1) \oplus (\alpha_{i+2} , \alpha_{i+3}, \ldots, \alpha_{\ell(\alpha)}).
\]
Then we construct a $\C$-linear map
\begin{equation*}
\begin{tikzcd}
\displaystyle
\partial_1: \bigoplus_{i \in \calI(\alpha)} \bfP_{\bal^{(i)}} \arrow[r] & \bfP_\alpha,
\end{tikzcd}
\end{equation*}
which turns out to be an $H_n(0)$-module homomorphism.
Additionally, we show that
\[
\ker (\projco) = \rmIm(\partial_1)
\quad \text{and} \quad
\ker(\partial_1) \subseteq \rad \left(\bigoplus_{i \in \calI(\alpha)} \bfP_{\bal^{(i)}} \right).
\]
Hence we obtain the following minimal projective presentation of $\calV_\alpha$
\begin{equation*}
\begin{tikzcd}
\displaystyle \bigoplus_{i \in \calI(\alpha)} \bfP_{\bal^{(i)}} \arrow[r, "\partial_1"] & \bfP_{\alpha} \arrow[r, "\Phi"] & \calV_\alpha \arrow[r] & 0,
\end{tikzcd}
\end{equation*}
which enables us to derive that
\[
\Ext_{H_n(0)}^1(\calV_\alpha,\bfF_{\beta})
\cong \begin{cases}
\mathbb C & \text{if $\beta \in \calJ(\alpha)$,}\\
0 & \text{otherwise}
\end{cases}
\]
with $\calJ(\alpha) := \bigcup_{i \in \calI(\alpha)} [\bal^{(i)}]$.
Here, given a generalized composition
$\bal = \alpha^{(1)} \oplus \alpha^{(2)} \oplus \cdots \oplus  \alpha^{(p)}$,
we are using the notation $[\bal]$ to denote the set of all compositions of the form
$$
\alpha^{(1)} \  \square \  \alpha^{(2)} \ \square \ \cdots \  \square  \ \alpha^{(p)},
$$
where $\square$ is the {\it concatenation} $\cdot$ or the {\it near concatenation} $\odot$ (Theorem~\ref{main thm for V}).

The second main result concerns a minimal injective presentation of $\calV_\alpha$.
Since $H_n(0)$ is a Frobenius algebra, every finitely generated injective $H_n(0)$-module is projective.
But, unlike the projective cover of $\calV_\alpha$, there are no known results for an injective hull of $\calV_\alpha$.
We consider the generalized composition
\begin{align*}
\balalp :=
(\alpha_{k_1} -1) \oplus (\alpha_{k_2}-1) \oplus \cdots \oplus (\alpha_{k_{m-1}}-1) \oplus (\alpha_{k_m},1^{\ell(\alpha)-1}),
\end{align*}
where
\[
\{k_1 < k_2 < \cdots < k_m\}=\{ 1\le i \le \ell(\alpha) : \alpha_{i} > 1\}.
\]
Then we construct an injective $H_n(0)$-module homomorphism $\injhull: \calV_\alpha \ra \bfP_{\balalp}$ and prove that it is an injective hull of $\calV_\alpha$, equivalently, $\soc(\bfP_{\balalp}) \subseteq \injhull(\calV_\alpha)$ (Theorem~\ref{Thm:injective hull V}).
The next step is to find a map $\partial^1: \bfP_{\balalp} \to \bdI$ with $\bdI$ injective such that
\begin{equation*}
\begin{tikzcd}
0 \arrow [r] & \calV_\alpha \arrow[r, "\injhull"] & \bfP_{\balalp} \arrow[r, "\partial^1"] & \bdI
\end{tikzcd}
\end{equation*}
is a minimal injective presentation.
To do this, to each index $1 \leq j \leq m$ we assign the generalized composition
$$
\balalpp{j}: =
\begin{cases}
(\alpha_{k_1}-1) \oplus \cdots \oplus  (\alpha_{k_j} - 2) \oplus \cdots \oplus
(\alpha_{k_m}, 1^{\ell(\alpha)-k_j+1}) \oplus (1^{k_j-1}) & \text{ if }  1 \leq j < m, \\
(\alpha_{k_1}-1 ) \oplus \cdots \oplus  (\alpha_{k_{m-1}}-1) \oplus \left((\alpha_{k_m}-1,1^{\ell(\alpha)-k_j+1}) \cdot (1^{k_j-1})\right) & \text{ if }  j  = m.
\end{cases}
$$
Then we construct a $\C$-linear map
\begin{equation*}
\begin{tikzcd}
\partial^1: \bfP_{\balalp} \arrow[r] & \displaystyle
\bdI := \bigoplus_{1 \leq j \leq m} \bfP_{\balalpp{j}},
\end{tikzcd}
\end{equation*}
which turns out to be an $H_n(0)$-module homomorphism.
We also show that
\[
{\rm Im}(\injhull) = \ker(\partial^1)
\quad \text{ and } \quad
\soc\left( \bdI \right)  \subseteq \mathrm{Im}(\partial^1).
\]
Hence we have the following minimal injective presentation of $\calV_\alpha$:
\begin{displaymath}
\begin{tikzcd}
0 \arrow[r] & \calV_\alpha \arrow[r, "\injhull"] & \bfP_{\balalp} \arrow[r, "\partial^1"] &
\displaystyle \bdI
\end{tikzcd}
\end{displaymath}
Let $\Omega^{-1}(\calV_\alpha)$ be the \emph{cosyzygy module} of $\calV_\alpha$,
the cokernel of $\injhull$.
Applying the formula
$\Ext_{H_n(0)}^1(\bfF_{\beta},\calV_\alpha) \cong {\rm Hom}_{H_n(0)}(\bfF_{\beta}, \Omega^{-1}(\calV_\alpha))$
to this minimal injective presentation enables us to derive that
\[
\Ext_{H_n(0)}^1(\bfF_{\beta},\calV_\alpha)
\cong \begin{cases}
\mathbb C^{[\calL(\alpha):\beta^\rmr]} & \text{if $\beta^\rmr \in \calL(\alpha)$,}\\
0 & \text{otherwise,}
\end{cases}
\]
where
$\calL(\alpha)$ is the multiset $\bigcup_{1 \leq j \leq m} [\balalpp{j}]$, $\beta^\rmr$ the reverse composition of $\beta$, and $[\calL(\alpha):\beta^\rmr]$ the multiplicity of $\beta^\rmr$ in $\calL(\alpha)$
(Theorem~\ref{Thm: Main Section4}).

The third main result
concerns $\Ext^i_{H_n(0)}(\calV_\alpha, \calV_\beta)$ for $i=0,1$.
We show that whenever $\beta \le_l \alpha$,
\begin{equation*}
\Ext_{H_n(0)}^1(\calV_\alpha,\calV_{\beta}) = 0
\qquad \text{and} \qquad
\Ext_{H_n(0)}^0(\calV_\alpha, \calV_{\beta})
\cong\begin{cases}
\mathbb C & \text{if } \beta = \alpha, \\
0 & \text{otherwise}.
\end{cases}
\end{equation*}
Given a finite dimensional $H_n(0)$-module $M$, we say that $M$ is \emph{rigid} if $\Ext_{H_n(0)}^1(M,M)=0$ and
\emph{essentially rigid} if ${\rm Hom}_{H_n(0)}(\Omega(M),M)=0$, where $\Omega(M)$ is the \emph{syzygy module} of $M$.
With this definition, we also prove that
$\calV_\alpha$ is essentially rigid for every composition $\alpha$ of $n$
(Theorem~\ref{thm: Ext for V V}).
In case of $\beta >_l \alpha$, the structure of $\Ext^i_{H_n(0)}(\calV_\alpha, \calV_\beta)$ for $i=0,1$
is still beyond our understanding.
For instance, each map in $\Ext^0_{H_n(0)}(\calV_\alpha, \calV_\beta)$ is completely determined by the value of
a cyclic generator of $\calV_\alpha$.
However, at the moment it seems difficult to characterize all possible values the generator can have.
Instead, we view $\Ext^0_{H_n(0)}(\calV_\alpha, \calV_\beta)$ as
the set of $H_n(0)$-module homomorphisms from $\bfP_{\alpha}$ to $\calV_{\beta}$ which vanish on $\Omega(\calV_{\alpha})$.
The most important reason for taking this view is that we know a minimal generating set of $\calV_\alpha$
as well as a combinatorial description of $\dim_{\mathbb C}\Ext^0_{H_n(0)}(\bfP_{\alpha}, \calV_\beta)$.
An approach in this direction is given in Theorem~\ref{Prop: Hom V V}.

This paper is organized as follows.
In Section~\ref{sec: preliminaries}, we introduce the prerequisites
on the $0$-Hecke algebra including the quasisymmetric characteristic, standard ribbon tableaux, standard immaculate tableaux and $H_n(0)$-modules associated to such tableaux.
In Section~\ref{sec: extensions of V by F},
we provide a minimal projective presentation of $\calV_\alpha$ and $\Ext^1_{H_n(0)}(\calV_\alpha, \bfF_\beta)$.
And, in Section~\ref{Sec: extensions of F by V},
we provide a minimal injective presentation of $\calV_\alpha$ and $\Ext^1_{H_n(0)}(\bfF_\beta,\calV_\alpha)$.
In Section~\ref{sec: Ext1 VV}, we investigate $\Ext^i_{H_n(0)}(\calV_\alpha, \calV_\beta)$ for $i=0,1$.
Section~\ref{Sec: Proof of Theorems} is devoted to proving the first and second main results of this paper.
In the last section, we provide some future directions to pursue.

\section{Preliminaries}\label{sec: preliminaries}

In this section, $n$ denotes a nonnegative integer.
Define $[n]$ to be $\{1,2,\ldots, n\}$ if $n > 0$ or $\emptyset$ otherwise.
In addition, we set $[-1]:=\emptyset$.
For positive integers $i\le j$, set $[i,j]:=\{i,i+1,\ldots, j\}$.

\subsection{Compositions and their diagrams}\label{subsec: comp and diag}

A \emph{composition} $\alpha$ of a nonnegative integer $n$, denoted by $\alpha \models n$, is a finite ordered list of positive integers $(\alpha_1, \alpha_2, \ldots, \alpha_k)$ satisfying $\sum_{i=1}^k \alpha_i = n$.
For each $1 \le i \le k$, let us call $\alpha_i$ a \emph{part} of $\alpha$. And we call $k =: \ell(\alpha)$ the \emph{length} of $\alpha$ and $n =:|\alpha|$ the \emph{size} of $\alpha$. For convenience we define the empty composition $\emptyset$ to be the unique composition of size and length $0$.
A \emph{generalized composition} $\bal$ of $n$ is a formal sum $\alpha^{(1)} \oplus \alpha^{(2)} \oplus \cdots \oplus  \alpha^{(k)}$,
where $\alpha^{(i)} \models n_i$ for positive integers $n_i$'s with $n_1 + n_2 + \cdots + n_k = n$.

For $\alpha = (\alpha_1, \alpha_2, \ldots, \alpha_{\ell(\alpha)}) \models n$,
we define the \emph{composition diagram} $\tcd(\alpha)$ of $\alpha$ as a left-justified array of $n$ boxes where the $i$th row from the top has $\alpha_i$ boxes for $1 \le i \le k$.
We also define the \emph{ribbon diagram} $\trd(\alpha)$ of $\alpha$ by the connected skew diagram without $2 \times 2$ boxes, such that the $i$th column from the left has $\alpha_i$ boxes.
Then, for a generalized composition $\bal$ of $n$, we define the {\em generalized ribbon diagram} $\trd(\bal)$ of $\bal$ to be the skew diagram whose connected components are $\trd(\alpha^{(1)}), \trd(\alpha^{(2)}), \ldots, \trd(\alpha^{(k)})$ such that $\trd(\alpha^{(i+1)})$ is strictly to the northeast of $\trd(\alpha^{(i)})$ for $i = 1, 2, \ldots, k-1$.
For example, if $\alpha = (3,1,2)$ and $\bal = (2,1) \oplus (1,1)$, then
\begin{displaymath}
\tcd(\alpha) =
\begin{array}{l}
\begin{ytableau}
~ & ~ & ~\\
~ \\
~ & ~
\end{ytableau}
\end{array},
\quad
\trd(\alpha) =
\begin{array}{l}
\begin{ytableau}
\none & \none & ~\\
~ & ~ & ~ \\
~  \\
~
\end{ytableau}
\end{array},
\quad \text{and} \quad
\trd(\bal) =
\begin{array}{l}
\begin{ytableau}
\none & \none & ~ & ~\\
~ & ~  \\
~
\end{ytableau}
\end{array}.
\end{displaymath}

Given $\alpha = (\alpha_1, \alpha_2, \ldots,\alpha_{\ell(\alpha)}) \models n$ and $I = \{i_1 < i_2 < \cdots < i_k\} \subset [n-1]$,
let
\begin{align*}
&\set(\alpha) := \{\alpha_1,\alpha_1+\alpha_2,\ldots, \alpha_1 + \alpha_2 + \cdots + \alpha_{\ell(\alpha)-1}\}, \\
&\comp(I) := (i_1,i_2 - i_1,\ldots,n-i_k).
\end{align*}
The set of compositions of $n$ is in bijection with the set of subsets of $[n-1]$ under the correspondence $\alpha \mapsto \set(\alpha)$ (or $I \mapsto \comp(I)$).
Let $\alpha^\rmr$ denote the composition $(\alpha_{\ell(\alpha)}, \alpha_{\ell(\alpha)-1},  \ldots, \alpha_1)$.

For compositions $\alpha = (\alpha_{1}, \alpha_{2}, \ldots, \alpha_{k})$ and $\beta= (\beta_{1}, \beta_{2}, \ldots, \beta_{l})$,
let $\alpha \cdot \beta$ be the \emph{concatenation}
and $\alpha \odot \beta$ the \emph{near concatenation} of $\alpha$ and $\beta$.
In other words,
$ \alpha \cdot \beta = (\alpha_1, \alpha_2, \ldots, \alpha_k, \beta_1, \beta_2, \ldots, \beta_l)$ and
$\alpha \odot \beta = (\alpha_1,\ldots, \alpha_{k-1},\alpha_k +
\beta_1,\beta_2, \ldots, \beta_l)$.
For a generalized composition
$\bal = \alpha^{(1)} \oplus \alpha^{(2)} \oplus \cdots \oplus  \alpha^{(m)}$,
define
$$
[\bal] := \{
\alpha^{(1)} \  \square \  \alpha^{(2)} \ \square \ \cdots \  \square  \ \alpha^{(m)}
\mid
\square = \text{$\cdot$ or $\odot$}
\}.
$$

\subsection{The $0$-Hecke algebra and the quasisymmetric characteristic}\label{subsec: 0-Hecke alg}
The symmetric group $\SG_n$ is generated by simple transpositions $s_i := (i \  i \hspace{-.5ex} + \hspace{-.5ex} 1)$ with $1 \le i \le n-1$.
An expression for $\sigma \in \SG_n$ of the form $s_{i_1} s_{i_2} \cdots s_{i_p}$ that uses the minimal number of simple transpositions is called a \emph{reduced expression} for $\sigma$.
The number of simple transpositions in any reduced expression for $\sigma$, denoted by $\ell(\sigma)$, is called the \emph{length} of $\sigma$.

The $0$-Hecke algebra $H_n(0)$ is the $\C$-algebra generated by $\pi_1, \pi_2, \ldots,\pi_{n-1}$ subject to the following relations:
\begin{align*}
\pi_i^2 &= \pi_i \quad \text{for $1\le i \le n-1$},\\
\pi_i \pi_{i+1} \pi_i &= \pi_{i+1} \pi_i \pi_{i+1}  \quad \text{for $1\le i \le n-2$},\\
\pi_i \pi_j &=\pi_j \pi_i \quad \text{if $|i-j| \ge 2$}.
\end{align*}
Pick up any reduced expression $s_{i_1} s_{i_2} \cdots s_{i_p}$ for a permutation $\sigma \in \SG_n$.
It is well known that the element $\pi_{\sigma} := \pi_{i_1} \pi_{i_2} \cdots \pi_{i_p}$ is independent of the choice of reduced expressions
and $\{\pi_\sigma \mid \sigma \in \SG_n\}$ is a basis for $H_n(0)$.
For later use, set
\begin{align*}
\pi_{[i, j]}  := \pi_{i} \pi_{i+1} \cdots \pi_{j}
\quad \text{and} \quad
\pi_{[i, j]^{\mathrm{r}}}  := \pi_{j} \pi_{j-1} \cdots \pi_{i}
\end{align*}
for all $1 \le i \le j \le n-1$.

Let $\calR(H_n(0))$ denote the $\Z$-span of (representatives of) the isomorphism classes of finite dimensional representations of $H_n(0)$.
The isomorphism class corresponding to an $H_n(0)$-module $M$ will be denoted by $[M]$.
The \emph{Grothendieck group} $\calG_0(H_n(0))$ is the quotient of $\calR(H_n(0))$ modulo the relations $[M] = [M'] + [M'']$ whenever there exists a short exact sequence $0 \ra M' \ra M \ra M'' \ra 0$. 
The equivalence classes of irreducible representations of $H_n(0)$ form a free $\Z$-basis for $\calG_0(H_n(0))$. 
Let
\[
\calG := \bigoplus_{n \ge 0} \calG_0(H_n(0)).
\]
According to \cite{79Norton}, there are $2^{n-1}$ distinct irreducible representations of $H_n(0)$.
They are naturally indexed by compositions of $n$. Let $\bfF_{\alpha}$ denote the $1$-dimensional $\C$-vector space corresponding to $\alpha \models n$, spanned by a vector $v_{\alpha}$.
For each $1\le i \le n-1$, define an action of the generator $\pi_i$ of $H_n(0)$ as follows:
\[
\pi_i \cdot v_\alpha = \begin{cases}
0 & i \in \set(\alpha),\\
v_\alpha & i \notin \set(\alpha).
\end{cases}
\]
Then $\bfF_\alpha$ is an irreducible $1$-dimensional $H_n(0)$-representation.

In the following, let us review the connection between $\calG$ and the ring $\Qsym$ of quasisymmetric functions.
Quasisymmetric functions are power series of bounded degree in variables $x_{1},x_{2},x_{3},\ldots$  with coefficients in $\Z$, which are shift invariant in the sense that the coefficient of the monomial $x_{1}^{\alpha _{1}}x_{2}^{\alpha _{2}}\cdots x_{k}^{\alpha _{k}}$ is equal to the coefficient of the monomial $x_{i_{1}}^{\alpha _{1}}x_{i_{2}}^{\alpha _{2}}\cdots x_{i_{k}}^{\alpha _{k}}$ for any strictly increasing sequence of positive integers $i_{1}<i_{2}<\cdots <i_{k}$ indexing the variables and any positive integer sequence $(\alpha _{1},\alpha _{2},\ldots ,\alpha _{k})$ of exponents.

Given a composition $\alpha$, the \emph{fundamental quasisymmetric function} $F_\alpha$ is defined by $F_\emptyset = 1$ and
\[
F_\alpha = \sum_{\substack{1 \le i_1 \le i_2 \le \cdots \le i_k \\ i_j < i_{j+1} \text{ if } j \in \set(\alpha)}} x_{i_1} x_{i_2} \cdots x_{i_k}.
\]
It is well known that $\{F_\alpha \mid \text{$\alpha$ is a composition}\}$ is a basis for $\Qsym$.
In~\cite{96DKLT}, Duchamp, Krob, Leclerc, and Thibon show that, when $\calG$ is equipped with induction product, the linear map
\begin{align*}
\ch : \calG \ra \Qsym, \quad [\bfF_{\alpha}] \mapsto F_{\alpha},
\end{align*}
called the \emph{quasisymmetric characteristic}, is a ring isomorphism.

\subsection{Projective modules of the $0$-Hecke algebra}\label{subsec: PIM}

We begin this subsection by recalling
that $H_n(0)$ is a Frobenius algebra.
Hence it is self-injective, so that finitely generated projective and injective modules coincide
(see \cite[Proposition 4.1]{02DHT}, \cite[Proposition 4.1]{05Fayers}, and \cite[Proposition 1.6.2]{91Benson}).

It was Norton~\cite{79Norton} who first classified all projective indecomposable modules of $H_n(0)$ up to isomorphism,
which bijectively correspond to compositions of $n$.
Later Huang~\cite{16Huang} provided a combinatorial description of these modules and their induction products as well
by using standard ribbon tableaux of generalized composition shape.
We here review Huang's description very briefly.

\begin{definition}\label{def: SRT}
For a generalized composition $\bal$ of $n$, a \emph{standard ribbon tableau} (SRT) of shape $\bal$ is a filling of $\trd(\bal)$ with $\{1,2,\ldots,n\}$ such that
the entries are all distinct,
the entries in each row are increasing from left to right, and
the entries in each column are increasing from top to bottom.
\end{definition}
Let $\SRT(\bal)$ denote the set of all $\SRT$x of shape $\bal$.
For $T \in \SRT(\bal)$, let
$$
\Des(T) := \{i \in [n-1] \mid \text{$i$ appears weakly below $i+1$ in $T$}  \}.
$$
Define an $H_n(0)$-action on the $\C$-span of $\SRT(\bal)$ by 
\begin{align}\label{eq: action for ribbon}
\pi_i \cdot T = \begin{cases}
T & \text{if $i \notin \Des(T)$},\\
0 & \text{if $i$ and $i+1$ are in the same row of $T$},\\
s_i \cdot T & \text{if $i$ appears strictly below $i+1$ in $T$}
\end{cases}
\end{align}
for $1\le i \le n-1$ and $T \in \SRT(\bal)$.
Here $s_i \cdot T$ is obtained from $T$ by swapping $i$ and $i+1$.
The resulting module is denoted by $\bfP_\bal$.
It is known that the set $\{\bfP_\alpha \mid \alpha \models n\}$ forms a complete family of non-isomorphic projective indecomposable $H_n(0)$-modules and $\bfP_\alpha /\rad(\bfP_\alpha) \cong \bfF_\alpha$, where $\rad(\bfP_\alpha)$ is the radical of $\bfP_\alpha$ (for details, see~\cite{16Huang, 79Norton}).

\begin{remark}
It should be pointed out that the ribbon diagram and $H_n(0)$-action used here are slightly different from those in Huang's work \cite{16Huang}.
He describes the $H_n(0)$-action on $\bfP_\bal$ in terms of 
$\opi_i$'s, where $\opi_i= \pi_i -1$. 
On the other hand, we use $\pi_i$'s because 
the $H_n(0)$-action on $\calV_\alpha$ is described in terms of $\pi_i$'s.
This leads us to adjust Huang's ribbon diagram to the form of $\trd(\bal)$.
\end{remark}

Given any generalized composition $\bal$, let $T_\bal \in \SRT(\bal)$ be the $\SRT$ obtained by filling $\trd(\bal)$
with entries $1, 2, \ldots, n$ from top to bottom and from left to right.
Since $\bfP_{\bal}$ is cyclically generated by $T_\bal$, we call $T_\bal$ the \emph{source tableau} of $\bfP_\bal$.
For any $\SRT$ $T$, let $\bfw(T)$ be the word obtained by reading the entries from left to right starting with the bottom row.
Using this reading, Huang~\cite{16Huang} shows the following result.

\begin{theorem}{\rm (\cite[Theorem 3.3]{16Huang})}
\label{thm: bfP isom to calP}
Let $\bal$ be a generalized composition of $n$.
Then $\bfP_\bal$ is isomorphic to $\bigoplus_{\beta \in [\bal]} \bfP_\beta$ as an $H_n(0)$-module.
\end{theorem}

For later use, for every generalized composition $\bal$ of $n$, we define a partial order $\le$ on $\SRT(\bal)$ by
\begin{align*}
T \le T'
\quad \text{if and only if} \quad
T' = \pi_\sigma \cdot T \quad \text{for some $\sigma \in \SG_n$}.
\end{align*}
As usual, whenever $T \le T'$, the notation $[T, T']$ denotes the interval $\{U \in \SRT(\bal) \mid T \le U \le T'\}$.

\subsection{The $H_n(0)$-action on standard immaculate tableaux}\label{0-Hecke action on V}

Noncommutative Bernstein operators were introduced by Berg, Bergeron, Saliola, Serrano, and Zabrocki~\cite{14BBSSZ}.
Applied to the identity of the ring $\Nsym$ of noncommutative symmetric functions,
they yield the \emph{immaculate functions}, which form a basis of $\Nsym$.
Soon after, using the combinatorial objects called standard immaculate tableaux, they constructed indecomposable $H_n(0)$-modules whose quasisymmetric characteristics are the quasisymmetric functions which are dual to immaculate functions (see~\cite{15BBSSZ}).

\begin{definition}\label{def: SIT}
Let $\alpha \models n$. A \emph{standard immaculate tableau} (SIT) of shape $\alpha$  is a filling $\calT$ of the composition diagram $\tcd(\alpha)$ with $\{1,2,\ldots,n\}$ such that the entries are all distinct, the entries in each row increase from left to right, and the entries in the first column increase from top to bottom.
\end{definition}

We denote the set of all SITx of shape $\alpha$ by $\SIT(\alpha)$.
For $\calT \in \SIT(\alpha)$, let
$$
\Des(\calT) := \{i \in [n-1] \mid \text{$i$ appears strictly above $i+1$ in $\calT$}  \}.
$$
Define an $H_n(0)$-action on $\C$-span of $\SIT(\alpha)$ by
\begin{align}\label{eq: action on SIT}
\pi_i \cdot \calT = \begin{cases}
\calT & \text{if $i \notin \Des(\calT)$},\\
0 & \text{if $i$ and $i+1$ are in the first column of $\calT$},\\
s_i \cdot \calT & \text{otherwise}
\end{cases}
\end{align}
for $1\le i \le n-1$ and $\calT \in \SIT(\alpha)$.
Here $s_i \cdot \calT$ is obtained from $\calT$ by swapping $i$ and $i+1$.
The resulting module is denoted by $\calV_\alpha$.

Let $\calT_\alpha \in \SIT(\alpha)$ be the SIT obtained by filling $\tcd(\alpha)$
with entries $1, 2, \ldots, n$ from left to right and from top to bottom.

\begin{theorem}{\rm (\cite{15BBSSZ})}\label{Thm: V is cyclic-generated}
For $\alpha \models n$, $\calV_\alpha$ is a cyclic indecomposable $H_n(0)$-module generated by $\calT_\alpha$
whose quasisymmetric characteristic is the dual immaculate quasisymmetric function $\frakS^*_\alpha$.
\end{theorem}
\medskip

\noindent {\bf Convention.}
Regardless of a ribbon diagram or a composition diagram,
columns are numbered from left to right.
To avoid possible confusion, we adopt the following notation:
\begin{enumerate}[label = {\rm (\roman*)}, leftmargin = 5.0ex, itemsep = 1ex]
\item
Let $T$ be a filling of the ribbon diagram $\trd(\bal)$. 

\begin{enumerate}[label = -, leftmargin = 1ex, itemsep = 0.5ex]
\item $T^i_j$ = the entry at the $i$th box from the top of the $j$th column 

\item $T_j^{-1}$ = the entry at the bottommost box in the $j$th column 

\item $T^\bullet_j$ = the set of all entries in the $j$th column
\end{enumerate}
\item
Let $\calT$ be a filling of the composition diagram $\tcd(\alpha)$.
\begin{enumerate}[label = -, leftmargin = 3ex, itemsep = 0.5ex]
\item $\calT_{i,j}$ = the entry at the box in the $i$th row (from the top)  and in the $j$th column
\end{enumerate}
\end{enumerate}

\section{A minimal projective presentation of $\calV_\alpha$ and $\Ext^1_{H_n(0)}(\calV_\alpha, \bfF_\beta)$}
\label{sec: extensions of V by F}
From now on, $\alpha$ denotes an arbitrarily chosen composition of $n$.
We here construct a minimal projective presentation of $\calV_\alpha$.
Using this, we compute $\Ext^1_{H_n(0)}(\calV_\alpha, \bfF_\beta)$ for each $\beta \models n$.

Firstly, let us introduce necessary terminologies and notation.
Let $A,B$ be finitely generated $H_n(0)$-modules.
A surjective $H_n(0)$-module homomorphism $f:A\to B$ is called an \emph{essential epimorphism} if an $H_n(0)$-module homomorphism $g: X\to A$ is surjective
whenever $f \circ g:X\to B$ is surjective.
A \emph{projective cover} of $A$ is an essential epimorphism $f:P\to A$ with $P$ projective, which always exists and is unique up to isomorphism.
It is well known that $f:P\to A$ is an essential epimorphism
if and only if $\ker(f) \subset \rad(P)$
(for instance, see \cite[Proposition I.3.6]{95ARS}).
For simplicity, when $f$ is clear in the context,
we just write $\Omega(A)$ for $\ker(f)$ and call it the \emph{syzygy module} of $A$.
An exact sequence
\begin{equation*}
\begin{tikzcd}
\displaystyle P_1 \arrow[r,"\partial_1"] & P_0 \arrow[r,"\epsilon"] & A \arrow[r] & 0
\end{tikzcd}
\end{equation*}
with projective modules $P_0$ and $P_1$ is called a \emph{minimal projective presentation} if the $H_n(0)$-module homomorphisms
$\epsilon: P_0 \ra A$
and
$\partial_1: P_1 \ra \Omega(A)$ are projective covers of $A$ and $\Omega(A)$, respectively.

Next, let us review the projective cover of $\calV_\alpha$ obtained in~\cite{20CKNO2}.
Given any $T \in \SRT(\alpha)$, let $\calT_T$ be the filling of $\tcd(\alpha)$ given by $(\calT_T)_{i,j} = T^{j}_{i}$.
Then we define a $\C$-linear map $\projco : \bfP_{\alpha} \ra \calV_\alpha$ by
\begin{align}\label{eq: def of phi}
\projco(T) =
\begin{cases}
\calT_T & \text{if $\calT_T$ is an SIT,}\\
0 & \text{otherwise.}
\end{cases}
\end{align}
For example, if $\alpha = (1,2,2)$ and
\[
T_1 =
\begin{array}{l}
\begin{ytableau}
\none & \none & 4 \\
\none & 2 & 5 \\
1 & 3
\end{ytableau}
\end{array}
\in \SRT(\alpha)
\quad \text{and} \quad
T_2 =
\begin{array}{l}
\begin{ytableau}
\none & \none & 4 \\
\none & 1 & 5 \\
2 & 3
\end{ytableau}
\end{array}
\in \SRT(\alpha),
\]
then
\[
\calT_{T_1} =
\begin{array}{l}
\begin{ytableau}
1 & \none \\
2 & 3 \\
4 & 5
\end{ytableau}
\end{array}
\in \SIT(\alpha)
\quad \text{and} \quad
\calT_{T_2} =
\begin{array}{l}
\begin{ytableau}
2 & \none \\
1 & 3 \\
4 & 5
\end{ytableau}
\end{array}
\notin \SIT(\alpha).
\]
Therefore, $\projco(T_1) = \calT_{T_1}$ and  $\projco(T_2) = 0$.

\begin{theorem}{\rm (\cite[Theorem 3.2]{20CKNO2})}\label{thm: V hom ima of P}
For $\alpha \models n$, $\projco : \bfP_{\alpha} \ra \calV_\alpha$ is a projective cover of $\calV_\alpha$.
\end{theorem}

Now, let us construct a projective cover of $\Omega(\calV_\alpha)$ for each $\alpha \models n$.
To do this, we provide necessary notation.
For each integer $0\le i \le \ell(\alpha)-1$, we set $m_i$ to be $\sum_{j = 1}^{i}\alpha_j$ for $i > 0$ and $m_0 = 0$.
Let
\[
\calI(\alpha) := \{1 \le i \le \ell(\alpha)-1 \mid \alpha_{i+1} \neq 1 \}.
\]
Given $i \in \calI(\alpha)$, let
\[
T^{(i)}_{\alpha} := \pi_{[m_{i-1} + 1, m_{i}]} \cdot T_\alpha
\]
and
\[
\bal^{(i)} := (\alpha_1, \alpha_2, \ldots, \alpha_{i-1}, \alpha_i +1, \alpha_{i+1} - 1) \oplus (\alpha_{i+2} , \alpha_{i+3}, \ldots, \alpha_{\ell(\alpha)}).
\]

Given an SRT $\tau$ of shape $\bal^{(i)}$ $(i \in \calI(\alpha))$,
define $L(\tau)$ to be the filling of $\trd(\alpha)$ whose entries in each column are increasing from top to bottom and whose columns are given as follows: for $1 \le p \le \ell(\alpha)$,
\begin{align}\label{eq: def of L(tau)}
L(\tau)_p^\bullet =
\begin{cases}
\tau_i^\bullet \setminus \{\tau_i^1\} & \text{if $p = i$,}\\
\tau_{i+1}^\bullet \cup \{\tau_i^1\} & \text{if $p = i+1$,}\\
\tau_p^\bullet & \text{otherwise.}\\
\end{cases}
\end{align}

\begin{example}
\ytableausetup{boxsize=1em}

For
$\tau_1 =
\begin{ytableau}
\none & 3 \\
\none & 4 \\
1 & 5 \\
2
\end{ytableau}$ \
and 
$\tau_2 =
\begin{ytableau}
\none & 1 \\
\none & 2 \\
3 & 5 \\
4
\end{ytableau},$
we have
$L(\tau_1) =
\begin{ytableau}
\none & 1 \\
\none & 3 \\
\none & 4 \\
2 & 5
\end{ytableau}$ \
and
$L(\tau_2) =
\begin{ytableau}
\none & 1 \\
\none & 2 \\
\none & 3 \\
4 & 5
\end{ytableau}$.
\end{example}

For each $i \in \calI(\alpha)$, we define a $\C$-linear map $ \partial_1^{(i)}: \bfP_{\bal^{(i)}} \ra H_n(0) \cdot T^{(i)}_{\alpha}$ by
\[
\partial_1^{(i)} (\tau) = \begin{cases}
L(\tau) & \text{if $L(\tau) \in \SRT(\alpha)$,}\\
0 & \text{otherwise.}
\end{cases}
\]
Then we define a $\C$-linear map $\partial_1 : \bigoplus_{i \in \calI(\alpha)} \bfP_{\bal^{(i)}} \rightarrow \bfP_\alpha$ by
$$
\partial_1 := \sum_{i \in \calI(\alpha)} \partial_1^{(i)}.
$$

\begin{theorem}\label{main thm for V}
{\rm (This will be proven in Subsection~\ref{subsec: proof of proj. presentation}.)}
Let $\alpha$ be a composition of $n$.

\begin{enumerate}[label = {\rm (\alph*)}]
\item
$\rmIm(\partial_1) = \Omega(\calV_\alpha)$ and
$\partial_1 : \bigoplus_{i \in \calI(\alpha)} \bfP_{\bal^{(i)}} \rightarrow \Omega(\calV_\alpha)$ is a projective cover of $\Omega(\calV_\alpha)$.

\item
Let
$\calJ(\alpha) := \bigcup_{i \in \calI(\alpha)} [\bal^{(i)}]$.
Then we have
\[
\Ext_{H_n(0)}^1(\calV_\alpha,\bfF_{\beta})
\cong \begin{cases}
\mathbb C & \text{if $\beta \in \calJ(\alpha)$,} \\
0 & \text{otherwise.}
\end{cases}
\]
\end{enumerate}
\end{theorem}

\begin{example}
Let $\alpha = (1,2,1)$.
Then, we have that $\calI(\alpha) = \{1\}$ and $\bal^{(1)} = (2,1) \oplus (1)$.

(a) The map $\partial_1: \bfP_{(2,1) \oplus (1)} \ra \bfP_{(1,2,1)}$ is illustrated in {\sc Figure}~\ref{fig: partial_1 example},
where the entries $i$ in red in each $\SRT$ $T$ are being used to indicate that $\pi_i \cdot T = 0$.
\begin{figure}[ht]
\begin{tikzpicture}
\ytableausetup{boxsize=0.85em}
\def \hpp {8}

\def \hp {1.4}
\def \vp {1.8}

\node at (0.15*\hpp, -2.8*\vp) {\large $\bfP_{(2,1) \oplus (1)}$};
\node at (\hpp, -2.8*\vp) {\large $\bfP_{(1,2,1)}$};

\draw [|->] (3*\hp, \vp) -- (3.75*\hp, \vp);

\draw [dotted, red, line width = 0.4mm]
(0.5*\hp, -2.5*\vp) -- (-2.6*\hp +1*\hp, -0.5*\vp) -- (-2.6*\hp +1*\hp, 1.45*\vp) -- (-1.575*\hp +1*\hp, 1.45*\vp) -- (-0.925*\hp + 1*\hp, 0.95*\vp) -- (-0.925*\hp +1*\hp, -0.5*\vp) -- (2.3*\hp +1*\hp, -0.5*\vp) -- (2.3*\hp +1*\hp, -1.3*\vp) -- (2*\hp, -2.5*\vp) -- (0.5*\hp, -2.5*\vp);

\node[] at (-0.2*\hp,-2.4*\vp) {\large $\ker(\partial_1)$};

\draw [dotted, blue, line width = 0.4mm]
(-0.3*\hp +1*\hp, 2.4*\vp) -- (-1.7*\hp +1*\hp, 2.4*\vp) -- (-1.7*\hp +1*\hp, 1.7*\vp) -- (-0.8*\hp +1*\hp, 1*\vp) -- (-0.8*\hp +1*\hp, -0.4*\vp) -- (1.3*\hp +1*\hp, -0.4*\vp) -- (1.3*\hp +1*\hp, 1.1*\vp) -- (-0.3*\hp +1*\hp, 2.4*\vp);

\node[] at (0, 2*\vp) {$\begin{ytableau}
\none & \none & *(red!5) 4 \\
*(red!5) 1 & *(red!5) 3 \\
*(red!5) 2
\end{ytableau}$};

\node at (0.4*\hp, 2*\vp)  {} edge [out=35,in=325, loop] ();
\node at (0.9*\hp, 2*\vp) {$\pi_1$};

\draw [->] (-0.2*\hp, 1.65*\vp) --  (-0.6*\hp, 1.4*\vp);
\node at (-0.6*\hp, 1.6*\vp) {$\pi_2$};

\draw [->] (0.2*\hp, 1.65*\vp) --  (0.6*\hp, 1.4*\vp);
\node at (0.5*\hp, 1.6*\vp) {$\pi_3$};

\node[] at (-\hp, 1*\vp) {$\begin{ytableau}
\none & \none & 4 \\
\color{red} 1 & 2 \\
3
\end{ytableau}$};

\node at (-\hp + 0.4*\hp, 1*\vp)  {} edge [out=35,in=325, loop] ();
\node at (-\hp + 0.9*\hp, 1*\vp) {$\pi_2$};

\draw[->] (-\hp, 0.65*\vp) --  (-\hp, 0.4*\vp);
\node[] at (-0.85*\hp, 0.525*\vp) {$\pi_3$};

\node[] at (\hp, 1*\vp) {$\begin{ytableau}
\none & \none & *(blue!5) 3 \\
*(blue!5) 1 & *(blue!5) 4 \\
*(blue!5) 2
\end{ytableau}$};

\node at (\hp + 0.4*\hp, 1*\vp)  {} edge [out=35,in=325, loop] ();
\node at (\hp + 1.05*\hp, 1*\vp) {$\pi_1, \pi_3$};

\draw[->] (\hp, 0.65*\vp) --  (\hp, 0.4*\vp);
\node[] at (1.15*\hp, 0.525*\vp) {$\pi_2$};

\node[] at (-\hp, 0*\vp) {$\begin{ytableau}
\none & \none & 3 \\
\color{red} 1 & 2 \\
4
\end{ytableau}$};

\node at (-\hp + 0.4*\hp, 0*\vp)  {} edge [out=35,in=325, loop] ();
\node at (-\hp + 0.9*\hp, 0*\vp) {$\pi_3$};

\draw[->] (-0.6*\hp, 0.65*\vp - \vp) --  (-0.2*\hp, 0.4*\vp - \vp);
\node[] at (-0.55*\hp, 0.45*\vp - \vp) {$\pi_2$};

\node[] at (\hp, 0*\vp) {$\begin{ytableau}
\none & \none & *(green!5) 2 \\
*(green!5) 1 & *(green!5) 4 \\
*(green!5) 3
\end{ytableau}$};

\node at (\hp + 0.4*\hp, 0*\vp)  {} edge [out=35,in=325, loop] ();
\node at (\hp + 0.9*\hp, 0*\vp) {$\pi_2$};

\draw[->] (0.6*\hp, 0.65*\vp - \vp) --  (0.2*\hp, 0.4*\vp - \vp);
\node[] at (0.55*\hp, 0.45*\vp - \vp) {$\pi_3$};

\draw[->] (2*\hp + -0.6*\hp, 0.65*\vp - \vp) --  (2*\hp + -0.2*\hp, 0.4*\vp - \vp);
\node[] at (2*\hp + -0.55*\hp, 0.45*\vp - \vp) {$\pi_1$};

\node[] at (0, -1*\vp) {$\begin{ytableau}
\none & \none & 2 \\
1 & 3 \\
4
\end{ytableau}$};

\node at (0.4*\hp, -1*\vp)  {} edge [out=35,in=325, loop] ();
\node at (1.05*\hp, -1*\vp) {$\pi_2, \pi_3$};

\draw[->] (0.2*\hp, 0.65*\vp - 2*\vp) --  (0.6*\hp, 0.4*\vp - 2*\vp);
\node[] at (0.55*\hp, 0.6*\vp - 2*\vp) {$\pi_1$};

\node[] at (2*\hp, -1*\vp) {$\begin{ytableau}
\none & \none & 1 \\
2 & 4 \\
3
\end{ytableau}$};

\node at (2*\hp + 0.4*\hp, -1*\vp)  {} edge [out=35,in=325, loop] ();
\node at (2*\hp + 1.05*\hp, -1*\vp) {$\pi_1, \pi_2$};

\draw[->] (2*\hp + -0.2*\hp, 0.65*\vp - 2*\vp) --  (2*\hp + -0.6*\hp, 0.4*\vp - 2*\vp);
\node[] at (2*\hp + -0.55*\hp, 0.6*\vp - 2*\vp) {$\pi_3$};

\node[] at (\hp, -2*\vp) {$\begin{ytableau}
\none & \none & 1 \\
\color{red} 2 & 3 \\
4
\end{ytableau}$};

\node at (\hp + 0.4*\hp, -2*\vp)  {} edge [out=35,in=325, loop] ();
\node at (\hp + 1.05*\hp, -2*\vp) {$\pi_1, \pi_3$};

\node[] at (0 + \hpp, 3*\vp) {$\begin{ytableau}
\none & 2 & 4 \\
1 & 3
\end{ytableau}$};

\draw [->] (-0.2*\hp + \hpp, 2.65*\vp) --  (-0.8*\hp + \hpp, 2.35*\vp);
\node at (-0.6*\hp + \hpp, 2.6*\vp) {$\pi_1$};

\draw [->] (0.2*\hp + \hpp, 2.65*\vp) --  (0.8*\hp + \hpp, 2.35*\vp);
\node at (0.6*\hp + \hpp, 2.6*\vp) {$\pi_3$};

\node at (0.4*\hp + \hpp, 3*\vp)  {} edge [out=35,in=325, loop] ();
\node at (0.9*\hp + \hpp, 3*\vp) {$\pi_2$};

\node[] at (-\hp + \hpp, 2*\vp) {$\begin{ytableau}
\none & *(red!5) 1 & *(red!5)  4 \\
*(red!5) \color{red} 2 & *(red!5) 3
\end{ytableau}$};

\draw [->] (-0.8*\hp + \hpp, 1.65*\vp) --  (-0.2*\hp + \hpp, 1.35*\vp);
\node at (-0.6*\hp + \hpp, 1.4*\vp) {$\pi_3$};

\node at (-0.6*\hp + \hpp, 2*\vp)  {} edge [out=35,in=325, loop] ();
\node at (-0.1*\hp + \hpp, 2*\vp) {$\pi_1$};

\node[] at (\hp + \hpp, 2*\vp) {$\begin{ytableau}
\none & \color{red} 2 & 3 \\
1 & 4
\end{ytableau}$};

\draw [->] (0.8*\hp + \hpp, 1.65*\vp) --  (0.2*\hp + \hpp, 1.35*\vp);
\node at (0.6*\hp + \hpp, 1.4*\vp) {$\pi_1$};

\node at (1.4*\hp + \hpp, 2*\vp)  {} edge [out=35,in=325, loop] ();
\node at (1.95*\hp + \hpp, 2*\vp) {$\pi_3$};

\node[] at (0 + \hpp, \vp) {$\begin{ytableau}
\none & *(blue!5) 1 & *(blue!5) 3 \\
*(blue!5) 2 & *(blue!5) 4
\end{ytableau}$};

\draw [->] (0*\hp + \hpp, 0.65*\vp) --  (0*\hp + \hpp, 0.35*\vp);
\node at (0.15*\hp + \hpp, 0.5*\vp) {$\pi_2$};

\node at (0.4*\hp + \hpp, 1*\vp)  {} edge [out=35,in=325, loop] ();
\node at (1.1*\hp + \hpp, 1*\vp) {$\pi_1, \pi_3$};

\node[] at (0  + \hpp, 0) {$\begin{ytableau}
\none & *(green!5) \color{red} 1 & *(green!5) 2 \\
*(green!5) \color{red} 3 & *(green!5) 4
\end{ytableau}$};

\node at (0.4*\hp + \hpp, 0*\vp)  {} edge [out=35,in=325, loop] ();
\node at (0.9*\hp + \hpp, 0*\vp) {$\pi_2$};

\draw [dotted, blue, line width = 0.4mm]
(-0.3*\hp + \hpp, 2.4*\vp) -- (-1.7*\hp + \hpp, 2.4*\vp) -- (-1.7*\hp + \hpp, 1.7*\vp) -- (-0.8*\hp + \hpp, 1*\vp) -- (-0.8*\hp + \hpp, -0.4*\vp) -- (1.3*\hp + \hpp, -0.4*\vp) -- (1.3*\hp + \hpp, 1.1*\vp) -- (-0.3*\hp + \hpp, 2.4*\vp);

\node [] at (0.7*\hp + \hpp, -0.6*\vp) {\large $\Omega(\calV_{(1,2,1)})$};
\end{tikzpicture}
\caption{$\partial_1: \bfP_{(2,1) \oplus (1)} \ra \bfP_{(1,2,1)}$}
\label{fig: partial_1 example}
\end{figure}

(b) Note that $\calJ(\alpha) = [\bal^{(1)}]
= \{(2,2),(2,1,1)\}$.
By Theorem~\ref{main thm for V}(b), we have 
\[
\dim \Ext_{H_n(0)}^1(\calV_{(1,2,1)},\bfF_{\beta}) = 
\begin{cases}
1 & \text{ if } \beta = (2,2) \text{ or } (2,1,1), \\
0 & \text{ otherwise.}
\end{cases}
\]
\end{example}

\section{A minimal injective presentation of $\calV_\alpha$ and $\Ext^1_{H_n(0)}(\bfF_\beta,\calV_\alpha)$}
\label{Sec: extensions of F by V}
As before, $\alpha$ denotes an arbitrarily chosen composition of $n$.
In this section, we construct a minimal injective presentation of $\calV_\alpha$.
Using this, we compute $\Ext^1_{H_n(0)}(\bfF_\beta,\calV_\alpha)$ for each $\beta \models n$.

Let us introduce necessary terminologies and notation.
Let $M,N$ be finitely generated $H_n(0)$-modules with $N \subsetneq M$.
We say that $M$ is an \emph{essential extension} of $N$ if $X\cap N \ne 0$ for all nonzero submodules $X$ of $M$.
An injective $H_n(0)$-module homomorphism $\iota: M \ra \bdI$ with $\bdI$ injective is called an \emph{injective hull} of $M$ if $\bdI$ is an essential extension of $\iota(M)$, which always exists and is unique up to isomorphism.
By~\cite[Theorem 3.30 and Exercise 3.6.12]{99Lam} it follows that $\bdI$ is an injective hull of $M$ if and only if $\iota(M) \supseteq \soc(\bdI)$.
Here $\soc(\bdI)$ is the \emph{socle} of $\bdI$, that is, the sum of all simple submodules of $\bdI$.
When $\iota$ is clear in the context,
we write $\Omega^{-1}(M)$ for $\coker (\iota)$ and call it the \emph{cosyzygy module} of $M$.
An exact sequence
\begin{equation*}
\begin{tikzcd}
0 \arrow[r] & M  \arrow[r,"\iota"] & \bdI_0 \arrow[r,"\partial^1"] &   \bdI_1
\end{tikzcd}
\end{equation*}
with injective modules $\bdI_0$ and $\bdI_1$
is called a \emph{minimal injective presentation} if the $H_n(0)$-module homomorphisms $\iota: M \ra \bdI_0$
and $\partial^1: \Omega^{-1}(M) \ra \bdI_1$
are injective hulls of $M$ and $\Omega^{-1}(M)$, respectively.

We first describe an injective hull of $\calV_\alpha$.
Let
\begin{displaymath}
\calK(\alpha) := \{1 \leq  i \leq \ell(\alpha) \mid \alpha_i > 1\} \cup \{0\}.
\end{displaymath}
We write the elements of $\calK(\alpha)$ as $k_0:=0 < k_1 < k_2 < \cdots < k_m$.
Let
\begin{align*}
\balalp &:=
 (\alpha_{k_1} -1) \oplus (\alpha_{k_2}-1) \oplus \cdots \oplus \left((\alpha_{k_m}-1) \odot (1^{\ell(\alpha)})\right) \\
 & \ =
 (\alpha_{k_1} -1) \oplus (\alpha_{k_2}-1) \oplus \cdots \oplus (\alpha_{k_{m-1}}-1) \oplus (\alpha_{k_m},1^{\ell(\alpha)-1}).
\end{align*}
Let us depict $\trd(\balalp)$ in a pictorial manner.
When $j=0$, we define $\tS_{k_0}$ to be the vertical strip consisting of all the boxes in the first column of $\tcd(\alpha)$.
For $1 \le j \le m$, we define $\tS_{k_j}$ as the horizontal strip consisting of the boxes in the $k_j$th row of $\tcd(\alpha)$ (from the top), except for the leftmost box.
Then $\balalp$ is defined by the generalized composition obtained by placing  $\tS_{k_0},\tS_{k_1},\ldots, \tS_{k_m}$ in the following manner:
\begin{enumerate}[label = {\rm (\roman*)}]
\item $\tS_{k_0}$ is placed horizontally at the topmost row in the new diagram.
\item $\tS_{k_m}$ is placed vertically to the lower-left of $\tS_{k_0}$ so that $\tS_{k_0}$ and $\tS_{k_m}$ are connected.
\item
For $j=m-1,m-2, \ldots, 1$, place $\tS_{k_j}$ vertically to the lower-left of $\tS_{k_{j+1}}$
so that they are not connected to each other.
\end{enumerate}
{\sc Figure}~\ref{fig: rd(balalp)} illustrates the above procedure.
\begin{figure}[ht]
\begin{tikzpicture}
\def\hhh{4.5mm}
\def\vvv{5.5mm}
\def\www{0.20mm}
\def\hhhh{35mm}
\node[above] at (\hhh*4.5,\vvv*-6.3) {\small $\tcd(\alpha)$};
\node[above] at (\hhh*1.8+\hhhh*4,\vvv*-7) {\small $\trd(\balalp)$};
\draw[line width=\www, fill=red!20] (\hhh*3,-\vvv*0) rectangle (\hhh*4,-\vvv*5);
\draw[line width=\www, fill=green!10] (\hhh*4,-\vvv*0) rectangle (\hhh*5,-\vvv*1);

\draw[line width=\www, fill=blue!20] (\hhh*4,-\vvv*3) rectangle (\hhh*6,-\vvv*2);
\draw[line width=\www, fill=yellow!10] (\hhh*4,-\vvv*4) rectangle (\hhh*6,-\vvv*3);
\node[] at (\hhh*3.5,-\vvv*2.5) {\tiny $\tS_{k_0}$};
\node[] at (\hhh*4.5,-\vvv*0.5) {\tiny $\tS_{k_1}$};
\node[] at (\hhh*5,-\vvv*2.5) {\tiny $\tS_{k_2}$};
\node[] at (\hhh*5,-\vvv*3.5) {\tiny $\tS_{k_3}$};
\draw[->,decorate,decoration={snake,amplitude=.4mm,segment length=2mm,post length=1mm}]
(\hhh*6.5,-\vvv*2.5) -- (-\hhh*0+\hhhh,-\vvv*2.5);
\draw[line width=\www, fill=red!20] (\hhh*0+\hhhh,\vvv*0) rectangle (\hhh*5+\hhhh,-\vvv*1);
\node[] at (\hhh*2.5+\hhhh,-\vvv*0.5) {\tiny $\tS_{k_0}$};
\draw[->,decorate,decoration={snake,amplitude=.4mm,segment length=2mm,post length=1mm}]
(\hhh*5+\hhhh,-\vvv*2.5) -- (-\hhh*1+\hhhh*2,-\vvv*2.5);
\draw[line width=\www, fill=red!20] (\hhh*0+\hhhh*2,\vvv*0) rectangle (\hhh*5+\hhhh*2,-\vvv*1);
\draw[line width=\www, fill=yellow!20] (-\hhh*0+\hhhh*2,-\vvv*3) rectangle (\hhh*1+\hhhh*2,-\vvv*1);
\node[] at (\hhh*2.5+\hhhh*2,-\vvv*0.5) {\tiny $\tS_{k_0}$};
\node[] at (\hhh*0.5+\hhhh*2,-\vvv*2) {\tiny $\tS_{k_3}$};
\draw[->,decorate,decoration={snake,amplitude=.4mm,segment length=2mm,post length=1mm}]
(\hhh*5+\hhhh*2,-\vvv*2.5) -- (-\hhh*1.5+\hhhh*3,-\vvv*2.5);
\draw[line width=\www, fill=red!20] (\hhh*0+\hhhh*3,\vvv*0) rectangle (\hhh*5+\hhhh*3,-\vvv*1);
\draw[line width=\www, fill=yellow!20] (-\hhh*0+\hhhh*3,-\vvv*3) rectangle (\hhh*1+\hhhh*3,-\vvv*1);
\draw[line width=\www, fill=blue!20] (-\hhh*1+\hhhh*3,-\vvv*3) rectangle (\hhh*0+\hhhh*3,-\vvv*5);
\node[] at (\hhh*2.5+\hhhh*3,-\vvv*0.5) {\tiny $\tS_{k_0}$};
\node[] at (\hhh*0.5+\hhhh*3,-\vvv*2) {\tiny $\tS_{k_3}$};
\node[] at (-\hhh*0.5+\hhhh*3,-\vvv*4) {\tiny $\tS_{k_2}$};
\draw[->,decorate,decoration={snake,amplitude=.4mm,segment length=2mm,post length=1mm}]
(\hhh*5+\hhhh*3,-\vvv*2.5) -- (-\hhh*2+\hhhh*4,-\vvv*2.5);
\draw[line width=\www, fill=red!20] (\hhh*0+\hhhh*4,\vvv*0) rectangle (\hhh*5+\hhhh*4,-\vvv*1);
\draw[line width=\www, fill=yellow!20] (-\hhh*0+\hhhh*4,-\vvv*3) rectangle (\hhh*1+\hhhh*4,-\vvv*1);
\draw[line width=\www, fill=blue!20] (-\hhh*1+\hhhh*4,-\vvv*3) rectangle (\hhh*0+\hhhh*4,-\vvv*5);
\draw[line width=\www, fill=green!10] (-\hhh*2+\hhhh*4,-\vvv*5) rectangle (-\hhh*1+\hhhh*4,-\vvv*6);
\node[] at (\hhh*2.5+\hhhh*4,-\vvv*0.5) {\tiny $\tS_{k_0}$};
\node[] at (\hhh*0.5+\hhhh*4,-\vvv*2) {\tiny $\tS_{k_3}$};
\node[] at (-\hhh*0.5+\hhhh*4,-\vvv*4) {\tiny $\tS_{k_2}$};
\node[] at (-\hhh*1.5+\hhhh*4,-\vvv*5.5) {\tiny $\tS_{k_1}$};
\end{tikzpicture}
\caption{The construction of $\trd(\balalp)$ when $\alpha=(2,1,3^2,1)$}
\label{fig: rd(balalp)}
\end{figure}
\vskip 2mm

For simplicity, we introduce the following notation:
\begin{enumerate}[label = {$\bullet$}]
\item
For an SIT $\calT$ and a subdiagram $\tS$ of shape of $\calT$,
we denote by $\calT(\tS)$ the set of entries of $\calT$ in $\tS$.

\item
For an SRT $T$ and a subdiagram $\tS$ of shape of $T$,
we denote by $T(\tS)$ the set of entries of $T$ in $\tS$.
\end{enumerate}
For $\calT \in \SIT(\alpha)$, let $\TcalTt$ be the tableau of $\trd(\balalp)$ defined by
\begin{equation*}
(\TcalTt)(\tS_{k_j}) := \calT(\tS_{k_j})
\qquad \text{for } 0 \leq j \leq m.
\end{equation*}
Extending the assignment $\calT \mapsto T^{\calT}$ by linearity,
we define the $\C$-linear map
\begin{equation*}
\injhull: \calV_\alpha \ra \bfP_{\balalp}, \quad \calT \mapsto \TcalTt,
\end{equation*}
which is obviously injective.

\begin{theorem}\label{Thm:injective hull V}
{\rm (This will be proven in Subsection~\ref{Subsec: Proof Par0 is injective hull}.)}
$\injhull: \calV_\alpha \to \bfP_{\balalp}$ is an injective hull of $\calV_\alpha$.
\end{theorem}

For later use, we provide bases of $\injhull(\calV_\alpha)$ and $\Omega^{-1}(\calV_\alpha)$.
From the injectivity of $\injhull$ we derive that
$\injhull(\calV_\alpha)$ is spanned by
$$
\{ T \in \SRT(\balalp) \mid T_j^{1+\delta_{j,m}} > T^1_{m+k_j-1} \text{ for all } 1\leq j \leq m\}
$$
and $\Omega^{-1}(\calV_\alpha)$ is spanned by
$\{T + \injhull(\calV_\alpha) \mid T \in \Theta(\calV_\alpha)\}$ with
\begin{equation}\label{Eq: generators of cosyzygy}
\Theta(\calV_\alpha) := \{ T \in \SRT(\balalp) \mid T_j^{1+\delta_{j,m}} < T^1_{m+k_j-1} \text{ for some }1\leq j \leq m \}.
\end{equation}

\begin{example}
If $\alpha = (1,2,2) \models 5$, then $\calK(\alpha) = \{0,2,3\}$ and $\balalp = (1)\oplus(2,1^2)$.
For
$\tau=
\begin{ytableau}\ytableausetup{boxsize=1em}
1 \\
2 & 4\\
3 & 5
\end{ytableau}
\in \SIT(\alpha)$,
one sees that
$T^{\calT} =
\begin{ytableau}
\none & 1 & 2 & 3\\
\none & 5 \\
4
\end{ytableau}
\in \SRT(\balalp)$.
The map $\injhull:\calV_\alpha \ra  \bfP_{\balalp}$ is illustrated in {\sc Figure}~\ref{Fig: V to P}, where the red entries $i$ in tableaux are being used to indicate that $\pi_i$ acts on them as zero.
\begin{figure}[ht]
\begin{tikzpicture}
\ytableausetup{boxsize=1.3em}
\def \hp{40mm}
\def \hhh{30mm}
\def \vvv{20mm}
\node at (\hhh*0,\vvv*1) {\footnotesize
$\begin{ytableau}
*(red!5) \color{red} 1 \\
*(red!5) 2 & *(red!5) 3 \\
*(red!5) 4 & *(red!5) 5
\end{ytableau}$};

\node at (\hhh*0,\vvv*0) {\footnotesize
$\begin{ytableau}
*(blue!5) \color{red} 1 \\
*(blue!5) \color{red} 2 & *(blue!5) 4 \\
*(blue!5) 3 & *(blue!5) 5
\end{ytableau}$};

\node at (\hhh*0,\vvv*-1) {\footnotesize
$\begin{ytableau}
*(green!5) \color{red} 1 \\
*(green!5) \color{red} 2 & *(green!5) 5 \\
*(green!5) 3 & *(green!5) 4
\end{ytableau}$};
\node at (\hhh*0,\vvv*-1.5) {$\calV_{(1,2,2)}$};
\node[right] at (\hhh*2.5,\vvv*-1.5) { \ $\bfP_{(1)\oplus(2,1,1)}$};

\draw [dotted, red, line width = 0.3mm]
(\hhh*1.1,\vvv*2.5) --
(\hhh*1.1,\vvv*1.5) --
(\hhh*2.1,\vvv*1.5) --
(\hhh*2.1,\vvv*-0.5) --
(\hhh*3.1,\vvv*-0.5) --
(\hhh*3.1,\vvv*0.5) --
(\hhh*4.1,\vvv*0.5) --
(\hhh*4.1,\vvv*2.5);

\draw [|->] (\hhh*0.55, \vvv*0) -- (\hhh*0.95, \vvv*0);

\node[] at (\hhh*3,\vvv*-0.65) {\small $\Omega^{-1}(\calV_\alpha)$};

\draw [dotted, blue, line width = 0.3mm]
(\hhh*1.1,\vvv*1.3) --
(\hhh*2,\vvv*1.3) --
(\hhh*2,\vvv*-1.3) --
(\hhh*1.1,\vvv*-1.3) --
(\hhh*1.1,\vvv*1.3);

\node[] at (\hhh*1.7,\vvv*-1.45) {\small $\injhull(\calV_\alpha)$};


\node at (\hhh*0.15, \vvv*1.1)  {} edge [out=35,in=325, loop] ();
\node at (\hhh*0.15, \vvv*0.1)  {} edge [out=35,in=325, loop] ();
\node at (\hhh*0.2, \vvv*-1.1)  {} edge [out=35,in=325, loop] ();
\node[right] at (\hhh*0.315, \vvv*1.1)  {\footnotesize $\pi_2,\pi_4$};
\node[right] at (\hhh*0.315, \vvv*0.1)  {\footnotesize $\pi_3$};
\node[right] at (\hhh*0.34, \vvv*-1.1)  {\footnotesize $\pi_3,\pi_4$};

\draw[->] (\hhh*0,\vvv*0.7) -- (\hhh*0,\vvv*0.3) node[left,midway] {\footnotesize $\pi_3$};
\draw[->] (\hhh*0,\vvv*-0.3) -- (\hhh*0,\vvv*-0.7) node[left,midway] {\footnotesize $\pi_4$};


\node at (\hp + \hhh*0,\vvv*2) {\footnotesize
$\begin{ytableau}
\none & 1 & 3 & 5 \\
\none & 2 \\
4
\end{ytableau}$};
\node at (\hp + \hhh*1,\vvv*2) {\footnotesize
$\begin{ytableau}
\none & \color{red} 1 & 2 & 5 \\
\none & 4 \\
3
\end{ytableau}$};
\node at (\hp + \hhh*2,\vvv*2) {\footnotesize
$\begin{ytableau}
\none & 1 & \color{red} 3 & 4 \\
\none & 5 \\
2
\end{ytableau}$};

\node at (\hp + \hhh*0,\vvv*1) {\footnotesize
$\begin{ytableau}
\none & *(red!5) \color{red} 1 & *(red!5) 2 & *(red!5) 4 \\
\none & *(red!5) 5 \\
*(red!5) 3
\end{ytableau}$};
\node at (\hp + \hhh*1,\vvv*1) {\footnotesize
$\begin{ytableau}
\none & 1 & \color{red} 3 & 4\\
\none & 2 \\
5
\end{ytableau}$};
\node at (\hp + \hhh*2,\vvv*1) {\footnotesize
$\begin{ytableau}
\none & \color{red} 1 & 2 & 5\\
\none & 3 \\
4
\end{ytableau}$};

\node at (\hp + \hhh*0,\vvv*0) {\footnotesize
$\begin{ytableau}
\none & *(blue!5) \color{red} 1 & *(blue!5) \color{red} 2 & *(blue!5) 3 \\
\none & *(blue!5) 5 \\
*(blue!5) 4
\end{ytableau}$};
\node at (\hp + \hhh*1,\vvv*0) {\footnotesize
$\begin{ytableau}
\none & \color{red} 1 & 2 & 4 \\
\none & 3\\
5
\end{ytableau}$};

\node at (\hp + \hhh*0,\vvv*-1) {\footnotesize
$\begin{ytableau}
\none & *(green!5) \color{red} 1 & *(green!5) \color{red} 2 & *(green!5) 3 \\
\none & *(green!5) 4 \\
*(green!5) 5
\end{ytableau}$};

\node at (\hp + \hhh*0.24, \vvv*2.1)  {} edge [out=35,in=325, loop] ();
\node[right] at (\hp + \hhh*0.4, \vvv*2.1)  {\footnotesize $\pi_1,\pi_3$};
\node at (\hp + \hhh*0.24, \vvv*1.1)  {} edge [out=35,in=325, loop] ();
\node[right] at (\hp + \hhh*0.4, \vvv*1.1)  {\footnotesize $\pi_2,\pi_4$};

\node at (\hp + \hhh*0.24, \vvv*0.1)  {} edge [out=35,in=325, loop] ();
\node[right] at (\hp + \hhh*0.4, \vvv*0.1)  {\footnotesize $\pi_3$};

\node at (\hp + \hhh*0.24, \vvv*-1)  {} edge [out=35,in=325, loop] ();
\node[right] at (\hp + \hhh*0.4, \vvv*-1) {\footnotesize $\pi_3,\pi_4$};

\node at (\hp + \hhh*1.24, \vvv*2.1)  {} edge [out=35,in=325, loop] ();
\node[right] at (\hp + \hhh*1.4, \vvv*2.1)  {\footnotesize $\pi_2$};
\node at (\hp + \hhh*1.24, \vvv*1.1)  {} edge [out=35,in=325, loop] ();
\node[right] at (\hp + \hhh*1.4, \vvv*1.1)  {\footnotesize $\pi_1,\pi_4$};

\node at (\hp + \hhh*1.24, \vvv*0.1)  {} edge [out=35,in=325, loop] ();
\node[right] at (\hp + \hhh*1.4, \vvv*0.1)  {\footnotesize $\pi_2,\pi_4$};

\node at (\hp + \hhh*2.24, \vvv*2.1)  {} edge [out=35,in=325, loop] ();
\node[right] at (\hp + \hhh*2.4, \vvv*2.1)  {\footnotesize $\pi_1,\pi_4$};
\node at (\hp + \hhh*2.24, \vvv*1.1)  {} edge [out=35,in=325, loop] ();
\node[right] at (\hp + \hhh*2.4, \vvv*1.1)  {\footnotesize $\pi_2,\pi_3$};

\node at (\hp + \hhh, \vvv*2.5) {$\vdots$};

\draw[->] (\hp*1.1,\vvv*1.7) -- (\hp*1.2+\hhh*0.6,\vvv*1.3) node[above,midway] {\footnotesize $\pi_4$};

\draw[->] (\hp*1.3,\vvv*1.7) -- (\hp*1.2+\hhh*1.6,\vvv*1.3) node[above,midway] {\footnotesize $\pi_2$};

\draw[->] (\hp*1.2+\hhh,\vvv*1.8) -- (\hp*1.2+\hhh*1.7,\vvv*1.4) node[above,midway] {\footnotesize $\pi_3$};

\draw[->] (\hp*1.5,\vvv*1.7) -- (\hp*1.1,\vvv*1.3) node[below,midway] {\footnotesize $\pi_4$};

\draw[->] (\hp*1.2+\hhh*1.7,\vvv*1.7) -- (\hp*1.2+\hhh*0.0,\vvv*1.3) node[below,midway] {\footnotesize $\pi_2$};

\draw[->] (\hp*1,\vvv*0.7) -- (\hp*1,\vvv*0.4) node[left,midway] {\footnotesize $\pi_3$};

\draw[->] (\hp*1+\hhh,\vvv*0.7) -- (\hp*1+\hhh,\vvv*0.4) node[left,midway] {\footnotesize $\pi_2$};

\draw[->] (\hp*1+\hhh*1.8,\vvv*0.7) -- (\hp*1.3+\hhh,\vvv*0.3) node[below,midway] {\footnotesize $\pi_4$};

\draw[->] (\hp*1,\vvv*-0.3) -- (\hp*1,\vvv*-0.7) node[left,midway] {\footnotesize $\pi_4$};

\draw[->] (\hp*1+\hhh*0.8,\vvv*-0.3) -- (\hp*1.2,\vvv*-0.7) node[below,midway] {\footnotesize $\pi_3$};
\end{tikzpicture}

\caption{$\injhull: \calV_{(1,2,2)} \ra \bfP_{(1)\oplus (2,1,1)}$}
\protect \label{Fig: V to P}
\end{figure}
\end{example}

We next describe an injective hull of $\Omega^{-1}(\calV_\alpha)$.
To do this, we need an $H_n(0)$-module homomorphism $\partial^1: \bfP_{\balalp}\ra \bdI$
with $\bdI$ an injective module satisfying that $\ker(\partial^1)=\injhull(\calV_\alpha)$.

First, we provide the required injective module $\bdI$.
For $1 \leq j \leq m$, define $\balalpp{j}$ to be the generalized composition
$$
\balalpp{j}: = \left\{
\begin{array}{ll}
(\alpha_{k_1}-1) \oplus \cdots \oplus  (\alpha_{k_j} - 2) \oplus \cdots \oplus
(\alpha_{k_m}, 1^{\ell(\alpha)-k_j+1}) \oplus (1^{k_j-1}) & \text{if }  1 \leq j < m, \\
(\alpha_{k_1}-1 ) \oplus \cdots \oplus  (\alpha_{k_{m-1}}-1) \oplus \left((\alpha_{k_m}-1,1^{\ell(\alpha)-k_j+1}) \cdot (1^{k_j-1})\right) & \text{if }  j  = m.
\end{array}
\right.
$$
Then we set
\begin{equation}\label{Eq: I1}
\bdI:=\bigoplus_{1 \leq j \leq m}\bfP_{\balalpp{j}}.
\end{equation}

In the following, we provide a pictorial description of $\trd(\balalpp{j})$.
We begin by recalling that $\trd(\balalp)$ consists of the horizontal strip $\tS_{k_0}$ and the vertical strips $\tS_{k_1},\ldots,\tS_{k_m}$.
For each $-1\le r \le m$, we denote by $\tS'_{k_r}$ the connected horizontal strip of length
\begin{equation*}
|\tS'_{k_r}|:=
\begin{cases}
k_j-1 & \text{ if } r = -1,\\
\ell(\alpha)-k_j+2 & \text{ if } r = 0, \\
|\tS_{k_r}| - \delta_{r,j} & \text{ if } 1 \leq r \leq m,
\end{cases}
\end{equation*}
where $k_{-1} := -1$.
With this preparation, $\balalpp{j}$ is defined to be the generalized composition obtained by placing $\tS'_{k_{-1}},\tS'_{k_0},\tS'_{k_1},\ldots,\tS'_{k_m}$ in the following way:
\begin{enumerate}[label = {\rm (\roman*)}]
\item $\tS'_{k_1}$ is placed vertically to the leftmost column in the diagram we are going to create.

\item For $j = 2,3,\ldots, m$, $\tS'_{k_j}$ is placed vertically to the upper-right of $\tS'_{k_{j-1}}$ so that they are not connected to each other.

\item $\tS'_{k_0}$ is placed horizontally to $\tS'_{k_m}$ so that they are connected.

\item In case where $j \neq m$, $\tS'_{k_{-1}}$ is placed horizontally to the upper-right of $\tS'_{k_0}$ so that they are disconnected.
In case where $j = m$, $\tS'_{k_{-1}}$ is placed horizontally to the upper-right of $\tS'_{k_0}$ so that they are connected.
\end{enumerate}
{\sc Figure}~\ref{Fig: rd balalppj} illustrates the above procedure.
\begin{figure}[ht]
\begin{tikzpicture}
\def\hhh{5mm}
\def\vvv{5.5mm}
\def\www{0.20mm}
\def\hhhh{35mm}
\draw[line width=\www, fill=red!20] (\hhh*0,\vvv*0) rectangle (\hhh*4,-\vvv*1);
\draw[line width=\www, fill=yellow!20] (-\hhh*0,-\vvv*2) rectangle (\hhh*1,-\vvv*1);
\draw[line width=\www, fill=green!10] (\hhh*-1,-\vvv*4) rectangle (\hhh*0,-\vvv*2);
\draw[-,dotted] (\hhh*1,-\vvv*1) to (\hhh*1,\vvv*0);
\draw[-,dotted] (\hhh*2,-\vvv*1) to (\hhh*2,\vvv*0);
\draw[-,dotted] (\hhh*3,-\vvv*1) to (\hhh*3,\vvv*0);
\draw[-,dotted] (\hhh*0,-\vvv*2) to (\hhh*1,-\vvv*2);
\node[] at (\hhh*2,-\vvv*0.5) {\tiny $\tS_{k_0}$};
\node[] at (\hhh*0.5,-\vvv*1.5) {\tiny $\tS_{k_2}$};
\node[] at (\hhh*-0.5,-\vvv*3) {\tiny $\tS_{k_1}$};
\node[above] at (\hhh*2,-\vvv*5.1) {\small $\trd(\balalp)$};
\draw[->,decorate,decoration={snake,amplitude=.4mm,segment length=2mm,post length=1mm}]
(\hhh*4,-\vvv*1.5) -- (-\hhh*1.5+\hhhh,-\vvv*1.5);
\draw[line width=\www, fill=green!10] (\hhh*-1+\hhhh,-\vvv*4) rectangle (\hhh*0+\hhhh,-\vvv*3);
\node at (-\hhh*0.4+\hhhh,-\vvv*3.6) {\tiny $\tS'_{k_1}$};
\def\hhhh{26mm}
\draw[->,decorate,decoration={snake,amplitude=.4mm,segment length=2mm,post length=1mm}]
(\hhh*2.5+\hhhh,-\vvv*1.5) -- (-\hhh*1.5+\hhhh*2,-\vvv*1.5);
\draw[line width=\www, fill=green!10] (\hhh*-1+\hhhh*2,-\vvv*4) rectangle (\hhh*0+\hhhh*2,-\vvv*3);
\draw[line width=\www, fill=yellow!20] (-\hhh*0+\hhhh*2,-\vvv*3) rectangle (\hhh*1+\hhhh*2,-\vvv*2);
\node at (-\hhh*0.4+\hhhh*2,-\vvv*3.6) {\tiny $\tS'_{k_1}$};
\node at (\hhh*0.55+\hhhh*2,-\vvv*2.6) {\tiny $\tS'_{k_2}$};
\def\hhhh{25mm}
\draw[->,decorate,decoration={snake,amplitude=.4mm,segment length=2mm,post length=1mm}]
(\hhh*2+\hhhh*2,-\vvv*1.5) -- (-\hhh*1+\hhhh*3,-\vvv*1.5);
\draw[line width=\www, fill=green!10] (\hhh*-1+\hhhh*3,-\vvv*4) rectangle (\hhh*0+\hhhh*3,-\vvv*3);
\draw[line width=\www, fill=yellow!20] (-\hhh*0+\hhhh*3,-\vvv*3) rectangle (\hhh*1+\hhhh*3,-\vvv*2);
\draw[line width=0, fill=red!10] (\hhh*0+\hhhh*3,-\vvv*2) rectangle (\hhh*4+\hhhh*3,-\vvv*1);
\draw[line width=\www] (\hhh*0+\hhhh*3,-\vvv*2) rectangle (\hhh*4+\hhhh*3,-\vvv*1);
\draw[-,dotted] (\hhh*2+\hhhh*3,-\vvv*2) -- (\hhh*2+\hhhh*3,\vvv*-1);
\draw[-,dotted] (\hhh*3+\hhhh*3,-\vvv*2) to (\hhh*3+\hhhh*3,\vvv*-1);
\draw[-,dotted] (\hhh*1+\hhhh*3,-\vvv*2) -- (\hhh*1+\hhhh*3,\vvv*-1);
\node at (-\hhh*0.4+\hhhh*3,-\vvv*3.6) {\tiny $\tS'_{k_1}$};
\node at (\hhh*0.55+\hhhh*3,-\vvv*2.6) {\tiny $\tS'_{k_2}$};
\draw [decorate,
    decoration = {calligraphic brace}] (\hhh*0+\hhhh*3,-\vvv*0.8) --  (\hhh*4+\hhhh*3,-\vvv*0.8);
\node at (\hhh*2+\hhhh*3,-\vvv*0.3) {\tiny $\ell(\alpha)-k_1+2$};
\node at (\hhh*2+\hhhh*3,-\vvv*1.5) {\tiny $\tS'_{k_0}$};
\def\hhhh{28mm}
\draw[->,decorate,decoration={snake,amplitude=.4mm,segment length=2mm,post length=1mm}]
(\hhh*3+\hhhh*3,-\vvv*1.5) -- (-\hhh*1+\hhhh*4,-\vvv*1.5);
\node[above] at (\hhh*2.5+\hhhh*4,-\vvv*5.1) {\small $\trd(\balalpp{1})$};
\draw[line width=\www, fill=green!10] (\hhh*-1+\hhhh*4,-\vvv*4) rectangle (\hhh*0+\hhhh*4,-\vvv*3);
\draw[line width=\www, fill=yellow!20] (-\hhh*0+\hhhh*4,-\vvv*3) rectangle (\hhh*1+\hhhh*4,-\vvv*2);
\draw[line width=\www, fill=red!20] (\hhh*4+\hhhh*4,\vvv*0) rectangle (\hhh*5.2+\hhhh*4,-\vvv*1);
\draw[line width=\www, fill=red!10] (\hhh*0+\hhhh*4,-\vvv*2) rectangle (\hhh*4+\hhhh*4,-\vvv*1);
\draw[-,dotted] (\hhh*1+\hhhh*4,-\vvv*2) to (\hhh*1+\hhhh*4,\vvv*-1);
\draw[-,dotted] (\hhh*2+\hhhh*4,-\vvv*2) to (\hhh*2+\hhhh*4,\vvv*-1);
\draw[-,dotted] (\hhh*3+\hhhh*4,-\vvv*2) to (\hhh*3+\hhhh*4,\vvv*-1);
\node at (-\hhh*0.4+\hhhh*4,-\vvv*3.6) {\tiny $\tS'_{k_1}$};
\node at (\hhh*0.55+\hhhh*4,-\vvv*2.6) {\tiny $\tS'_{k_2}$};
\node at (\hhh*2+\hhhh*4,-\vvv*1.5) {\tiny $\tS'_{k_0}$};
\draw [decorate,
    decoration = {calligraphic brace}] (\hhh*4+\hhhh*4,\vvv*0.1) --  (\hhh*5.3+\hhhh*4,\vvv*0.1) node[above,midway] {\tiny $k_1-1$};
\node at (\hhh*4.65+\hhhh*4,-\vvv*0.6) {\tiny $\tS'_{k_{\hspace{-0.1ex} - \hspace{-0.2ex} 1}}$};
\end{tikzpicture}
\\
\begin{tikzpicture}
\def\hhh{5mm}
\def\vvv{5.5mm}
\def\www{0.20mm}
\def\hhhh{35mm}
\node[above] at (\hhh*2+\hhhh*0,-\vvv*5.1) {\small $\trd(\balalp)$};
\draw[line width=\www, fill=red!20] (\hhh*0,\vvv*0) rectangle (\hhh*4,-\vvv*1);
\draw[line width=\www, fill=yellow!20] (-\hhh*0,-\vvv*2) rectangle (\hhh*1,-\vvv*1);
\draw[line width=\www, fill=green!10] (\hhh*-1,-\vvv*4) rectangle (\hhh*-0,-\vvv*2);
\draw[-,dotted] (\hhh*1,-\vvv*1) to (\hhh*1,\vvv*0);
\draw[-,dotted] (\hhh*2,-\vvv*1) to (\hhh*2,\vvv*0);
\draw[-,dotted] (\hhh*3,-\vvv*1) to (\hhh*3,\vvv*0);
\draw[-,dotted] (\hhh*0,-\vvv*2) to (\hhh*1,-\vvv*2);
\node[] at (\hhh*2,-\vvv*0.5) {\tiny $\tS_{k_0}$};
\node[] at (\hhh*0.5,-\vvv*1.5) {\tiny $\tS_{k_2}$};
\node[] at (\hhh*-0.5,-\vvv*3) {\tiny $\tS_{k_1}$};
\draw[->,decorate,decoration={snake,amplitude=.4mm,segment length=2mm,post length=1mm}]
(\hhh*4,-\vvv*1.5) -- (-\hhh*1.5+\hhhh*1,-\vvv*1.5);
\draw[line width=\www, fill=green!10] (\hhh*-1+\hhhh,-\vvv*4) rectangle (\hhh*0+\hhhh,-\vvv*2);
\draw[-,dotted] (\hhh*-1+\hhhh,\vvv*-3) to (\hhh*0+\hhhh,\vvv*-3);
\node[] at (\hhh*-0.5+\hhhh,-\vvv*3) {\tiny $\tS'_{k_1}$};
\def\hhhh{30mm}
\draw[->,decorate,decoration={snake,amplitude=.4mm,segment length=2mm,post length=1mm}]
(\hhh*1.5+\hhhh*1,-\vvv*1.5) -- (-\hhh*3+\hhhh*2,-\vvv*1.5);
\draw[line width=\www, fill=red!10] (\hhh*-1+\hhhh*2,-\vvv*2) rectangle (\hhh*2+\hhhh*2,-\vvv*1);
\draw[line width=\www, fill=green!10] (\hhh*-2+\hhhh*2,-\vvv*4) rectangle (\hhh*-1+\hhhh*2,-\vvv*2);
\node[] at (\hhh*-1.45+\hhhh*2,-\vvv*3) {\tiny $\tS'_{k_1}$};
\node[] at (\hhh*0.6+\hhhh*2,-\vvv*1.55) {\tiny $\tS'_{k_0}$};
\draw [decorate,
    decoration = {calligraphic brace}] (\hhh*-1+\hhhh*2,-\vvv*0.85) --  (\hhh*2+\hhhh*2,-\vvv*0.85);
\node[] at (\hhh*0.5+\hhhh*2,-\vvv*0.4) {\tiny $\ell(\alpha)-k_2+2$};
\draw[-,dotted] (\hhh*1+\hhhh*2,\vvv*-2) to (\hhh*1+\hhhh*2,\vvv*-1);
\draw[-,dotted] (\hhh*0+\hhhh*2,\vvv*-2) to (\hhh*0+\hhhh*2,\vvv*-1);
\def\hhhh{30mm}
\draw[->,decorate,decoration={snake,amplitude=.4mm,segment length=2mm,post length=1mm}]
(\hhh*2.5+\hhhh*2,-\vvv*1.5) -- (-\hhh*2.5+\hhhh*3,-\vvv*1.5);
\draw[line width=\www, fill=red!20] (\hhh*1+\hhhh*3,\vvv*0) rectangle (\hhh*3+\hhhh*3,-\vvv*1);
\draw[line width=\www, fill=red!10] (\hhh*-1+\hhhh*3,-\vvv*2) rectangle (\hhh*2+\hhhh*3,-\vvv*1);
\draw[line width=\www, fill=green!10] (\hhh*-2+\hhhh*3,-\vvv*4) rectangle (\hhh*-1+\hhhh*3,-\vvv*2);
\node[] at (\hhh*-1.45+\hhhh*3,-\vvv*3) {\tiny $\tS'_{k_1}$};
\node[] at (\hhh*0.6+\hhhh*3,-\vvv*1.55) {\tiny $\tS'_{k_0}$};
\draw [decorate,
    decoration = {calligraphic brace}] (\hhh*1+\hhhh*3,\vvv*0.15) --  (\hhh*3+\hhhh*3,\vvv*0.15) node[above,midway] {\tiny $k_2-1$};
\node[] at (\hhh*2+\hhhh*3,-\vvv*0.5) {\tiny $\tS'_{k_{-1}}$};
\draw[-,dotted] (\hhh*0+\hhhh*3,\vvv*-2) to (\hhh*0+\hhhh*3,\vvv*-1);
\draw[-,dotted] (\hhh*1+\hhhh*3,\vvv*-2) to (\hhh*1+\hhhh*3,\vvv*-1);
\draw[-,dotted] (\hhh*2+\hhhh*3,\vvv*-1) to (\hhh*2+\hhhh*3,\vvv*0);
\node[above] at (\hhh*1+\hhhh*3,-\vvv*5.2) {\small $\trd(\balalpp{2})$};
\node at (\hhhh*4.6,-\vvv*3.2) {};
\end{tikzpicture}

\caption{The construction of $\trd(\balalpp{1})$ and $\trd(\balalpp{2})$ when $\alpha = (1,3,2,1)$}
\protect \label{Fig: rd balalppj}
\end{figure}

Now, let us construct $\partial^1: \bfP_{\balalp}\ra \bdI$.
Choose any tableau $T$ in $\SRT(\balalp)$.
Recall that $\bfw(T)$ is the word obtained by reading the entries of $T$ from left to right starting with the bottom row.
Let $\bfw(T) = w_1 w_2 \cdots w_n$.
For each $1 \leq j \leq m$, we consider the subword $\bfw_{T;j}$ of $\bfw(T)$ defined by
\begin{align}\label{Eq: w(T;j)}
\bfw_{T;j} := w_{u_1(j)} w_{u_2(j)} \cdots w_{u_{l_j}(j)},
\end{align}
where the subscripts $u_i(j)$'s are defined via the following recursion:
\begin{align*}
&u_1(j) = \sum_{1 \le r \le j} (\alpha_{k_r}-1),\\
&u_{i+1}(j) = \min \{u_i(j) < u \le n - \ell(\alpha) \mid w_u < w_{u_i(j)} \} \quad (i \ge 1), \text{ and }\\
&l_j := \max \{i \mid u_i(j) < \infty \}.
\end{align*}
In the second identity, whenever $\{u_i(j) < u \le n - \ell(\alpha) \mid w_u < w_{u_i(j)} \}=\emptyset$,
we set $u_{i+1}(j):=\infty$.
Henceforth we simply write $u_i$'s for $u_i(j)$'s
and thus $\bfw_{T;j}= w_{u_1} w_{u_2} \cdots w_{u_{l_j}}$.
Given an arbitrary word $w$, we use $\mathsf{end}(w)$ to denote the last letter of $w$.
With the notations above, we introduce the following two sets:
\begin{align*}
\begin{aligned}
\tA_{T;j} &: = \{y \in T(\tS_{k_0}) \mid y > {\sf end}(\bfw_{T;j})\}, \\
\setA{T}{j}  &: =  \left\{ A \subseteq {\tt A}_{T;j}   \mid |A| = \ell(\alpha)-k_j+1 \right\}.
\end{aligned}
\end{align*}
For $A \in \setA{T}{j}$, we define $\tau_{T;j;A}$ to be an SRT of shape $\balalpp{j}$
which is uniquely determined by the following conditions:
\begin{enumerate}[label = {\rm (\roman*)}]
\item
$\tau_{T;j;A}(\tS'_{k_{-1}}) = T(\tS_{k_0}) \setminus A$,
\item
$\tau_{T;j;A}(\tS'_{k_0}) =\{{\sf end}(\bfw_{T;j})\} \cup A$,
\item
$\tau_{T;j;A}(\tS'_{k_r}) = T(\tS_{k_r})$ for $1 \leq r < j$,
\item
$\tau_{T;j;A}(\tS'_{k_j}) = T(\tS_{k_j}) \setminus \{w_{u_1} \}$, and
\item
for $j < r \leq m$, $\tau_{T;j;A}(\tS'_{k_r})$ is obtained from $T(\tS_{k_r})$ by substituting $w_{u_{i}}$ with $w_{u_{i-1}}$
for $w_{u_{i}}$'s ($1 < i \le l_j$) contained in $T(\tS_{k_r})$.
\end{enumerate}
We next explain the notion of the {\it signature} $\sgn(A)$ of $A$.
Enumerate the elements in $\tA_{T;j}$ in the increasing order
\[
a_1 < a_2 < \cdots < a_{|\tA_{T;j}|}.
\]
Then, let $A^1_{T;j}$ be the set of the consecutive $(\ell(\alpha)-k_j+1)$ elements starting from the rightmost and moving to the left, precisely,
$$
A^1_{T;j} = \{a_{|\tA_{T;j}|-\ell(\alpha)+k_j},a_{|\tA_{T;j}|-\ell(\alpha)+k_j+1}, \ldots,a_{|\tA_{T;j}|} \}.
$$
There is a natural right $\SG_{|\tA_{T;j}|}$-action on $\tA_{T;j}$ given by
\begin{align}\label{eq: SG action on ATj}
a_i \cdot \omega = a_{\omega^{-1}(i)} \text{ for } 1 \le i \le |\tA_{T;j}| \text{ and }\omega \in \SG_{|\tA_{T;j}|}.
\end{align}
We define $\sgn(A) := (-1)^{\ell(\omega^1)}$,
where $\omega^1$ is any minimal length permutation in  $\{\omega \in \SG_{|\tA_{T;j}|} \mid A = A^1_{T;j} \cdot \omega\}$.

For each $1 \leq j \leq m$, set
$$
\bstau{T}{j}:=
\displaystyle \sum_{A \in \setA{T}{j}} \sgn(A) \tau_{T;j;A},
$$
where the summation in the right hand side is zero in case where $\setA{T}{j}=\emptyset$.
Finally, we define a $\C$-linear map
$$
\partial^1: \bfP_{\balalp}\ra \bdI,
\quad
T \mapsto \sum_{1 \leq j \leq m} \bstau{T}{j}
$$
with $\bdI$ in \eqref{Eq: I1}.

\begin{theorem}\label{Thm: Main Section4}
{\rm (This will be proven in Subsection~\ref{Subsec: Proof of Theorem4}.)}
Let $\alpha$ be a composition of $n$.
\begin{enumerate}[label = {\rm (\alph*)}]
\item $\partial^1: \bfP_{\balalp} \ra \bdI$ is an $H_n(0)$-module homomorphism.
\item The sequence
\begin{equation*}
\begin{tikzcd}
\calV_\alpha \arrow[r, "\injhull"] & \bfP_{\balalp} \arrow[r, "\partial^1"] & \displaystyle \bdI
\end{tikzcd}
\end{equation*}
is exact.

\item
The $H_n(0)$-module homomorphism
\[
\opone: \Omega^{-1}(\calV_\alpha) \ra \bdI, \quad
T + \injhull(\calV_\alpha) \mapsto \partial^1(T)
\quad  (T \in \Theta(\calV_\alpha))
\]
induced from $\partial^1$ is an injective hull of $\Omega^{-1}(\calV_\alpha)$.

\item
Let
$\calL(\alpha) := \bigcup_{1 \leq j \leq m} [\balalpp{j}]$, which is viewed as a multiset.
Then we have
\[\Ext_{H_n(0)}^1(\bfF_{\beta},\calV_\alpha)
\cong \begin{cases}
\mathbb C^{[\calL(\alpha):\beta^\rmr]} & \text{ if } \beta^\rmr \in \calL(\alpha)\\
0 & \text{otherwise,}
\end{cases}
\]
where $[\calL(\alpha):\beta^\rmr]$ denotes the multiplicity of $\beta^\rmr$ in $\calL(\alpha)$.
\end{enumerate}
\end{theorem}

\begin{example}
Let $\alpha = (2,1,2,3) \models 8$.
Then $\calK(\alpha) = \{0,1,3,4\}$ and $\balalp = (1) \oplus (1) \oplus (3,1^3)$. By definition we get
\begin{eqnarray*}
\balalpp{1} & = & (1) \oplus (3,1^4), \\
\balalpp{2} & = & (1) \oplus (3,1^2) \oplus (1^2), \\
\balalpp{3} & = & (1) \oplus (1) \oplus (2^2,1^2).
\end{eqnarray*}

(a) Let $T =
\ytableausetup{aligntableaux=center, boxsize=0.9em}
\begin{ytableau}
\none & \none & 1 & 3 & 7 & 8\\
\none & \none & 4 \\
\none & \none & 5 \\
\none & 2 \\
6
\end{ytableau}$.
Then one sees that
\begin{displaymath}
\begin{array}{llll}
\bfw_{T;1} = 6 \ 2 \quad & \ebfw{T}{1} = 2  \quad & \tA_{T;1} = \{3, 7,8\} \quad & \setA{T}{1} = \emptyset,\\
\bfw_{T;2} = 2   & \ebfw{T}{2} = 2  & \tA_{T;2} = \{3,7,8 \}  & \setA{T}{2} = \{\{3,7\},\{3,8\},\{7,8\}\}, \\
\bfw_{T;3} = 4  & \ebfw{T}{3} = 4   & \tA_{T;3} = \{7,8 \}  & \setA{T}{3} = \{\{7\},\{8\} \}.
\end{array}
\end{displaymath}
Since
\begin{displaymath}
\begin{split}
&\tau_{T;2;\{3,7\}} =
\begin{ytableau}
\none & \none & \none & \none & 1 & 8\\
\none & 2 & 3 & 7 \\
\none & 4 \\
\none & 5\\
6
\end{ytableau}
\qquad
\tau_{T;2;\{3,8\}} =
\begin{ytableau}
\none & \none & \none & \none & 1 & 7\\
\none & 2 & 3 & 8 \\
\none & 4 \\
\none & 5 \\
6
\end{ytableau}
\qquad
\tau_{T;2;\{7,8\}} =
\begin{ytableau}
\none & \none & \none & \none & 1 & 3\\
\none & 2 & 7 & 8 \\
\none & 4 \\
\none & 5 \\
6
\end{ytableau}
\\
&\tau_{T;3;\{7\}} =
\begin{ytableau}
\none & \none & \none &  1 & 3 & 8\\
\none & \none & 4 & 7\\
\none & \none & 5  \\
\none & 2\\
6
\end{ytableau}
\qquad
\tau_{T;3;\{8\}} =
\begin{ytableau}
\none & \none & \none &  1 & 3 & 7\\
\none & \none & 4 & 8\\
\none & \none & 5  \\
\none & 2\\
6
\end{ytableau},
\end{split}
\end{displaymath}
it follows that
\begin{align*}
\bstau{T}{1} = 0 \qquad
\bstau{T}{2} = \tau_{T;2;\{3,7\}} - \tau_{T;2;\{3,8\}} + \tau_{T;2;\{7,8\}} \qquad
\bstau{T}{3} = - \tau_{T;3;\{7\}} + \tau_{T;3;\{8\}}.
\end{align*}
Therefore,
$$
\partial^1(T) = (\tau_{T;2;\{3,7\}} - \tau_{T;2;\{3,8\}} + \tau_{T;2;\{7,8\}}) + (- \tau_{T;3;\{7\}} + \tau_{T;3;\{8\}}).
$$

(b)
Note that
\begin{align*}
[\balalpp{1}] & = \left\{(1,3,1^4), (4,1^4)\right\}, \\
[\balalpp{2}] & = \left\{(1,3,1^4),(1,3,1,2,1),(4,1^4),(4,1,2,1)\right\}, \\
[\balalpp{3}] & = \left\{(1^2,2^2,1^2),(1,3,2,1^2),(2^3,1^2),(4,2,1^2) \right\}.
\end{align*}
Theorem~\ref{Thm: Main Section4}(d) implies that
$$
\dim \Ext^1_{H_n(0)}(\bfF_{\beta},\calV_\alpha) =
\begin{cases}
1 & \text{ if } \beta^\rmr \in   \calL(\alpha) \setminus \{(1, 3, 1^4),(4, 1^4)\}, \\
2 & \text{ if } \beta^\rmr \in \{(1,3,1^4), (4, 1^4)\}, \\
0 & \text{ otherwise.} 
\end{cases}
$$
\end{example}

\section{$\Ext^i_{H_n(0)}(\calV_\alpha, \calV_\beta)$ with $i=0,1$}
\label{sec: Ext1 VV}
In the previous sections, we computed $\Ext^1_{H_n(0)}(\calV_\alpha, \bfF_\beta)$ and
$\Ext^1_{H_n(0)}(\bfF_\beta, \calV_\alpha)$.
In this section, we focus on $\Ext^1_{H_n(0)}(\calV_\alpha, \calV_\beta)$
and $\Ext^0_{H_n(0)}(\calV_\alpha, \calV_\beta)\, (={\rm Hom}_{H_n(0)}(\calV_{\alpha}, \calV_{\beta}))$.

Let $M, N$ be finite dimensional $H_n(0)$-modules.
Given a short exact sequence
\begin{equation*}
\begin{tikzcd}
0 \arrow[r] & \Omega(M) \arrow[r, "\iota"] & P_0 \arrow[r, "\pi"] & M \arrow[r] & 0
\end{tikzcd}
\end{equation*}
with $(P_0,\pi)$ a projective cover of $M$,
it is well known that
\begin{equation*}
\Ext_{H_n(0)}^1(M,N)
\cong \frac{{\rm Hom}_{H_n(0)}(\Omega(M),N)}
{{\rm Im} \,\iota^\ast},
\end{equation*}
where $\iota^\ast:{\rm Hom}_{H_n(0)}(P_0, N) \to {\rm Hom}_{H_n(0)}(\Omega(M),N)$
is given by composition with $\iota$.
The kernel of $\iota^\ast$ equals
\[ \{f\in {\rm Hom}_{H_n(0)}(P_0,N) \mid f|_{\Omega(M)}=0\},\]
and therefore
\begin{align}\label{Computing Hom(M,N)}
\ker(\iota^\ast)
\cong {\rm Hom}_{H_n(0)}(P_0/\Omega(M),N)\cong {\rm Hom}_{H_n(0)}(M,N).
\end{align}
This says that $\Ext_{H_n(0)}^1(M,N)=0$ if and only if, as $\mathbb C$-vector spaces,
\begin{align}\label{Hom in the rigid case}
{\rm Hom}_{H_n(0)}(P_0, N) \cong  {\rm Hom}_{H_n(0)}(\Omega(M), N) \oplus {\rm Hom}_{H_n(0)}(M,N).
\end{align}
\begin{definition}
Given a finite dimensional $H_n(0)$-module $M$, we say that $M$ is \emph{rigid} if $\Ext_{H_n(0)}^1(M,M)=0$ and
\emph{essentially rigid} if
${\rm Hom}_{H_n(0)}(\Omega(M),M)=0$.
\end{definition}

Whenever $M$ is essentially rigid, one has that
${\rm Hom}_{H_n(0)}(P_0, M) \cong {\rm End}_{H_n(0)}(M).$
Typical examples of essentially rigid $H_n(0)$-modules are simple modules and projective modules.
Also, the syzygy and cosyzygy modules of a rigid module are also rigid
since
$\Ext_{H_n(0)}^1(M,N)=\Ext_{H_n(0)}^1(\Omega(M),\Omega(N))$
and $M \cong \Omega\Omega^{-1}(M) \oplus ({\rm projective})$
(for example, see \cite{91Benson}).

Let us use $\le_l$ to represent the lexicographic order on compositions of $n$.
Using the results in the preceding sections, we derive some interesting results on $\Ext_{H_n(0)}^1(\calV_\alpha,\calV_{\beta})$.
To do this, we need the following lemmas.

\begin{lemma}{\rm (\cite[Lemma 1.7.6]{91Benson})} \label{dimension of Hom}
Let $M$ be a finite dimensional $H_n(0)$-module. Then
$\dim \Hom_{H_n(0)}( \bfP_\alpha, M)$ is the multiplicity of $\bfF_\alpha$ as a composition factors of $M$.
\end{lemma}

\begin{lemma} {\rm (\cite[Proposition 3.37]{14BBSSZ})} \label{F-expansion od dual immaculate}
The dual immaculate functions $\frakS^*_\alpha$ are fundamental positive.
Specifically, they expand as $\frakS^*_\alpha=\sum_{\beta \le_l \alpha}L_{\alpha, \beta}F_\beta$,
where $L_{\alpha, \beta}$ denotes the number of standard immaculate tableaux $\calT$ of shape $\alpha$ and descent composition $\beta$, i.e., $\comp(\Des(\calT)) = \beta$.
\end{lemma}

We now state the main result of this section.

\begin{theorem} \label{thm: Ext for V V}
Let $\alpha$ be a composition of $n$.

\begin{enumerate}[label = {\rm (\alph*)}]
\item
For all $\beta \le_l \alpha$,
$\Ext_{H_n(0)}^1(\calV_\alpha,\calV_{\beta})=0.$
In particular, $\calV_\alpha$ is essentially rigid.

\item
For all $\beta \le_l \alpha$, we have
\[
{\rm Hom}_{H_n(0)}(\calV_\alpha, \calV_{\beta})
\cong\begin{cases}
\mathbb C & \text{if $\beta =\alpha$,}\\
0 & \text{otherwise.}
\end{cases}
\]

\item
Let $M$ be any nonzero quotient of $\calV_\alpha$. Then \({\rm End}_{H_n(0)}(M) \cong \mathbb C\).
\end{enumerate}
\end{theorem}

\begin{proof}
(a) Due to Theorem~\ref{main thm for V}, there is a projective resolution of $\calV_\alpha$ of the form 
\begin{equation*}
\begin{tikzcd}
\cdots \arrow[r] & \displaystyle \bigoplus_{i \in \calI(\alpha)} \bfP_{\bal^{(i)}} \arrow[r] & \bfP_\alpha \arrow[r] & \calV_{\alpha} \arrow[r] & 0.
\end{tikzcd}
\end{equation*}
Hence, for the assertion, it suffices to show that
\[{\rm Hom}_{H_n(0)}\left(\bigoplus_{i \in \calI(\alpha)} \bfP_{\bal^{(i)}},\calV_{\beta} \right)=0.\]
Observe that
\begin{align*}
\dim{\rm Hom}_{H_n(0)}\left(\bigoplus_{i \in \calI(\alpha)} \bfP_{\bal^{(i)}},\calV_{\beta} \right)
&=\sum_{ \gamma \in \calJ(\alpha)}\dim{\rm Hom}_{H_n(0)}\left(\bfP_{\gamma},\calV_{\beta} \right)\\
&=\sum_{ \gamma \in \calJ(\alpha)} [\calV_{\beta}:\bfF_{\gamma}] \quad \text{ (by Lemma~\ref{dimension of Hom})}.
\end{align*}
Here $[\calV_{\beta}:\bfF_{\gamma}]$ denotes the multiplicity of $\bfF_{\gamma}$ as a composition factor of $\calV_{\beta}$,
thus equals the coefficient of $F_\gamma$ in the expansion of $\frakS^*_\beta$ into fundamental quasisymmetric functions.
From Lemma~\ref{F-expansion od dual immaculate} it follows that this coefficient vanishes unless $\beta \ge_l \gamma$.
Since $\alpha <_l \gamma$ for all $\gamma \in \calJ(\alpha)$, the assumption $\beta \le_l \alpha$ yields the desired result.

(b) Combining \eqref{Hom in the rigid case}
with (a) yields that
\begin{align*}
{\rm Hom}_{H_n(0)}(\bfP_{\alpha}, \calV_{\beta})\cong  {\rm Hom}_{H_n(0)}(\Omega(\calV_{\alpha}), \calV_{\beta}) \oplus {\rm Hom}_{H_n(0)}(\calV_{\alpha},\calV_{\beta}).
\end{align*}
But, by Lemma~\ref{dimension of Hom} and Lemma~\ref{F-expansion od dual immaculate}, we see that
\begin{align*}
\dim{\rm Hom}_{H_n(0)}(\bfP_{\alpha}, \calV_{\beta})=L_{\beta, \alpha}=
\begin{cases}
1 & \text{ if } \beta =\alpha\\
0 & \text{ otherwise.}
\end{cases}
\end{align*}
This justifies the assertion since $\dim{\rm End}_{H_n(0)}( \calV_{\alpha})\ge 1$.

(c) Let $f:\bfP_\alpha \to M$ be a surjective $H_n(0)$-module homomorphism.
Then
\[{\rm End}_{H_n(0)}(M)\cong
{\rm Hom}_{H_n(0)}(\bfP_{\alpha}/\ker(f), M),\]
and therefore
\[1\le \dim{\rm End}_{H_n(0)}(M)\le
\dim{\rm Hom}_{H_n(0)}(\bfP_{\alpha}, M)=[M: \bfF_\alpha].\]
Now the assertion follows from the inequality $[M: \bfF_\alpha]\le [\calV_\alpha: \bfF_\alpha]=L_{\alpha, \alpha}=1$.
\end{proof}

\begin{remark}
To the best of the authors' knowledge, the classification or distribution of indecomposable rigid modules is
completely unknown. For the reader's understanding, we provide some related examples.

(a) Let $M := \bfP_{(1,2,2)} / H_5(0) \cdot \left\{
\scalebox{0.7}{
\begin{ytableau}
\none & \none & 4 \\
\none & 1 & 5 \\
2 & 3
\end{ytableau}
}
\right\}$.
A simple computation shows that
$M$ is a rigid indecomposable module.
But, since \(\dim \Hom_{H_5(0)}(\Omega(M), M) = 1\),
it is not essentially rigid.

(b) Let $V := \bfP_{(1,2,2)} / H_5(0) \cdot  \left\{
\scalebox{0.7}{$
\begin{ytableau}
\none & \none & 3 \\
\none & 1 & 5 \\
2 & 4
\end{ytableau},~
\begin{ytableau}
\none & \none & 1 \\
\none & 3 & 4 \\
2 & 5
\end{ytableau}$
}
\right\}$.
By adding two $V$'s appropriately, one can produce a non-split sequence
\begin{equation*}
\begin{tikzcd}
0 \arrow[r] & V \arrow[r] & M \arrow[r] & V \arrow[r] & 0.
\end{tikzcd}
\end{equation*}
Hence $V$ is a non-rigid indecomposable module.
\end{remark}

Theorem~\ref{thm: Ext for V V} (b) is no longer valid unless $\beta \le_l \alpha$.
In view of $\calV_\alpha \cong \bfP_{\alpha} / \Omega(\calV_\alpha)$, one can view
${\rm Hom}_{H_n(0)}(\calV_{\alpha},\calV_{\beta})$ as the $\C$-vector space consisting of
$H_n(0)$-module homomorphisms from $\bfP_{\alpha}$ to $\calV_{\beta}$
which vanish on $\Omega(\calV_{\alpha})$.
Therefore, in order to understand ${\rm Hom}_{H_n(0)}(\calV_{\alpha},\calV_{\beta})$,
it is indispensable to understand
${\rm Hom}_{H_n(0)}(\bfP_{\alpha}, \calV_{\beta})$ first.
To do this, let us fix a linear extension $\preccurlyeq_{L}^{\rm r}$ of the partial order $\preccurlyeq^{\rm r}$ on $\SIT(\beta)$ given by
$$
\tau' \preccurlyeq^{\rm r}\tau
\quad \text{if and only if} \quad
\text{$\tau' = \pi_\gamma \cdot \tau$ for some $\gamma\in\SG_n$.}
$$
Given $f\in {\rm Hom}_{H_n(0)}(\bfP_{\alpha},\calV_{\beta})$, let
$f(T_\alpha)=\sum_{\calT \in \SIT(\beta)}c_{f,\calT}\calT$.
We define ${\mathsf {Lead}}(f)$ to be the largest tableau in $\{\calT \in \SIT(\beta): c_{f,\calT}\ne 0\}$
with respect to $\preccurlyeq_{L}^{\rm r}$.
When $f=0$, ${\mathsf {Lead}}(f)$ is set to be $\emptyset$.
\begin{theorem}\label{Prop: Hom V V}
Let $\alpha, \beta$ be compositions of $n$ and let $\mathfrak B$ be the set of standard immaculate tableaux $U$ of shape $\beta$ with $\Des(U)={\rm set}(\alpha)$.
\begin{enumerate}[label = {\rm (\alph*)}]
\item
For each standard immaculate tableau $U$ of shape $\beta$ with $\Des(U)={\rm set}(\alpha)$,
there exists a unique homomorphism $f_U\in {\rm Hom}_{H_n(0)}(\bfP_{\alpha},\calV_{\beta})$
such that ${\mathsf {Lead}}(f)=U$, $c_{f,U}=1$, and $c_{f,U'}=0$ for all $U'\in \mathfrak B\setminus \{U\}$.
\item
The dimension of ${\rm Hom}_{H_n(0)}(\calV_{\alpha},\calV_{\beta})$ is the same as the dimension of
\[\{(c_U)_{U\in \mathfrak B} \in \mathbb C^{|\mathfrak B|}: \sum_{U}c_U \,\pi_{[m_{i-1} + 1, m_i]} \cdot f_U(T_\alpha)=0 \text{ for all $i \in \calI(\alpha)$} \}.\]
\end{enumerate}
\end{theorem}

\begin{proof}
(a)
Observe that every homomorphism in ${\rm Hom}_{H_n(0)}(\bfP_{\alpha},\calV_{\beta})$ is completely determined by the value at
the source tableau $T_\alpha$ of $\bfP_\alpha$.
We claim that $\Des({\mathsf {Lead}}(f)) ={\rm set}(\alpha)$ for all nonzero $f\in {\rm Hom}_{H_n(0)}(\bfP_{\alpha},\calV_{\beta})$.
To begin with, from the equalities $f(\pi_i \cdot T_\alpha)=f(T_\alpha)$ for all $i\notin \Des(T_\alpha)={\rm set}(\alpha)$, we see that
$f$ satisfies the condition that
$\Des({\mathsf {Lead}}(f)) \subseteq {\rm set}(\alpha)$.
Recall that we set $m_i:=\sum_{1\le k \le i}\alpha_i$ for all $1\le i \le \ell(\alpha)$ in Section~\ref{sec: extensions of V by F}.
Suppose that there is an index $j$ such that
\[m_j \in {\rm set}(\alpha) \setminus \Des({\mathsf {Lead}}(f)).\]
Then
\[m_{j-1}+1, m_{j-1}+2, \ldots, m_{j+1}-1 \in {\rm set}(\alpha) \setminus \Des({\mathsf {Lead}}(f)).\]
But, this is absurd since
\[
\pi_{[m_{j-1}+1, m_{j+1}-\alpha_j]^{\rm r}} \cdots \pi_{[m_j-1,m_{j+1}-2]^{\rm r}}\pi_{[m_j, m_{j+1}-1]^{\rm r}} \cdot T_\alpha = 0,
\]
whereas
\[
\pi_{[m_{j-1}+1, m_{j+1}-\alpha_j]^{\rm r}} \cdots \pi_{[m_j-1,m_{j+1}-2]^{\rm r}}\pi_{[m_j, m_{j+1}-1]^{\rm r}} \cdot {\mathsf {Lead}}(f) = {\mathsf {Lead}}(f).
\]
So the claim is verified.

For each $U \in \mathfrak B$, consider the $\C$-vector space
$$H(U):=\{f\in {\rm Hom}_{H_n(0)}(\bfP_{\alpha},\calV_{\beta}):
{\mathsf {Lead}}(f) \preccurlyeq_{L}^{\rm r} U\}.$$
Write $\mathfrak B$ as $\{U_1\preccurlyeq_{L}^{\rm r}U_{2} \preccurlyeq_{L}^{\rm r} \cdots \preccurlyeq_{L}^{\rm r} U_{l-1} \preccurlyeq_{L}^{\rm r} U_{l}\}$, where $l=|\mathfrak B|$.
For any $f,g \in H(U_i)$, it holds that
\[
c_{g,{\mathsf {Lead}}(g)}f- c_{f,{\mathsf {Lead}}(f)}g \in H(U_{i-1})
\]
with $H(U_0):=0$.
This implies that $\dim H(U_i)/H(U_{i-1}) \le 1$ for all $1\le i \le l$.

Combining these inequalities with the equality $\dim{\rm Hom}_{H_n(0)}\left(\bfP_{\alpha},\calV_{\beta} \right)=|\mathfrak B|$,
we deduce that, for each $U\in \mathfrak B$, there exists a unique $f_U\in {\rm Hom}_{H_n(0)}(\bfP_{\alpha},\calV_{\beta})$
with the desired property.

(b) By (a), one sees that $\{f_U: U\in \mathfrak B\}$ forms a basis for ${\rm Hom}_{H_n(0)}\left(\bfP_{\alpha},\calV_{\beta} \right)$.
Since ${\rm Hom}_{H_n(0)}(\calV_{\alpha},\calV_{\beta})$ is isomorphic to
the $\C$-vector space consisting of
$H_n(0)$-module homomorphisms from $\bfP_{\alpha}$ to $\calV_{\beta}$
which vanish on $\Omega(\calV_{\alpha})$, our assertion follows from
Lemma~\ref{lem: char of tab in T^i}, which says that $\{\pi_{[m_{i-1} + 1, m_i]} \cdot T_\alpha \,: \, i \in \calI(\alpha)\}$
is a generating set of $\Omega(\calV_{\alpha})$.
\end{proof}

\begin{example}
(a) Let $\alpha=(1,1,2,1)$ and $\beta=(1,2,2)$.
Then $\mathfrak B=\{
U:=\begin{array}{l}
\begin{ytableau}
1 & \none \\
2 & 4 \\
3 & 5
\end{ytableau}
\end{array}\}
$
and
\[f_U(T_\alpha)=\begin{array}{l}
\begin{ytableau}
1 & \none \\
2 & 4 \\
3 & 5
\end{ytableau}
\end{array}
-
\begin{array}{l}
\begin{ytableau}
1 & \none \\
2 & 5 \\
3 & 4
\end{ytableau} \,\,.
\end{array}\]
Note that $\calI(\alpha)=\{2\}$ and $m_1=1, m_2=2$.
Since $\pi_{2} \cdot f_U(T_\alpha)=0$,
it follows that
${\rm Hom}_{H_n(0)}(\calV_{\alpha},\calV_{\beta})$ is $1$-dimensional.

(b) Let $\alpha=(1,1,3,2)$ and $\beta=(2,3,2)$.
Then
\[\mathfrak B=\left\{
U_1:=\begin{array}{l}
\begin{ytableau}
1 & 5 & \none\\
2 & 4&7 \\
3 & 6&\none
\end{ytableau}
\end{array},
U_2:=\begin{array}{l}
\begin{ytableau}
1 & 7 & \none\\
2 & 4& 5 \\
3 & 6 &\none
\end{ytableau}
\end{array},
U_3:=\begin{array}{l}
\begin{ytableau}
1 & 5 & \none\\
2 & 6& 7 \\
3 & 4 &\none
\end{ytableau}
\end{array}\right\}
\]
and
$f_{U_i}(T_\alpha)=U_i$ for $i=1,2,3$.
Note that $\calI(\alpha)=\{2,3\}$ and $m_1=1, m_2=2, m_3 = 5$.
Since $\pi_{2} \cdot f_{U_i}(T_\alpha) =0$ for all $1\le i \le 3$ and
\[
\pi_{[3,5]} \cdot \left(c_{1}f_{U_1}(T_\alpha) + c_{2}f_{U_2}(T_\alpha)+c_{3}f_{U_3}(T_\alpha)\right)=
(c_1 + c_3) 
\begin{array}{l}
\begin{ytableau}
1 & 6 & \none\\
2 & 5& 7 \\
3 & 4 &\none
\end{ytableau}
\end{array}
+
c_2\begin{array}{l}
\begin{ytableau}
1 & 7 & \none\\
2 & 5& 6 \\
3 & 4 &\none
\end{ytableau}
\end{array},
\]
it follows that
${\rm Hom}_{H_n(0)}(\calV_{\alpha},\calV_{\beta})$ is $1$-dimensional.
\end{example}
We end up with an interesting consequence of Theorem~\ref{Thm: Main Section4},
where we successfully compute $\Ext_{H_n(0)}^1(\bfF_\beta,\calV_{\alpha})$
by constructing an injective hull of $\Omega^{-1}(\calV_\alpha)$.
To compute it in a different way,
let us consider a short exact sequence
\begin{equation*}
\begin{tikzcd}
0 \arrow[r] & {\rad}(\bfP_\beta) \arrow[r, "\iota"] & \bfP_\beta \arrow[r, "\mathrm{pr}"] & \bfF_\beta \arrow[r] & 0.
\end{tikzcd}
\end{equation*}
Here $\iota$ is the natural injection.
Then we have
\begin{equation}\label{compute Ext using projectives}
\Ext_{H_n(0)}^1(\bfF_\beta,\calV_{\alpha})
\cong \frac{{\rm Hom}_{H_n(0)}\left({\rad}(\bfP_\beta) ,\calV_{\alpha} \right)}
{{\rm Im}\, \iota^\ast},
\end{equation}
where $\iota^\ast:  {\rm Hom}_{H_n(0)}( \bfP_\beta ,\calV_{\alpha} )\ra {\rm Hom}_{H_n(0)}({\rad}(\bfP_\beta),\calV_{\alpha} ) $ is given by
composition by with $\iota$.
By \eqref{Computing Hom(M,N)}, one has that
\begin{align*}
\dim {\rm Im}\, \iota^\ast
&=\dim {\rm Hom}_{H_n(0)}\left(\bfP_\beta,\calV_{\alpha} \right)-\dim {\rm Hom}_{H_n(0)}\left(\bfF_\beta,\calV_{\alpha} \right)\\
&=[\calV_{\alpha}:\bfF_\beta]-[{\rm soc}(\calV_\alpha):\bfF_\beta]\\
&=L_{\alpha, \beta}-[[\balalp]:\beta^\rmr] \quad \text{(by Lemma \ref{F-expansion od dual immaculate} and Theorem~\ref{Thm:injective hull V})},
\end{align*}
where $[[\balalp]:\beta^\rmr]$ is the multiplicity of $\beta^\rmr \in [\balalp]$.
Comparing Theorem~\ref{Thm: Main Section4} with \eqref{compute Ext using projectives}
yields the following result.

\begin{corollary}
Let $\alpha, \beta$ be compositions of $n$. Then we have
\begin{align*}
\dim {\rm Hom}_{H_n(0)}\left({\rad}(\bfP_\beta) ,\calV_{\alpha} \right)
=L_{\alpha, \beta}-[[\balalp]:\beta^\rmr]+  [\calL(\alpha):\beta^\rmr].
\end{align*}
\end{corollary}

\section{Proof of Theorems}
\label{Sec: Proof of Theorems}

\subsection{Proof of Theorem \protect\ref{main thm for V}}\label{subsec: proof of proj. presentation}

We first prove that $\Omega(\calV_\alpha)$ is generated by $\{T^{(i)}_{\alpha} \mid i \in \calI(\alpha)\}$.
By the definition of $\projco$, one can easily derive that
$$
\Omega(\calV_\alpha) = \C \{T \in \SRT(\alpha) \mid T_p^1 > T_{p+1}^1 \text{ for some } 1\leq p < \ell(\alpha)\}.
$$

Given $\sigma \in \SG_n$, let
\[
\Des_L(\sigma):= \{i \in [n-1] \mid \ell(s_i \sigma) < \ell(\sigma)\}
\ \  \text{and} \ \
\Des_R(\sigma):= \{i \in [n-1] \mid \ell(\sigma s_i) < \ell(\sigma)\}.
\]
The \emph{left weak Bruhat order} $\preceq_L$ on $\SG_n$ is the partial order on $\SG_n$ whose covering relation $\preceq_L^c$ is defined as follows:
$\sigma \preceq_L^c s_i \sigma$ if and only if $i \notin \Des_L(\sigma)$.
It should be remarked that a word of length $n$ can be confused with a permutation in $\SG_n$ if each of $1,2,\ldots, n$ appears in it exactly once.

The following lemma plays a key role in proving Lemma~\ref{lem: char of tab in T^i}.

\begin{lemma}{\rm (\cite[Proposition 3.1.2 (vi)]{05BB})}\label{lem: left order and descent}
Suppose that $i \in \Des_R(\sigma) \cap \Des_R(\rho)$.
Then, $\sigma \preceq_L \rho$ if and only if $\sigma s_i \preceq_L \rho s_i$.
\end{lemma}

\begin{lemma}\label{lem: char of tab in T^i}
For each $i \in \calI(\alpha)$, $H_n(0) \cdot T^{(i)}_{\alpha} = \C\{T \in \SRT(\alpha) \mid T_i^1 > T_{i+1}^1 \}$.
Thus, $\Omega(\calV_\alpha) = \sum_{i\in \calI(\alpha)} H_n(0) \cdot T^{(i)}_{\alpha}$.
\end{lemma}

\begin{proof}
For simplicity, let $\SRT(\alpha)^{(i)}$ be the set of $\SRT$x of shape $\alpha$ such that
the topmost entry in the $i$th column is greater than
that in the $(i+1)$st column.

We first show that $H_n(0) \cdot T^{(i)}_{\alpha}$ is included in the $\mathbb C$-span of $\SRT(\alpha)^{(i)}$, equivalently
$\pi_\sigma \cdot T^{(i)}_{\alpha} \in \SRT(\alpha)^{(i)} \cup \{0\}$ for all $\sigma \in \SG_n$.
Suppose that there exists $\sigma \in \SG_n$ such that $\pi_\sigma \cdot T^{(i)}_{\alpha} \neq 0$ and $\pi_\sigma \cdot T^{(i)}_{\alpha} \notin \SRT(\alpha)^{(i)}$.
Let $\sigma_0$ be such a permutation with minimal length and $j$ a left descent of $\sigma_0$.
By the minimality of $\sigma_0$, we have $\pi_{s_j \sigma_0} \cdot T^{(i)}_{\alpha} \in \SRT(\alpha)^{(i)}$,
and therefore
\[(\pi_{s_j \sigma_0} \cdot T^{(i)}_{\alpha})^1_{i} > (\pi_{s_j \sigma_0} \cdot T^{(i)}_{\alpha})^1_{i+1}.\]
By the definition of the $\pi_j$-action on $\SRT(\alpha)$, we have
\[( \pi_j \cdot (\pi_{s_j \sigma_0} \cdot T^{(i)}_{\alpha}) )^1_i > ( \pi_j \cdot (\pi_{s_j \sigma_0} \cdot T^{(i)}_{\alpha}) )^1_{i+1}.\]
However, since $\pi_j \cdot (\pi_{s_j \sigma_0} \cdot T^{(i)}_{\alpha}) = \pi_{\sigma_0} \cdot T^{(i)}_{\alpha}$,
this contradicts the assumption that $\pi_{\sigma_0} \cdot T^{(i)}_{\alpha} \notin \SRT(\alpha)^{(i)}$.

We next show the opposite inclusion $\SRT(\alpha)^{(i)} \subseteq H_n(0) \cdot T^{(i)}_{\alpha}$.
Our strategy is to use \cite[Theorem 3.3]{16Huang}, which implicitly says that for $T_1, T_2 \in \SRT(\alpha)$,
$T_2 \in H_n(0) \cdot T_1$ if and only if $\bfw(T_1) \preceq_L \bfw(T_2)$.
Here $\bfw(T_i) \,(i=1,2)$ denotes the word obtained from $T_i$ by reading the entries from left to right starting with the bottom row.
For each $T \in \SRT(\alpha)^{(i)}$, we define $\tau_T$ to be the filling of $\trd(\bal^{(i)})$ whose entries in each column are increasing from top to bottom and whose columns are given as follows: for $1 \le p \le \ell(\alpha)$,
\begin{align}\label{eq: tau_T}
(\tau_T)_p^\bullet =
\begin{cases}
T_i^\bullet \cup \{T_{i+1}^1\} & \text{if $p = i$,}\\
T_{i+1}^\bullet \setminus \{ T_{i+1}^1 \}   & \text{if $p = i+1$,}\\
T_p^\bullet & \text{otherwise.}\\
\end{cases}
\end{align}
The inequality $(\tau_T)_i^1 < (\tau_T)_{i+1}^{-1}$ shows that $\tau_T \in \SRT(\bal^{(i)})$.
Combining
\begin{equation*}\label{eq: col and w}
\begin{aligned}
\bfw(\tau_T) = \bfw(T) s_{m_{i+1} - 1} s_{m_{i+1} - 2} \cdots s_{m_i}
\end{aligned}
\end{equation*}
with $\tau_{T^{(i)}_{\alpha}} = T_{\bal^{(i)}}$ (=the source tableau of $\bfP_{\bal^{(i)}}$)
yields that $\bfw(\tau_{T^{(i)}_{\alpha}}) \preceq_L \bfw(\tau_T)$ for $T \in \SRT(\alpha)^{(i)}$.
Moreover, for each $m_i \le j < m_{i+1}$, it holds that
\begin{equation}\label{eq: sj in Des}
s_j \in \Des_R(\bfw(\tau_{T^{(i)}_{\alpha}}) s_{m_i} s_{m_i + 1} \cdots s_{j-1})\cap \Des_R(\bfw(\tau_T) s_{m_i} s_{m_i + 1} \cdots s_{j-1}).
\end{equation}
Here $s_{m_i} s_{m_i + 1} \cdots s_{j - 1}$ is regarded as the identity when $j = m_i$.
Finally, applying Lemma~\ref{lem: left order and descent} to ~\eqref{eq: sj in Des}
yields that $\bfw(T^{(i)}_{\alpha}) \preceq_L \bfw(T)$, as required.
\end{proof}

Combining Lemma \ref{lem: char of tab in T^i} with the equalities $L(\tau)_i^1 = \tau_i^2$ and $L(\tau)_{i+1}^1 = \min(\tau_i^1, \tau_{i+1}^1)$, we derive that $\partial^{(i)}_1$ is well-defined.

\begin{lemma}\label{lem: phi_i is epic}
For $i \in \calI(\alpha)$, $\partial_1^{(i)}: \bfP_{\bal^{(i)}} \ra H_n(0) \cdot T^{(i)}_{\alpha}$ is a surjective $H_n(0)$-module homomorphism.
\end{lemma}

\begin{proof}
For each $T \in H_n(0) \cdot T^{(i)}_{\alpha}$,
let $\tau_T$ be the filling of $\trd(\bal^{(i)})$ defined in~\eqref{eq: tau_T}.
The surjectivity of $\partial_1^{(i)}$ is straightforward since
$\tau_T \in \SRT(\bal^{(i)})$ and $L(\tau_T)=T$.
Thus, to prove our assertion, it suffices to show that
\[
\partial_1^{(i)}(\pi_k \cdot \tau) = \pi_k \cdot \partial_1^{(i)}(\tau)
\]
for all $k = 1,2, \ldots, n-1$ and $\tau \in \SRT(\bal^{(i)})$.

{\it Case 1: $\pi_k \cdot \tau = \tau$.}
If $\partial_1^{(i)}(\tau) = 0$, then there is nothing to prove.
Suppose that $\partial_1^{(i)}(\tau) \neq 0$, that is, $L(\tau) \in \SRT(\alpha)$.
We claim that $k \notin \Des(L(\tau))$.
If $k = \tau_i^1$ and $k+1 = \tau_i^2$, then $k \in L(\tau)_{i+1}^\bullet$ and $k+1 \in L(\tau)_i^\bullet$. 
If $k \in \tau_{i+1}^\bullet$ and $k+1 = \tau_i^1$, then both $k$ and $k+1$ are in $L(\tau)_{i+1}^\bullet$.
In the remaining cases, from the fact that $k$ is weakly right of $k+1$ in $\tau$ it follows that $k$ is weakly right of $k+1$ in $L(\tau)$.
For any cases we can see that $k \notin \Des(L(\tau))$.

{\it Case 2: $\pi_k \cdot \tau = 0$.}
If $\partial_1^{(i)}(\tau) = 0$, then there is nothing to prove.
Suppose that $\partial_1^{(i)}(\tau) \neq 0$.
Since $k$ and $k+1$ are in the same row of $\tau$,
$k$ is the top and $k+1$ is the bottom for some two consecutive columns of $\tau$.
If $k \neq \tau_i^1$, then $k$ and $k+1$ are still in the same row of $L(\tau)$,
so $\pi_k \cdot L(\tau) = \pi_k \cdot \partial_1^{(i)}(\tau) = 0$, as required.
Assume that $k = \tau_i^1$.
Note that $|\tau_i^\bullet| = \alpha_i +1 \geq 2$ and $\tau_i^2$ greater than both $k$ and $k+1$.
By the definition of $L(\tau)$, we have that $L(\tau)_i^1 = \tau_i^2 > L(\tau)_{i+1}^{-1} = k+1$.
This implies that $\partial_1^{(i)}(\tau)=0$, which contradicts to our assumption $\partial_1^{(i)}(\tau) \neq 0$.

{\it Case 3: $\pi_k \cdot \tau = s_k \cdot \tau$.}
First, consider the case where $\partial_1^{(i)}(\tau) = 0$, that is, $L(\tau) \notin \SRT(\alpha)$.
Then $\tau$ must satisfy either $\tau_i^2 > \tau_{i+1}^{-1}$
or $\min(\tau_i^1, \tau_{i+1}^1) > \tau_{i+2}^{-1}$.
Thus, in order to $L(\pi_k \cdot \tau) \in \SRT(\alpha)$,
either $\tau_i^2 = k+1$ and  $\tau_{i+1}^{-1}=k$ or
$\min(\tau_i^1, \tau_{i+1}^1)= k+1$ and $\tau_{i+2}^{-1}=k$.
However, these are absurd because $k$ is strictly left of $k+1$ in $\tau$.

Next, consider the case where $\partial_1^{(i)}(\tau) \neq 0$, that is, $L(\tau) \in \SRT(\alpha)$.
Since $\pi_k \cdot \tau = s_k \cdot \tau$, $k$ is strictly left of $k+1$ in $\tau$.
Therefore, $k$ is weakly left of $k+1$ in $L(\tau)$ by the definition of $L(\tau)$.
Hence if neither $k$ and $k+1$ are in the same column in $L(\tau)$
nor they are in the same row in $L(\tau)$,
then $\pi_k \cdot L(\tau) = s_k \cdot L(\tau)$.
Therefore, in such case, we have that
\[
\pi_k \cdot \partial_1^{(i)}(\tau) = \pi_k \cdot L(\tau) = s_k \cdot L(\tau) = L(s_k \cdot \tau) = L(\pi_k \cdot \tau)= \partial_1^{(i)}(\pi_k \cdot \tau) \, .
\]

Suppose that $k$ and $k+1$ are in the same column in $L(\tau)$.
This is possible only the case where $k = \tau_i^1$ and $k+1 \in \tau_{i+1}^\bullet$
since $k$ is strictly left of $k+1$ in $\tau$.
Moreover, $k+1 \neq \tau_{i+1}^{-1}$ since $\pi_k \cdot \tau = s_k \cdot \tau$.
Hence $k+1 = (\pi_k \cdot \tau)_{i}^1$ and $k \in (\pi_k \cdot \tau)_{i+1}^\bullet$,
which implies that $L(\tau) = L(\pi_k \cdot \tau)$.
Therefore, we have
\[
\pi_k \cdot \partial_1^{(i)}(\tau) = \pi_k \cdot L(\tau) = L(\tau) = L(\pi_k \cdot \tau)= \partial_1^{(i)}(\pi_k \cdot \tau) \, .
\]
Here the second equality follows from the assumption that $k$ and $k+1$ are in the same column in $L(\tau)$.

Suppose that $k$ and $k+1$ are in the same row in $L(\tau)$.
Then $\pi_k \cdot L(\tau) = 0$.
In addition, since $\pi_k \cdot \tau = s_k \cdot \tau$, we have that either $L(\tau_{i+1}^1) = k$ and $L(\tau)_{i+2}^{-1} = k+1$, or $L(\tau)_i^1 = k$ and $L(\tau)_{i+1}^{-1} = k+1$.
In case where $L(\tau)_{i+1}^1 = k$ and $L(\tau)_{i+2}^{-1} = k+1$, the assumption $\pi_k \cdot \tau = s_k \cdot \tau$ implies that
$L(\pi_k \cdot \tau)_{i+1}^1 = k+1$ and $L(\pi_k \cdot \tau)_{i+2}^{-1} = k$.
Thus, $L(\pi_k \cdot \tau) \notin \SRT(\alpha)$, that is, $\partial_1^{(i)}(\pi_k \cdot \tau) = 0$ as desired.
In case where $L(\tau)_i^1 = k$ and $L(\tau)_{i+1}^{-1} = k+1$, one can easily see that $L(\pi_k \cdot \tau) \notin \SRT(\alpha)$.
Thus $\pi_k \cdot \partial_1^{(i)}(\tau) = 0 = \partial_1^{(i)}(\pi_k \cdot \tau)$.
\end{proof}

Due to Lemma~\ref{lem: char of tab in T^i} and Lemma \ref{lem: phi_i is epic},
we can view
$\partial_1 = \sum_{i \in \calI(\alpha)} \partial_1^{(i)}$ as an
$H_n(0)$-module homomorphism from $\bigoplus_{i \in \calI(\alpha)} \bfP_{\bal^{(i)}}$ onto $\Omega(\calV_\alpha)$.
Now, we verify that $\partial_1$ is an essential epimorphism,
that is, $\ker(\partial_1) \subseteq \rad(\bigoplus_{i \in \calI(\alpha)} \bfP_{\bal^{(i)}})$.

To ease notation,
we write $\tauf$ for the source tableau $\tau_{\bal^{(i)}}$ in $\SRT(\bal^{(i)})$.
When $i \neq \ell(\alpha)-1$,
we can see that
\[
\begin{aligned}
(\tauf)_{i+1}^q &= m_i+1+q \quad \text{for } 1 \leq q \leq \alpha_{i+1}-1\, \text{, and} \\
(\tauf)_{i+2}^q &= m_{i+1}+q \quad \text{for } 1 \leq q \leq \alpha_{i+2}\, , \end{aligned}
\]
where $m_i = \sum_{j = 1}^{i}\alpha_j$.
Let $\taus$ denote the $\SRT$ of shape $\bal^{(i)}$ such that
\begin{align*}
\begin{aligned}
(\taus)_{i+1}^q &= m_i+1+ \alpha_{i+2}+q \quad \text{for } 1 \leq q \leq \alpha_{i+1}-1\, , \\
(\taus)_{i+2}^q &= m_{i}+1+q \quad \text{for } 1 \leq q \leq \alpha_{i+2}\, \text{, and} \\
(\taus)_p &= (\tauf)_p \quad \text{for }p \neq i, i+1.
\end{aligned}
\end{align*}
For example, if $\alpha = (1,3,3,1)$ and $i = 1$, then
\[
\tauf =
\begin{array}{l}
\begin{ytableau}
\none & \none & 5 & 8 \\
\none & \none & 6 \\
\none & \none & 7 \\
\none & 3 \\
1& 4 \\
2 \\
\end{ytableau}
\end{array}
\quad \text{and} \quad
\taus =
\begin{array}{l}
\begin{ytableau}
\none & \none & 3 & 8 \\
\none & \none & 4 \\
\none & \none & 5 \\
\none & 6 \\
1& 7 \\
2 \\
\end{ytableau}.
\end{array}
\]
Observe that $(\tauf)^\bullet_j = (\taus)^\bullet_j$ for $j \neq i+1, i+2$.

\begin{lemma}\label{lem: phi_i is essential}
For $i \in \calI(\alpha)$, $\ker(\partial_1^{(i)}) \subseteq \rad(\bfP_{\bal^{(i)}} )$.
\end{lemma}
\begin{proof}
If $i = \ell(\alpha)-1$, then $\bal^{(i)}$ is a composition.
Therefore, $\rad(\bfP_{\bal^{(i)}})$ is the $\C$-span of $\SRT(\bal^{(i)})\setminus \{ \tauf \}$.
Since $\partial_1^{(i)}(\tauf) \neq 0$, this implies that $\ker(\partial_1^{(i)}) \subseteq \rad(\bfP_{\bal^{(i)}} )$.

Suppose that $i \neq \ell(\alpha)-1$.
Let
\begin{equation}\label{compositions of alphai}
\begin{aligned}
\beta^{(1)} & = (\alpha_1, \alpha_2, \ldots, \alpha_{i-1}, \alpha_i +1, \alpha_{i+1} - 1, \alpha_{i+2} , \alpha_{i+3}, \ldots, \alpha_{\ell(\alpha)}), \\
\beta^{(2)} & = (\alpha_1, \alpha_2, \ldots, \alpha_{i-1}, \alpha_i +1, \alpha_{i+1} - 1 + \alpha_{i+2} , \alpha_{i+3}, \ldots, \alpha_{\ell(\alpha)}).
\end{aligned}
\end{equation}
To ease notation, we denote the source tableaux of $\bfP_{\beta^{(1)}}$ and $\bfP_{\beta^{(2)}}$ by $\tau^{(1)}$ and $\tau^{(2)}$, respectively.
By Theorem~\ref{thm: bfP isom to calP}, we may choose an $H_n(0)$-module isomorphism
\[
f:\bfP_{\bal^{(i)}} \ra  \bfP_{\beta^{(1)}} \oplus \bfP_{\beta^{(2)}}.
\]
Let
$$
f(\tauf) = \sum_{\tau \in \SRT(\beta^{(1)})} c_\tau \tau + \sum_{\tau \in \SRT(\beta^{(2)})} d_\tau \tau
\quad \text{for $c_\tau, d_\tau \in \C$}.
$$
Since $f(\tauf)$ is a generator of $\bfP_{\beta^{(1)}} \oplus \bfP_{\beta^{(2)}}$, $c_{\tau^{(1)}}$ and $d_{\tau^{(2)}}$ are nonzero.

We claim that $[\tauf, \taus]^\rmc \subset \rad (\bfP_{\bal^{(i)}})$.
Take any $\tau \notin [\tauf,\taus]$.
To get $\tau$ from $\tauf$, there should exist an $H_n(0)$-action switching two entries such that at least one of them lies apart from the $(i+1)$st and $(i+2)$nd columns.
Thus there exist $\sigma, \rho \in \SG_n$ and $k \notin [m_i+2, m_{i+2} - 1]$ such that
$$
\tau = \pi_\sigma \pi_k \pi_\rho \cdot \tauf, \quad
\pi_\rho \cdot \tauf \in [\tauf,\taus],
\quad \text{and} \quad
\pi_k \pi_\rho \cdot \tauf = s_k \cdot (\pi_\rho \cdot \tauf).
$$
Ignoring the columns filled with entries $[m_i+2, m_{i+2}]$,
we can see that all $\pi_\rho \cdot \tauf$, $\tau^{(1)}$, and $\tau^{(2)}$ are the same.
This implies that
$\pi_k \cdot \tau^{(j)} = s_k \cdot \tau^{(j)}$ for $j = 1,2$.
In all, we have
\begin{align*}
f(\tau) &=  \pi_\sigma \pi_k \pi_\rho \cdot f(\tauf)\\
&=
\pi_\sigma \pi_k \pi_\rho \cdot \left(
\sum_{\tau \in \SRT(\beta^{(1)})} c_\tau \tau + \sum_{\tau \in \SRT(\beta^{(2)})} d_\tau \tau
\right) \\
&= \sum_{\substack{\tau \in \SRT(\beta^{(1)}) \\ \tau > \tau^{(1)} }} c'_\tau \tau + \sum_{\substack{\tau \in \SRT(\beta^{(2)})  \\ \tau > \tau^{(2)} }} d'_\tau \tau
\end{align*}
for some $c'_\tau, d'_\tau \in \C$.
This implies that $f(\tau) \in \rad(\bfP_{\beta^{(1)}} \oplus \bfP_{\beta^{(2)}})$, hence $\tau \in \rad(\bfP_{\bal^{(i)}})$.

By virtue of the above discussion, to complete our assertion, it is enough to show that $\ker(\partial_1^{(i)}) \subseteq \C [\tauf, \taus]^\rmc$, equivalently, $L(\tau) \in \SRT(\alpha)$ for every $\tau \in [\tauf,\taus]$.
But this is obvious since $L(\tau)_i^1 = \tau_i^2 = m_{i-1}+2$, $L(\tau)_{i+1}^1 = \tau_{i}^{1} = m_{i-1}+1$, and $L(\tau)_{i+1}^{-1}, L(\tau)_{i+2}^{-1} \in [m_i+2, m_{i+2}]$.
\end{proof}

We are now in place to prove Theorem~\ref{main thm for V}.

\begin{proof}[Proof of Theorem~\ref{main thm for V}]
(a)
As mentioned after the proof of Lemma \ref{lem: phi_i is epic}, $\partial_1: \bigoplus_{i \in \calI(\alpha)} \bfP_{\bal^{(i)}} \ra \Omega(\calV_\alpha)$ is a surjective $H_n(0)$-module homomorphism.
Therefore, we only need to check $\ker(\partial_1) \subseteq \rad \left(\bigoplus_{i \in \calI(\alpha)} \bfP_{\bal^{(i)}} \right)$ to complete the proof of the assertion.
Let
$$
\bfT := \bigoplus_{i \in \calI(\alpha)} \C[\tauf,\taus]
\quad \text{and} \quad
\bfB := \bigoplus_{i \in \calI(\alpha)} \C[\tauf, \taus]^\rmc.
$$
In the proof of Lemma \ref{lem: phi_i is essential} we see that $[\tauf, \taus]^\rmc \subseteq \rad \, \bfP_{\bal^{(i)}}$ for $i \in \calI(\alpha)$ and thus
$\bfB \subseteq \rad \left(\bigoplus_{i \in \calI(\alpha)} \bfP_{\bal^{(i)}} \right)$.

In the following, we will prove $\ker(\partial_1) \subseteq \bfB$,
which is obviously a stronger inclusion than necessary.
We begin by collecting the following properties which were shown in the proof of Lemma \ref{lem: phi_i is essential}:
For all $i \in \calI(\alpha)$, $1\le j < i$, and $\tau \in [\tauf,\taus]$,
\begin{align*}
&\ker(\partial_1^{(i)}) \subseteq \C[\tauf, \taus]^\rmc,\\
&\partial_1^{(i)}(\tau)_i^1 = m_{i-1}+2\,, \text{ and} \\
&\partial_1^{(i)}(\tau)_j^1 = m_{j-1} +1.
\end{align*}
Therefore, for any $i, j \in \calI(\alpha)$ with $j < i$, if $\tau \in [\tauf,\taus] \subset \bfP_{\bal^{(i)}}$ and $\tau' \in [\taufj,\tausj] \subset \bfP_{\bal^{(j)}}$, then $\partial_1(\tau)_j^1 = \partial_1^{(i)} (\tau)_j^1 = m_{j-1} + 1$ and $\partial_1(\tau')_j^1 = \partial_1^{(j)} (\tau')_j^1 = m_{j-1} + 2$, that is, $\partial_1(\tau) \neq \partial_1(\tau')$.
This implies that that the set $\{\partial_1(\tau) \mid \tau \in [\tauf,\taus] \text{ for $i \in \calI(\alpha)$}\}$ is linearly independent, hence every $\bfx \in \ker(\partial_1) \setminus \{0\}$ is decomposed as $\bfx = \bfx^{(1)} + \bfx^{(2)}$ for some $\bfx^{(1)} \in \bfT$ and $\bfx^{(2)} \in \bfB \setminus \{0\}$.

We claim that $\bfx^{(1)} = 0$.
Suppose on the contrary that $\bfx^{(1)} \neq 0$.
Let
\[
\partial_1(\bfx^{(1)}) = \sum_{T \in \SRT(\alpha) \cap \Omega(\calV_\alpha) }c_T T
\quad \text{and} \quad
\partial_1(\bfx^{(2)}) = \sum_{T \in \SRT(\alpha) \cap \Omega(\calV_\alpha) }d_T T.
\]
Since $\partial_1(\bfx^{(1)}) \neq 0$, there exists $T \in \SRT(\alpha) \cap \Omega(\calV_\alpha)$ such that $c_{T} \neq 0$.
In addition, since $\SRT(\alpha) \cap \Omega(\calV_\alpha) $ is linearly independent and $\partial_1(\bfx) = 0$, we have $c_{T} = -d_{T}$.
Therefore, there exist $i,j \in \calI(\alpha)$, $\tau_\bfT \in [\tauf, \taus]$, and $\tau_\bfB \in [\taufj, \tausj]^\rmc$ such that $\partial_1(\tau_\bfT) = T = \partial_1(\tau_\bfB)$.
Since $\{\partial_1(\tau) \mid \tau \in \SRT(\bal^{(i)})\} \setminus \{0\}$ is linearly independent, we have $i \neq j$.
Note that $\partial_1(\tau_\bfB) = \partial_1^{(j)}(\tau_\bfB) \in H_n(0)\cdot T^{(j)}_\alpha$.
By Lemma \ref{lem: char of tab in T^i}, we have $T_j^1 > T_{j+1}^1$.
On the other hand, since $T = \partial_1^{(i)}(\tauT)$ and $\tau_\bfT \in [\tauf, \taus]$, $T$ is equal to $T^{(i)}_{\alpha}$ except for the $(i+1)$st and $(i+2)$nd columns.
Note that the $(i+1)$st and $(i+2)$nd columns of them are filled with $\{ (\tauT)_i^1 \} \cup [m_i+2, m_{i+2}]$ and
$T_{i+1}^1 = \partial_1^{(i)}(\tauT)_{i+1}^1 = m_{i-1} + 1$.
This shows that $T_j^1 < T_{j+1}^1$, which is absurd.
Hence $\bfx^{(1)} = 0$, and it follows that $\ker(\partial_1) \subseteq \bfB$, as required.

(b) For all $\beta \models n$, it is known that
\[\Ext_{H_n(0)}^1(\calV_\alpha,\bfF_{\beta})={\rm Hom}_{H_n(0)}(P_1,\bfF_{\beta}) \]
with $P_1:=\displaystyle{ \bigoplus_{i \in \calI(\alpha)} \bfP_{\bal^{(i)}}}$
(for instance, see~\cite[Corollary 2.5.4]{91Benson}).
In case of projective indecomposable modules, one has that
$\dim{\rm Hom}_{H_n(0)}(\bfP_\gamma,\bfF_{\gamma'})=\delta_{\gamma,\gamma'}$ for all $\gamma, \gamma' \models n$
(see~\cite[Lemma 1.7.5]{91Benson}).
This tells us that $\dim\Ext_{H_n(0)}^1(\calV_\alpha,\bfF_{\beta})$ counts the multiplicity of $\bfP_{\beta}$ in
the decomposition of $P_1$ into indecomposables.
The indecomposables which occur in the decomposition are precisely $\bfP_{\beta}$ with $\beta \in \calJ(\alpha)$.
We claim that all of them are multiplicity-free.
For $i \in \calI(\alpha)$, note that
$[\bal^{(i)}]=\{\beta^{(1)}, \beta^{(2)}\}$ with $\beta^{(1)}, \beta^{(2)}$ in \eqref{compositions of alphai}.
Obviously $\beta^{(1)}$ and $\beta^{(2)}$ are distinct.
Furthermore, for $i<j$, $[\bal^{(i)}]$ and $[\bal^{(j)}]$ are disjoint since
the $i$th entry of the compositions in the former is $\alpha_i+1$, whereas
that of the compositions in the latter is $\alpha_i$.
Hence the claim is verified, which completes the proof.
\end{proof}

\subsection{Proof of Theorem \protect\ref{Thm:injective hull V}}
\label{Subsec: Proof Par0 is injective hull}

We begin by introducing the necessary terminologies, notations, and lemmas.
First, we recall the notation related to parabolic subgroups of $\SG_n$.
For each subset $I$ of $[n-1]$, we write $(\SG_n)_I$ for the parabolic subgroup of $\SG_n$
generated by simple transpositions $s_i$ with $i\in I$
and $\paralong{I}$ for the longest element of $(\SG_n)_I$.
When $I$ is a subinterval $[k_1,k_2]$ of $[n-1]$ and $c \in I$, we write $(\SG_n)_I^{(c)}$ for
\begin{equation*}
\left\{
\sigma \in (\SG_{n})_I \, \middle| \,
\genfrac{}{}{0pt}{}{ \ \sigma(k_1) < \sigma(k_1 + 1) <
\cdots <  \sigma(c) \text{ and}
}{\sigma(c + 1) < \sigma(c + 2) <
\cdots
< \sigma(k_2+1) \
}
\right\},
\end{equation*}
and $\pqlong{I}{c}$ for the longest element of $(\SG_n)_I^{(c)}$ (see~\cite[Chapter 2]{05BB}).

Next, we introduce the sink tableau of $\bfP_\bal$.
Given a generalized composition $\bal$ of $n$, $\bfP_\bal$ contains a unique tableau $T$ such that $\pi_i \cdot T = 0$ or $T$ for all $i \in [n-1]$. We call it the {\em sink tableau} of $\bfP_\bal$, denoted by $T^{\leftarrow}_\bal$.
Explicitly, $T^{\leftarrow}_\bal$ is obtained by filling in $\trd(\bal)$ with entries $1, 2, \ldots, n$ from left to right, and from top to bottom.
Let us define a bijection
\[
\chi_\bal: \SRT(\bal) \ra \bigcup_{\beta \in [\bal]} \SRT(\beta), \quad T \mapsto T',
\]
where $T'$ is uniquely determined by the condition $\bfw(T) = \bfw(T')$.
With this bijection, we define
\begin{equation*}
\TLba{\beta}{\bal}  := \chi_\bal^{-1}(T^{\leftarrow}_\beta) \quad \text{ for every } \beta \in [\bal].
\end{equation*}
For $\beta \in [\balalp]$, we let
$$
\tJ{\beta}{\balalp}:= \{i \in [n-1] \mid \pi_i \cdot T^{\leftarrow}_\beta = 0, \text{ but } \pi_i \cdot \TLba{\beta}{\balalp} \neq 0 \}.
$$

For each $1 \le i \le n-1$, let $\opi_i := \pi_i -1$.
Pick up any reduced expression $s_{i_1}\cdots s_{i_p}$ for  $\sigma \in \SG_n$.
Let $\opi_{\sigma}$ be the element of $H_n(0)$ defined by $\opi_{\sigma} := \opi_{i_1} \cdots \opi_{i_p}$.
It is well known that the element $\opi_{\sigma}$ is independent of the choice of reduced expressions.

\begin{lemma}\label{lem: zero characterization}
{\rm \cite[Lemma 3.4 (1)]{21JKLO}}
For any $\sigma, \rho \in \SG_n$, $\pi_\sigma \opi_\rho$ is nonzero if and only if $\ell(\sigma\rho) = \ell(\sigma) + \ell(\rho)$.
\end{lemma}

The following lemma gives an explicit description for $\soc(\bfP_{\balalp})$.

\begin{lemma}\label{Lem: nessary lemmas}
For $\beta \in [\balalp]$, $\C T^{\leftarrow}_\beta$ is isomorphic to
$\C\left( \opi_{w_0(\tJ{\beta}{\balalp})} \cdot \TLba{\beta}{\balalp}\right)$
as an $H_n(0)$-module.
\end{lemma}

\begin{proof}
First, we claim that
$\opi_{\longelto{\tJ{\beta}{\balalp}}} \cdot \TLba{\beta}{\balalp}$ is stabilized under the action of $\pi_i$
for all $i \in \Des(T_\beta^{\leftarrow})^{\rm c}$.
Note that $\opi_{\longelto{\tJ{\beta}{\balalp}}} \cdot \TLba{\beta}{\balalp}$ is of the form
\begin{align}\label{eq: opi Tab expansion}
\sum_{T \in [\TLba{\beta}{\balalp}, T_{\balalp}^{\leftarrow}]} c_T T
\quad \text{for some $c_T \in \Z$.}
\end{align}
But, from the definitions of $\TLba{\beta}{\balalp}$ and $T_{\balalp}^{\leftarrow}$, it follows that $\pi_i \cdot T = T$ for $i \in \Des(T_\beta^{\leftarrow})^{\rm c}$.
Thus our claim is verified.

Next, we claim that $\pi_i \cdot (\opi_{\longelto{\tJ{\beta}{\balalp}}} \cdot \TLba{\beta}{\balalp}) = 0$ for all $i \in \Des(T_\beta^{\leftarrow})$.
Take any $i \in \Des(T_\beta^{\leftarrow})$.
Note that $T(\tS_{k_0}) = \{1,2,\ldots,\ell(\alpha)\}$ for any $T \in [\TLba{\beta}{\balalp}, T_{\balalp}^{\leftarrow}]$.
Therefore, if $1 \leq i < \ell(\alpha)$, then $\pi_i \opi_{\longelto{\tJ{\beta}{\balalp}}} \cdot \TLba{\beta}{\balalp} = 0$ by~\eqref{eq: opi Tab expansion}.
In case where $i \ge \ell(\alpha)$,  $i \in \tJ{\beta}{\balalp}$ and thus
$\pi_i \opi_{\longelto{\tJ{\beta}{\balalp}}}= 0$ by Lemma~\ref{lem: zero characterization}.
\end{proof}

\begin{example}
Given $\alpha = (2^3)$, let $\beta = (1^2,2,1^2)$ and $\gamma = (2^2,1^2)$ be compositions in $[\balalp] = [(1) \oplus (1) \oplus (2,1^2)]$.
Note that
$$
\ytableausetup{aligntableaux=center, boxsize=0.8em}
T^{\leftarrow}_{\beta} = \begin{ytableau}
\none & \none & 1 & 2 & 3 \\
4 & 5 & 6
\end{ytableau}
\quad
\TLba{\beta}{\balalp} = \begin{ytableau}
\none & \none & 1 & 2 & 3 \\
\none & \none & 6 \\
\none & 5 \\
4
\end{ytableau}
\quad\text{and}\quad
T^{\leftarrow}_{\gamma}  = \begin{ytableau}
\none & 1 & 2 & 3  \\
4 & 5\\
6
\end{ytableau}
\quad
\TLba{\gamma}{\balalp}  = \begin{ytableau}
\none & \none & 1 & 2 & 3  \\
\none & \none & 5 \\
\none & 4 \\
6
\end{ytableau}.
$$
Since $\tJ{\beta}{\balalp} = \{4,5\}$ and
$\tJ{\gamma}{\balalp} = \{4\}$, it follows that $\longelto{\tJ{\beta}{\balalp}} = s_4 s_5 s_4$ and $\longelto{\tJ{\gamma}{\balalp}} = s_4 $.
Thus we have
\begin{align*}
\ytableausetup{aligntableaux=center, boxsize=0.8em}
\C T^{\leftarrow}_{\beta} & \cong
\C\left(
\begin{ytableau}
\none & \none & 1 & 2 & 3 \\
\none & \none & 6 \\
\none & 5 \\
4
\end{ytableau}
-
\begin{ytableau}
\none & \none & 1 & 2 & 3 \\
\none & \none & 5 \\
\none & 6 \\
4
\end{ytableau}
-
\begin{ytableau}
\none & \none & 1 & 2 & 3 \\
\none & \none & 6 \\
\none & 4 \\
5
\end{ytableau}
+
\begin{ytableau}
\none & \none & 1 & 2 & 3 \\
\none & \none & 4 \\
\none & 6 \\
5
\end{ytableau}
+
\begin{ytableau}
\none & \none & 1 & 2 & 3 \\
\none & \none & 5 \\
\none & 4 \\
6
\end{ytableau}
-
\begin{ytableau}
\none & \none & 1 & 2 & 3 \\
\none & \none & 4 \\
\none & 5 \\
6
\end{ytableau}
\right)
\\
\C T^{\leftarrow}_{\gamma} & \cong
\C\left(
\begin{ytableau}
\none & \none & 1 & 2 & 3  \\
\none & \none & 5 \\
\none & 4 \\
6
\end{ytableau}
-\begin{ytableau}
\none & \none & 1 & 2 & 3  \\
\none & \none & 4 \\
\none & 5 \\
6
\end{ytableau}\right).
\end{align*}
\end{example}
\medskip

\begin{proof}[Proof of Theorem~\ref{Thm:injective hull V}]
We first claim that $\injhull: \calV_\alpha \to \bfP_{\balalp}$ is an $H_n(0)$-module homomorphism,
that is,
\begin{displaymath}
\injhull(\pi_i \cdot \calT) = \pi_i \cdot \injhull(\calT)  \quad\text{ for } i=1,2,\ldots,n-1 \text{ and }\calT \in \SIT(\alpha).
\end{displaymath}
Let us fix $1 \le i \le n-1$ and $\calT \in \SIT(\alpha)$.
Let $0 \leq x, y \leq m$ be integers satisfying that
$i \in \calT(\tS_{k_x})$ and $i+1 \in \calT(\tS_{k_y})$.

{\it Case 1: $\pi_i \cdot \calT = \calT$.}
First, we handle the case where $x=0$.
Then $i$ will be placed in the top row in $\TcalTt$.
In view of the given condition $\pi_i \cdot \calT = \calT$, one sees that $x \neq y$.
This implies that $i+1$ is strictly below $i$ in $\TcalTt$.
Next, we handle the case where $x>0$. 
The condition $\pi_i \cdot \calT = \calT$ says that $0 < x \leq y$, thus $i+1$ is strictly below $i$ in $\TcalTt$.
In either case, it is immediate from~\eqref{eq: action for ribbon} that $\pi_i \cdot \TcalTt = \TcalTt$.

{\it Case 2: $\pi_i \cdot \calT = 0$.}
From~\eqref{eq: action on SIT} it follows that $i$ and $i+1$ are in the first column in $\calT$, that is, $x = y = 0$. Hence, in $\TcalTt$, both of them will appear in $\TcalTt(\tS_{k_0})$.
As in {\it Case 1}, one can derive from~\eqref{eq: action for ribbon} that $\pi_i \cdot \TcalTt = 0$.

{\it Case 3: $\pi_i \cdot \calT = s_i \cdot \calT$.}
We claim that $\injhull (s_i \cdot \calT) = s_i \cdot \TcalTt$.
Observe that $i$ appears strictly above $i+1$ in $\calT$.
If $i+1 \in \calT(\tS_{k_0})$, then we see that $i \notin \calT(\tS_{k_0})$, which means that
$i$ appears strictly left of $i+1$ in $\TcalTt$.
Otherwise, we also see that $i \notin \calT(\tS_{k_0})$.
More precisely, if $i+1 \notin \calT(\tS_{k_0})$ and $i \in \calT(\tS_{k_0})$, then $\calT$ is not an SIT since the entries in the row containing $i+1$ of $\calT$ do not increase from left to right.
It follows from the construction of $\TcalTt$ that $i$ is strictly below $i+1$ in $\TcalTt$.
In either case, it holds that $T^{s_i \cdot \calT} = s_i \cdot \TcalTt$.
Thus we conclude that
$$
\pi_i \cdot \injhull(\calT) =  \pi_i \cdot \TcalTt = T^{s_i \cdot \calT} = \injhull(s_i \cdot \calT) = \injhull(\pi_i \cdot \calT).
$$

We next claim that $\bfP_{\balalp}$ is an essential extension of $\injhull(\calV_\alpha)$.
To do this, we see that $\soc(\bfP_{\balalp}) \subset \injhull(\calV_\alpha)$.
Note that
$$\soc(\bfP_{\balalp}) \cong \soc \Big(\bigoplus_{\beta \in [\balalp]} \bfP_{\beta} \Big)
\cong \bigoplus_{\beta \in [\balalp]} \C T^{\leftarrow}_{\beta}.
$$
In view of Lemma~\ref{Lem: nessary lemmas}, one sees that
\begin{align}\label{eq: socle of bfP balalp}
\soc(\bfP_{\balalp}) =
\bigoplus_{\beta \in [\balalp]}
\C\left(
\opi_{\longelto{\tJ{\beta}{\balalp}}} \cdot \TLba{\beta}{\balalp}\right).
\end{align}
Choose any $\beta \in [\balalp]$.
Then
$$
\opi_{\longelto{\tJ{\beta}{\balalp}}} \cdot \TLba{\beta}{\balalp}
= \sum_{\sigma \in (\SG_n)_{\tJ{\beta}{\balalp}}} (-1)^{\ell(\longelto{\tJ{\beta}{\balalp}}) - \ell(\sigma)} \pi_\sigma \cdot \TLba{\beta}{\balalp}.
$$
For $\sigma \in (\SG_n)_{\tJ{\beta}{\balalp}}$, since $(\pi_\sigma \cdot \TLba{\beta}{\balalp})(\tS_{k_0}) = \{1,2, \ldots, \ell(\alpha)\}$,
we have
$$
(\pi_\sigma \cdot \TLba{\beta}{\balalp})^1_{m+k_j-1} <
\begin{cases}
(\pi_\sigma \cdot \TLba{\beta}{\balalp})^{1}_j & \text{if $1 \leq j < m$}, \\[1ex]
(\pi_\sigma \cdot \TLba{\beta}{\balalp})^{2}_j & \text{if $j = m$}.
\end{cases}
$$
It means that $\pi_\sigma \cdot \TLba{\beta}{\balalp} \in \injhull(\calV_\alpha)$ for all $\sigma \in  (\SG_n)_{\tJ{\beta}{\balalp}}$.
Combining this with~\eqref{eq: socle of bfP balalp} yields that $\soc(\bfP_{\balalp}) \subset \injhull(\calV_\alpha)$.
\end{proof}

\subsection{Proof of Theorem \protect\ref{Thm: Main Section4}}
\label{Subsec: Proof of Theorem4}

Throughout this section, let us fix an integer $1 \le j \le m$ unless otherwise stated.

Let $T \in \SRT(\balalp)$.
In the same notation as in Section~\ref{Sec: extensions of F by V},
we claim that
\begin{equation}\label{Eq: bstau neq 0}
\bstau{T}{j} \neq 0 \quad \text{if and only if} \quad T^{1+\delta_{j,m}}_j < T^1_{m+k_j-1}.
\end{equation}
This is because that if $T^{1+\delta_{j,m}}_j < T^1_{m+k_j-1}$, then $\ebfw{T}{j} < T^1_{m+k_j-1}$
and therefore $\bstau{T}{j} \neq 0$.
Otherwise, $\bstau{T}{j}$ should be zero since $\ebfw{T}{j} > T^1_{m+k_j-1}$.

Let $\beta \in [\balalpp{j}]$.
Recall that $\TLba{\beta}{\balalpp{j}}  = \chi_{\balalpp{j}}^{-1}(T^{\leftarrow}_\beta)$
and
$$
\tJ{\beta}{\balalpp{j}}= \{i \in [n-1] \mid \pi_i \cdot T^{\leftarrow}_\beta = 0, \text{ but } \pi_i \cdot \TLba{\beta}{\balalpp{j}} \neq 0 \}.
$$
Note that if $\min(\tJ{\beta}{\balalpp{j}}) \leq \ell(\alpha)$, then
\begin{equation}\label{Eq: property of Jbetabalalppj}
\min(\tJ{\beta}{\balalpp{j}}) = |\tS'_{k_0}|
\quad \text{and} \quad
\min\left(\tJ{\beta}{\balalpp{j}} \setminus \{|\tS'_{k_0}|\} \right)  > \ell(\alpha) + 1.
\end{equation}
Set
$$
\htJ{\beta}{\balalpp{j}}:=
\begin{cases}
\tJ{\beta}{\balalpp{j}} \setminus \{|\tS'_{k_0}|\} & \text{if $1 \leq \min(\tJ{\beta}{\balalpp{j}}) \leq \ell(\alpha)$,} \\
\tJ{\beta}{\balalpp{j}} & \text{otherwise,}
\end{cases}
$$
and
$$
\longeltt{\beta}{j}:=
\begin{cases}
\pqlong{[\ell(\alpha)]}{|\tS'_{k_0}|}
\cdot \paralong{\htJ{\beta}{\balalpp{j}}}
& \text{if $1 \leq \min(\tJ{\beta}{\balalpp{j}}) \leq \ell(\alpha)$}, \\
\paralong{\tJ{\beta}{\balalpp{j}}}
& \text{otherwise.}
\end{cases}
$$
In view of~\eqref{Eq: property of Jbetabalalppj}, we know that every element of $\SGGG{n}{[\ell(\alpha)]}{(|\tS'_{k_0}|)}$ commutes with that of $(\SG_n)_{\htJ{\beta}{\balalpp{j}}}$.
The following lemma is necessary to show that $\soc(\bigoplus_{1 \leq j \leq m}\bfP_{\balalpp{j}})  \subseteq \mathrm{Im}(\opone)$.

\begin{lemma}\label{Lem: nessary lemmas II}
For $1 \leq j \leq m$ and $\beta \in [\balalpp{j}]$,
$\C T^{\leftarrow}_\beta \cong \C (\opi_{\longeltt{\beta}{j}} \cdot \TLba{\beta}{\balalpp{j}})$ as $H_n(0)$-modules.
\end{lemma}

\begin{proof}
Let $1 \leq j \leq m$ and $\beta \in [\balalpp{j}]$.
If $\min(\tJ{\beta}{\balalpp{j}}) > \ell(\alpha)$, then one can prove the assertion in the same way as in Lemma~\ref{Lem: nessary lemmas}.
We now assume that $\min(\tJ{\beta}{\balalpp{j}}) \leq \ell(\alpha)$.
We first show that
\begin{equation*}
\pi_i \cdot (\opi_{\longeltt{\beta}{j}} \cdot \TLba{\beta}{\balalpp{j}}) = \opi_{\longeltt{\beta}{j}} \cdot \TLba{\beta}{\balalpp{j}}
\end{equation*}
for $i \notin \Des(T^\leftarrow_\beta)$. Since
\begin{align*}
\opi_{\longeltt{\beta}{j}} \cdot \TLba{\beta}{\balalpp{j}}
= \sum_{T \in [\TLba{\beta}{\balalpp{j}}, T_{\balalpp{j}}^{\leftarrow}]} c_T T
\quad \text{for some $c_T \in \Z$,}
\end{align*}
it suffices to show that $\pi_i \cdot T = T$ for $i \notin \Des(T_\beta^{\leftarrow})$ and $T \in [\TLba{\beta}{\balalpp{j}}, T_{\balalpp{j}}^{\leftarrow}]$.
Since $\{1,2,\ldots, \ell(\alpha)\}\subseteq \Des(T^\leftarrow_\beta)$ by definition, we only consider that $i \geq \ell(\alpha)+1$.
If $i=\ell(\alpha)+1$, then the assertion follows from the fact that $T(\tS'_{k_0})\cup T(\tS'_{k_{-1}}) = \{1,2,\ldots,\ell(\alpha)+1\}$.
Otherwise, from the definitions of $\TLba{\beta}{\balalpp{j}}$ and $T_{\balalpp{j}}^{\leftarrow}$, it follows that $\pi_i \cdot T = T$ for $i \notin \Des(T_\beta^{\leftarrow})$.
Thus our claim is verified.

We next show that
$\pi_i \cdot (\opi_{\longeltt{\beta}{j}} \cdot \TLba{\beta}{\balalpp{j}}) = 0$ for $i \in \Des(T^\leftarrow_\beta)$.
Take any $i \in \Des(T_\beta^{\leftarrow})$.
If $i > \ell(\alpha) + 1$, then $i \in \htJ{\beta}{\balalpp{j}}$. Therefore, by Lemma~\ref{lem: zero characterization}, we have $\pi_i \opi_{\longeltt{\beta}{j}} = 0$, which implies $\pi_i \opi_{\longeltt{\beta}{j}} \cdot T^{\leftarrow}_{\beta,\balalp} = 0$.
Suppose that $i \le \ell(\alpha) + 1$.
Since $\ell(\alpha) + 1 \notin \Des(T^\leftarrow_\beta)$, we have that $1 \le i \le \ell(\alpha)$.
If $i \in \Des_L(\longeltt{\beta}{j})$,\
then $\pi_i \opi_{\longeltt{\beta}{j}} = 0$.
Thus, $\pi_i \opi_{\longeltt{\beta}{j}} \cdot T^{\leftarrow}_{\beta,\balalp} = 0$.
Otherwise, we have $s_i
\pqlong{[\ell(\alpha)]}{|\tS'_{k_0}|}
= \sigma s_{i'}$
for some $\sigma \in \SGGG{n}{[\ell(\alpha)]}{(|\tS'_{k_0}|)}$ and $1\le i' \le \ell(\alpha)$ with $i' \neq |\tS'_{k_0}|$
since $\pqlong{[\ell(\alpha)]}{|\tS'_{k_0}|}$ is the unique longest element in $\SGGG{n}{[\ell(\alpha)]}{(|\tS'_{k_0}|)}$.
Combining this with~\cite[Lemma 3.2]{21JKLO}, we have that $\pi_i \opi_{\longeltt{\beta}{j}}
= \mathbf{h} \pi_{i'}$ for some $\mathbf{h} \in H_n(0)$ and
$1 \le i' \le \ell(\alpha)$ with $i' \neq |\tS'_{k_0}|$.
Since $\pi_{i'} \cdot  \TLba{\beta}{\balalpp{j}} = 0$ for all $1 \le i' \le \ell(\alpha)$ with $i' \neq |\tS'_{k_0}|$, it follows that
\begin{equation*}
\pi_i \cdot (\opi_{\longeltt{\beta}{j}} \cdot \TLba{\beta}{\balalpp{j}})
= \mathbf{h} \pi_{i'} \cdot  \TLba{\beta}{\balalpp{j}}
= 0.
\qedhere
\end{equation*}
\end{proof}

\begin{example}\label{Ex: Lemma Cbeta cong Cw0}
Let $\alpha = (2,1,2,3) \models 8$. Note that $\calK(\alpha) = \{0,1,3,4\}$ and $\ell(\alpha)=4$.
Then $\balalpp{2} = (1) \oplus (3,1^2) \oplus (1^2)$.
Let $\beta = (1,3,1^4)$ and $\gamma = (1,3,1,2,1)$ in $[\balalpp{2}]$.
Note that
\begin{align*}
\ytableausetup{aligntableaux=center, boxsize=0.9em}
\begin{tikzpicture}
\def \hhh{35mm}
\def \hhhh{25mm}
\def \vvv{30mm}
\node at (\hhh*0,0) (A1) {
$T^\leftarrow_{\beta} =
\begin{ytableau}
\none & {\color{red} 1} & {\color{red} 2} & {\color{red} 3} & {\color{red} 4} & 5 \\
\none & 6\\
{\color{red} 7} & 8
\end{ytableau}$};
\node at (\hhh*1,0) (A2) {
$\TLba{\beta}{\balalpp{2}} =
\begin{ytableau}
\none & \none & \none & \none & {\color{red} 4} & 5 \\
\none & {\color{red} 1} & {\color{red} 2} & {3}\\
\none & 6 \\
\none & 8 \\
{7}
\end{ytableau}$};
\node at (\hhh*2,0) (A3) {
$T^\leftarrow_{\gamma} =
\begin{ytableau}
\none & \none & \none & {\color{red} 1} & 2\\
\none & {\color{red} 3} & {\color{red} 4} & 5\\
\none & 6\\
{\color{red} 7} & 8
\end{ytableau}$};
\node at (\hhh*3,0) (A4) {
$\TLba{\gamma}{\balalpp{2}} =
\begin{ytableau}
\none & \none & \none & \none &  {\color{red} 1} & 2 \\
\none & {\color{red} 3} & {\color{red} 4} & 5\\
\none & 6 \\
\none & 8 \\
{7}
\end{ytableau}$};
\end{tikzpicture}
\end{align*}
Here the entries $i$ in red in each SRT $T$ are being used to indicate that $\pi_i \cdot T = 0$.
Since $\min(\tJ{\beta}{\balalpp{2}}) = 3 \le \ell(\alpha)$ and $ \min(\tJ{\gamma}{\balalpp{2}}) = 7 > \ell(\alpha)$, 
\[
\longeltt{\beta}{2} = s_2 s_3 s_4 s_1 s_2 s_3 \cdot s_7 
\qquad\text{and}\qquad
\longeltt{\gamma}{2} = s_7.
\]
Therefore, by Lemma~\ref{Lem: nessary lemmas II}, we have
\begin{align*}
\ytableausetup{boxsize=0.81em}
\C T^\leftarrow_{\beta}
\cong \C ( \opi_2 \opi_3 \opi_4 \opi_1 \opi_2 \opi_3 \opi_7 \cdot \TLba{\beta}{\balalpp{2}})
\qquad\text{and}\qquad
\C T^\leftarrow_{\gamma}
\cong \C (\opi_7 \cdot \TLba{\gamma}{\balalpp{2}}).
\end{align*}
\end{example}

From now on, suppose that $n \ge 3$.
Fix $l \in [2, n-1]$ and $c \in [2,l]$.
For $\omega \in \PBQ{n}{l}{c}$, let $\DeltaOmega{\omega}$ be the permutation in $\PBQ{n}{l}{c}$ such that
$\DeltaOmega{\omega}(i)=\omega(1) + i-1$ for $1\le i \le c$.
Then we consider the map
\begin{equation*}
\phi: \PBQ{n}{l}{c}\ra (\SG_n)_{[l]},
\quad
\omega \mapsto \omega \DeltaOmega{\omega}^{-1}.
\end{equation*}
It can be easily seen that
\begin{equation}\label{decom of maximal quoteints}
\begin{aligned}
    &\bullet\, \phi(\omega)(i) = i \text{ for } 1 \leq i \leq \omega(1), \\
    &\bullet\, \phi(\omega)(\omega(1)+1) <\phi(\omega)(\omega(1)+2) < \cdots <\phi(\omega)(\omega(1)+c-1),\\
    &\bullet\,\phi(\omega)(\omega(1)+c) <\phi(\omega)(\omega(1)+c+1) < \cdots <\phi(\omega)(l+1)
\end{aligned}
\end{equation}
and particularly $\phi$ is an injective map.
Note that $\omega(1)$ can have values belonging to $[l-c+2]$.
For $1 \leq u \leq l-c+2$, \eqref{decom of maximal quoteints} implies that
\begin{align*}\label{Eq: image of phi omega(1)}
\phi\left(\{\omega\in \PBQ{n}{l}{c}: \omega(1)=u\}\right) = (\SG_n)^{(c+u-1)}_{[u+1,l]}.
\end{align*}
Here $(\SG_n)^{(l+1)}_{[u + 1, l]}$ is set to be $\{\id\}$.
Hence, letting $\deltau$ be the permutation in $\PBQ{n}{l}{c}$ such that $\deltau(i) = u+i-1$ for $1 \leq i \leq c$,
we have the following decomposition:
\begin{equation}\label{Eq: Bij PBQ}
\PBQ{n}{l}{c} = \bigsqcup_{1 \le u \le l - c + 2}
\left\{\zeta  \deltau \mid  \zeta \in (\SG_n)^{(c + u - 1)}_{[u + 1, l]} \right\}.
\end{equation}

In the following, for each $\omega \in \PBQ{n}{l}{c}$,
we will show that $\pi_{\omega} = \pi_{\phi(\omega)} \pi_{\DeltaOmega{\omega}}$.
Note that
\begin{align}\label{Eq: length of Delta and omega}
\ell(\DeltaOmega{\omega}) = c(\omega(1)-1)
\quad \text{and} \quad \ell(\omega) = \sum_{1 \leq i \leq c} (\omega(i)-i).
\end{align}
Since $\phi(\omega) \in (\SG_n)_{[\omega(1) + 1, l]}^{(c + \omega(1) - 1)}$,
\[
\ell(\phi(\omega))
= \sum_{\omega(1) + 1 \le i \le \omega(1) + c - 1} \left( \phi(\omega)(i) - i \right)
= \sum_{1 \le i \le c -1} \left( \phi(\omega)(\omega(1) + i) - \omega(1) - i \right).
\]
From the construction of $\phi$ one sees that $\phi(\omega)(\omega(1) + i) = \omega(i + 1)$, thus
\[
\ell(\phi(\omega))
= \sum_{1 \le i \le c -1} \left( \omega(i + 1) - \omega(1) - i \right).
\]
Combining this equality with~\eqref{Eq: length of Delta and omega} yields that
\begin{align*}
\ell(\phi(\omega)) + \ell(\DeltaOmega{\omega})
& = \sum_{1 \le i \le c -1} \left( \omega(i + 1) - \omega(1) - i \right) + c(\omega(1)-1)  \\
& = \sum_{1 \leq i \leq c} (\omega(i)-i)
= \ell(\omega).
\end{align*}
Since $\omega = \phi(\omega) \DeltaOmega{\omega}$, we have that $\DeltaOmega{\omega} \preceq_L \omega$, thus
\begin{equation}\label{Eq: Delta leq Omega}
\pi_{\omega} = \pi_{\phi(\omega)} \pi_{\DeltaOmega{\omega}}.
\end{equation}

Let $1 \le j \le m$ and $\beta \in [\balalpp{j}]$.
For $\sigma \preceq_L \longeltt{\beta}{j}$,
we define $\sigmatab{\sigma}{\beta}{j}$ to be the filling  of $\trd(\balalp)$
such that the column strip $\tS_{k_r}$ ($1 \leq r \leq m$) is filled with the entries of
$$
\begin{cases}
(\pi_\sigma \cdot \TLba{\beta}{\balalpp{j}})(\tS'_{k_j}) \cup \{\min((\pi_\sigma \cdot \TLba{\beta}{\balalpp{j}})(\tS'_{k_0}))\} & \text{ if } r = j,\\
(\pi_\sigma \cdot \TLba{\beta}{\balalpp{j}})(\tS'_{k_r}) & \text{ otherwise }
\end{cases}
$$
in such a way that the entries are increasing from top to bottom
and the row strip $\tS_{k_0}$ is filled with the entries of
$$
\left(
(\pi_\sigma \cdot \TLba{\beta}{\balalpp{j}})(\tS'_{k_{-1}}) \cup (\pi_\sigma \cdot \TLba{\beta}{\balalpp{j}})(\tS'_{k_{0}})
\right) \setminus \{\min((\pi_\sigma \cdot \TLba{\beta}{\balalpp{j}})(\tS'_{k_0}))\}
$$
in such a way that the entries are increasing from left to right.

\begin{example}\label{Ex: tausigmabetaj}
Let us revisit Example~\ref{Ex: Lemma Cbeta cong Cw0}.
Recall $\beta=(1,3,1^4)$ and $\balalpp{2} =(1) \oplus (3,1^2) \oplus (1^2)$.
For $\sigma = s_{[1,3]}, s_4s_{[1,3]}$, and $s_{[3,4]}s_{[1,3]}$,
it holds that $\sigma \preceq_L \longeltt{\beta}{2}$ and
\begin{align*}
\pi_{[1,3]} \cdot \TLba{\beta}{\balalpp{2}}
=
\scalebox{0.9}{$
\begin{ytableau} 
\none & \none & \none & \none & 1 & 5 \\
\none & 2 & 3 & 4\\
\none & 6 \\
\none & 8 \\
7
\end{ytableau}
$}, \
\pi_{4}\pi_{[1,3]} \cdot \TLba{\beta}{\balalpp{2}}
=
\scalebox{0.9}{$
\begin{ytableau} 
\none & \none & \none & \none & 1 & 4 \\
\none & 2 & 3 & 5\\
\none & 6 \\
\none & 8 \\
7
\end{ytableau}
$}, \
\pi_{[3,4]}\pi_{[1,3]} \cdot \TLba{\beta}{\balalpp{2}}
=
\scalebox{0.9}{$
\begin{ytableau} 
\none & \none & \none & \none & 1 & 3 \\
\none & 2 & 4 & 5\\
\none & 6 \\
\none & 8 \\
7
\end{ytableau}
$}.
\end{align*}
Using these, we can check that
\begin{displaymath}
\sigmatab{\sigma}{(1,3,1^4)}{2} =
\scalebox{0.9}{$
\begin{ytableau}
\none & \none & 1 & 3 & 4 & 5\\
\none & \none & 6 \\
\none & \none & 8 \\
\none & 2 \\
7
\end{ytableau}
$}
\end{displaymath}
for all $\sigma = s_{[1,3]},  s_4s_{[1,3]},s_3s_4s_{[1,3]}$.
\end{example}

If there is no confusion for $j$ and $\beta$, then we simply write $T(\sigma)$ for $\sigmatab{\sigma}{\beta}{j}$.
For $\Theta(\calV_\alpha)$ defined in~\eqref{Eq: generators of cosyzygy}, we have the following lemma.

\begin{lemma}\label{Lem: tausigmabetaj}
Suppose that we have a pair $(j,\beta)$ with $1 \le j \le m$ and $\beta \in [\balalpp{j}]$
satisfying that $\min(\tJ{\beta}{\balalpp{j}}) \leq \ell(\alpha)$.
Then, for every permutation $\sigma \in \SG_n$ with $\sigma \preceq_L \longeltt{\beta}{j}$, it holds that
$T(\sigma) \in \Theta(\calV_\alpha)$.
\end{lemma}

\begin{proof}
It is clear that $\sigmatabsh{\sigma} \in \SRT(\balalp)$.
Thus, for the assertion, we have only to show that $\sigmatabsh{\sigma}_j^{1+\delta_{j,m}} < \sigmatabsh{\sigma}^1_{m+k_j-1}$.
Note that
\begin{equation*}
(\pi_\sigma \cdot \TLba{\beta}{\balalpp{j}})(\tS'_{k_{-1}}) \cup  (\pi_\sigma \cdot \TLba{\beta}{\balalpp{j}})(\tS'_{k_0})
= \{1,2,\ldots,\ell(\alpha)+1\},
\end{equation*}
which implies that
\[1 \leq \min((\pi_\sigma \cdot \TLba{\beta}{\balalpp{j}})(\tS'_{k_0})) \leq |\tS'_{k_{-1}}|+1.\]
Since $|\tS'_{k_{-1}}| = k_j-1$, it follows that
$\min((\pi_\sigma \cdot \TLba{\beta}{\balalpp{j}})(\tS'_{k_0})) \leq k_j$.
On the other hand, from the observation that $\sigmatabsh{\sigma}^1_{m+k_j-1}$ is
the $k_j$th smallest element in the set
\[\{1,2,\ldots,\ell(\alpha)+1\}\setminus\{\min(\pi_\sigma \cdot \TLba{\beta}{\balalpp{j}}(\tS'_{k_0}))\},\]
we see that
$k_j <  \sigmatabsh{\sigma}^1_{m+k_j-1}$.
As a consequence, we derive the following inequality:
\begin{equation*}
\sigmatabsh{\sigma}^{1+\delta_{j,m}}_j = \min((\pi_\sigma \cdot \TLba{\beta}{\balalpp{j}})(\tS'_{k_0})) \leq k_j <  \sigmatabsh{\sigma}^1_{m+k_j-1}. \qedhere
\end{equation*}
\end{proof}

We are now ready to prove Theorem~\ref{Thm: Main Section4}.

\begin{proof}[Proof of Theorem~\ref{Thm: Main Section4}]
(a)
Given $1 \le i \le n-1$ and $T \in \SRT(\balalp)$, we have three cases.

{\it Case 1:} $\pi_i \cdot T = T$.
We claim that $i \notin \Des(\tau_{T;j;A})$ for all $1 \leq j \leq m$ and $A \in \setA{T}{j}$.
Fix $j \in [m]$ and $A \in \setA{T}{j}$.
Since $i \notin \Des(T)$,  $i$ is weakly right of $i+1$ in $T$.
If neither $i$ nor $i+1$ appear in $\bfw_{T;j}$, then $i$ and $i+1$ still hold their positions in $\tau_{T;j;A}$, so $i \notin \Des(\tau_{T;j;A})$.
If $i$ appears in $\bfw_{T;j}$ and $i+1$ does not appear in $\bfw_{T;j}$,
then $i+1$ holds its position in $\tau_{T;j;A}$ but $i$ is moved to the right in $\tau_{T;j;A}$, so $i \notin \Des(\tau_{T;j;A})$.
Suppose that $i$ does not appear in $\bfw_{T;j}$ and $i+1 = w_{u_k}$ for some $1 \le k \le l$, where $\bfw_{T;j} = w_{u_1} w_{u_2} \cdots w_{u_l}$.
By the definition of $\bfw_{T;j}$, $w_{u_{k+1}} < i$ and appears strictly left of $i$ if $k < l$, and $i \in T(\tS_{k_0})$ if $k = l$.
Thus, $i \notin \Des(\tau_{T;j;A})$.

{\it Case 2:} $\pi_i \cdot T = 0$.
We claim that $\pi_i \cdot \bstau{T}{j} = 0$ for all $1 \leq j \leq m$.
Fix $j \in [m]$.
Since  $i, i+1 \in T(\tS_0)$ by the shape of $T$, $\ebfw{T}{j} \neq i, i+1$.
So we have from the definition of $\tA_{T;j}$ that either $i, i+1 \notin \tA_{T;j}$ or $i, i+1 \in \tA_{T;j}$.
If $i, i+1 \notin \tA_{T;j}$, then $i,i+1 \in \tau_{T;j;A}(\tS'_{k_{-1}})$ for all $A \in \setA{T}{j}$, so $\pi_i \cdot \bstau{T}{j} = 0$.
If $i, i+1 \in \tA_{T;j}$, then $\setA{T}{j} = \mathcal{X} \cup \mathcal{Y} \cup \mathcal{Z}$, where
\begin{align*}
\mathcal{X} & := \{ A \in \setA{T}{j} \mid i \in A, \ i+1 \notin A \} \\
\mathcal{Y} & := \{ A \in \setA{T}{j} \mid i \notin A, \ i+1 \in A \} \\
\mathcal{Z} & :=\{ A \in \setA{T}{j} \mid i,i+1 \in A \} \cup \{ A \in \setA{T}{j} \mid i,i+1 \notin A \}
\end{align*}
Note that $\pi_i \cdot \tau_{T;j;A} = 0$ for any $A \in \mathcal{Z}$.
Therefore, the claim can be shown by proving that
\begin{align}\label{eq: pi_i sum = 0}
\pi_i \left(\sum_{A \in \mathcal{X}} \sgn(A)\tau_{T;j;A} + \sum_{A \in \mathcal{Y}} \sgn(A)\tau_{T;j;A} \right) = 0.
\end{align}
Let us consider the bijection $f: \mathcal{X} \ra \mathcal{Y}$ by
$$
A \mapsto (A\setminus{\{i\}})\cup{\{i+1\}}.
$$
Since $\sgn(A) + \sgn(f(A)) = 0$ and $\tau_{T;j;f(A)} = s_i \cdot \tau_{T;j;A}$, we obtain~\eqref{eq: pi_i sum = 0}.

{\it Case 3:} $\pi_i \cdot T = s_i \cdot T$.
We claim that $\pi_i \cdot \bstau{T}{j} = \bstau{(\pi_i \cdot T)}{j}$ for all $1 \leq j \leq m$.
Fix $1 \leq j \leq m$ with $\bstau{T}{j} \neq 0$.
If $i+1 \notin T(\tS_{k_0})$, then $\ebfw{T}{j} = \ebfw{\pi_i \cdot T}{j}$
and $\tA_{T;j} = \tA_{(\pi_i \cdot T);j}$, so $\setA{T}{j} = \setA{(\pi_i \cdot T)}{j}$. This implies that
$$
\pi_i \cdot \bstau{T}{j} = \pi_i \left(\sum_{A\in \setA{T}{j}} \sgn(A) \tau_{T;j;A}\right) = \sum_{A\in \setA{\pi_i \cdot T}{j}} \sgn(A) \tau_{\pi_i \cdot T;j;A} = \bstau{(\pi_i\cdot T)}{j}.
$$

Let us assume that $i+1 \in T(\tS_{k_0})$.
First, we consider the case where $\ebfw{T}{j} = i$.
Combining the assumption $\bstau{T}{j} \neq 0$ with~\eqref{Eq: bstau neq 0} yields that $T^1_{m+k_j-1} > i$.
In addition, for any $A \in \setA{T}{j}$ with $i+1 \in A$, we have $\pi_i \cdot \tau_{T;j;A} = 0$.
Therefore,
\begin{equation}\label{eq: pi cdot bstau}
\pi_i \cdot \bstau{T}{j}  
= 
\sum_{\substack{A \in \setA{T}{j} \\ i+1 \notin A} } \sgn(A) \ \pi_i \cdot \tau_{T;j;A}.
\end{equation}
On the other hand, since $\ebfw{\pi_i \cdot T}{j} = i+1$, we have
\[
\setA{\pi_i \cdot T}{j} = \{A \in \setA{T}{j} \mid i+1 \notin A\}.
\]
This implies that  
\begin{align}\label{Eq: bstau def}
\bstau{\pi_i \cdot T}{j} = \sum_{A \in \setA{\pi_i \cdot T}{j}} \sgn(A) \tau_{\pi_i \cdot T;j;A}
= \sum_{\substack{A \in \setA{T}{j} \\ i+1 \notin A} }
 \sgn(A) \ \tau_{\pi_i \cdot T;j;A}.
\end{align}
For any $A \in \setA{T}{j}$ with $i+1 \notin A$, one can see that $\pi_i \cdot \tau_{T;j;A} = \tau_{\pi_i \cdot T;j;A}$.
Combining this equality with the equalities~\eqref{eq: pi cdot bstau} and~\eqref{Eq: bstau def}, we have $\pi_i \cdot \bstau{T}{j} = \bstau{\pi_i \cdot T}{j}$.

Next, we consider the case where $\ebfw{T}{j} \neq i$.
Then one sees that
$$
\tA_{(\pi_i \cdot T);j} =
\begin{cases}
\tA_{T;j}  & \text{ if } \ebfw{T}{j} > i, \\
(\tA_{T;j} \setminus \{i+1\})\cup \{ i\} & \text{ if }  \ebfw{T}{j} < i.
\end{cases}
$$
In the former case,
one can see that $\pi_i \cdot \bstau{T}{j} = \bstau{(\pi_i \cdot T)}{j}$ by mimicking the proof of the case where $i+1 \notin T(\tS_{k_0})$.
For the latter case, set
$$
f: \setA{T}{j} \rightarrow \setA{(\pi_i \cdot T)}{j}, \quad
A \mapsto f(A) :=
\begin{cases}
(A \setminus \{i+1 \})\cup \{ i\} & \text{ if } i+1 \in A, \\
A & \text{ otherwise.}
\end{cases}
$$
It is clear that $f$ is bijective.
Moreover, since $\sgn(A) = \sgn(f(A))$ and $\pi_i \cdot \tau_{T;j;A} = \tau_{(\pi_i\cdot T);j;f(A)}$, it follows that
\[
\pi_i \cdot \bstau{T}{j}
=\sum_{A \in \setA{T}{j}} \sgn (A) \pi_i \cdot \tau_{T;j;A}
=\sum_{f(A) \in \setA{(\pi_i \cdot T)}{j}} \sgn (f(A)) \tau_{(\pi_i\cdot T);j;f(A)}
=  \bstau{(\pi_i \cdot T)}{j}.
\]

(b)
Let us show $\ker(\partial^1) \supseteq \injhull(\calV_\alpha)$.
Recall that
$$
\injhull(\calV_\alpha) = \C \{ T \in \SRT(\balalp) \mid T_j^{1+\delta_{j,m}} > T^1_{m+k_j-1} \text{ for all } 1\leq j \leq m \}.
$$
Therefore, it suffices to show that
$$
\ker(\partial^1) \supseteq \{ T \in \SRT(\balalp) \mid T_j^{1+\delta_{j,m}} > T^1_{m+k_j-1} \text{ for all } 1\leq j \leq m \}.
$$
Let $T \in \{ T \in \SRT(\balalp) \mid T_j^{1+\delta_{j,m}} > T^1_{m+k_j-1} \text{ for all } 1\leq j \leq m \}$.
For every $1\leq j \leq m$, there exists $j'>j$ such that
$\ebfw{T}{j} = T_{j'}^{1+\delta_{j',m}}$.
By definition one has
$$
T_{j'}^{1+\delta_{j',m}} > T^1_{m+k_{j'}-1} > T^1_{m+k_{j}-1},
$$
so $\setA{T}{j} = \emptyset$. By definition $\bstau{T}{j} = 0$, thus $T \in \ker(\partial^1)$.
\smallskip

Let us show  $\ker(\partial^1) \subseteq \injhull(\calV_\alpha)$.
Suppose that there exists $x \in \ker(\partial^1) \setminus \injhull(\calV_\alpha)$.
Let $x = \sum_{T \in \SRT(\balalp)} c_T T$ with $c_T \in \C$.
Since $\partial^1(T) = 0$ for all $T$ satisfying that $T_j^{1+\delta_{j,m}} > T^1_{m+k_j-1}~(1\leq j \leq m)$, all $T$'s in the expansion of $x$ are contained in $\Theta(\calV_\alpha)$ (see~\eqref{Eq: generators of cosyzygy}).
Define
\[
\supp(x) := \{T \in \Theta(\calV_\alpha) \mid c_T \neq 0\}
\]
and
choose any tableau $U$ in $\supp(x)$ such that $\bfw(U)$ is maximal in $\{\bfw(T): T \in \supp(x)\}$
with respect to the Bruhat order.
Let
\begin{align*}
J & := \{j \in [m] \mid \setA{U}{j} \neq \emptyset \} \text{ and}\\
\tau_0 &:= \tau_{U;\max(J);A^1_{U; \max(J)}}.
\end{align*}
It should be noted that $J$ is nonempty because $U \in \Theta(\calV_\alpha)$ and the coefficient of $\tau_0$ is nonzero in the expansion of $\partial^1(U)$ in terms of $\bigcup_{1 \le j \le m} \SRT(\balalpp{j})$.
Note that $\partial^1(x) = \partial^1(c_U U) + \partial^1(x - c_U U)$
and
\begin{align*}
\partial^1(x - c_U U)
& = \sum_{T \in \supp(x) \setminus \{U\}} c_T
\left( \sum_{1\le j \le m} \bstau{T}{j} \right) \\
& = \sum_{T \in \supp(x) \setminus \{U\}} c_T
\left( \sum_{1\le j \le m} \sum_{A \in \setA{T}{j}} \sgn(A) \tau_{T;j;A} \right).
\end{align*}
We claim that there is no triple $(T,j,A)$ with $T \in \supp(x) \setminus \{U\}$, $1\le j \le m$, and $A \in \setA{T}{j}$ such that $\tau_{T;j;A} = \tau_0$.
Suppose not, that is, $\tau_0 = \tau_{T;j;A}$ for some $(T,j,A)$.
Comparing the shapes of $\tau_0$ and $\tau_{T;j;A}$, we see that $j$ must be $\max(J)$.
Let $\bfw(T) = w_1 w_2 \cdots w_n$.
According to the definition of $\bfw_{T;\max(J)}$ in~\eqref{Eq: w(T;j)}, it is a decreasing subword $w_{u_1} w_{u_2} \cdots w_{u_l}$ of $\bfw(T)$
subject to the conditions:
\begin{align}\label{eq: w_u comparing}
w_{u_r} < w_{i} \quad \text{for all $1 \le r < l$ and $u_r < i <u_{r+1}$.}
\end{align}
Since $\tau_{T;\max(J);A} = \tau_0$, one has that
\[
\bfw(T) =
\bfw(U) \cdot (u_{1}~u_l) (u_{1}~u_{l-1}) \cdots (u_1~u_2),
\]
where $\bfw(T), \bfw(U)$ are viewed as permutations and $(a~b)$ denote a transposition.
For $\sigma \in \SG_n$ and $a,b \in [n]$,
it is stated in \cite[Lemma 2.1.4]{05BB} that $\sigma \prec \sigma \cdot (a \ b)$ and $\ell(\sigma \cdot (a \ b)) = \ell(\sigma) + 1$ if and only if $\sigma(a) < \sigma(b)$ and there is no $c$ such that $\sigma(a) < \sigma(c) < \sigma(b)$. Here $\prec$ is the Bruhat order.
Combining this with~\eqref{eq: w_u comparing} yields that $\bfw(U) \prec \bfw(T)$.
This contradicts the maximality of $U$, thus our claim is verified.
It tells us that the coefficient of $\tau_0$ in the expansion of $\partial^1(x)$ in terms of $\bigcup_{1 \le j \le m} \SRT(\balalpp{j})$ is nonzero, which is absurd by the assumption that
$x \in \ker(\partial^1)$.
Consequently, we can conclude that there is no $x \in \ker(\partial^1) \setminus \injhull(\calV_\alpha)$.
\medskip

(c) Observe the following $H_n(0)$-module isomorphisms:
\begin{align*}
\soc \left(\bigoplus_{1 \leq j \leq m}\bfP_{\balalpp{j}} \right)
& \underset{\text{Theorem~\ref{thm: bfP isom to calP}}}{\cong} \bigoplus_{1 \leq j \leq m} \bigoplus_{\beta \in [\balalpp{j}]} \soc(\bfP_{\beta})
\cong \bigoplus_{1 \leq j \leq m}\bigoplus_{\beta \in [\balalpp{j}]} \C T^{\leftarrow}_{\beta} \\
& \hspace{0.55ex} \underset{\text{Lemma~\ref{Lem: nessary lemmas II}}}{\cong}
\C \left(
\opi_{\longeltt{\beta}{j}} \cdot \TLba{\beta}{\balalpp{j}}
\right)
\end{align*}
Hence our assertion can be verified by showing that
$\opi_{\longeltt{\beta}{j}} \cdot \TLba{\beta}{\balalpp{j}} \in \mathrm{Im}(\opone)$ for $1 \le j \le m$ and $\beta \in [\balalpp{j}]$.
Let us fix $j \in [m]$ and $\beta \in [\balalpp{j}]$.
To begin with, we note that
\begin{equation}\label{Eq: 616}
\opi_{\longeltt{\beta}{j}} \cdot \TLba{\beta}{\balalpp{j}}
 = \sum_{\sigma \preceq_L \longeltt{\beta}{j}} (-1)^{\ell(\longeltt{\beta}{j}) - \ell(\sigma)} \pi_\sigma \cdot \TLba{\beta}{\balalpp{j}}.
\end{equation}
According to the definition of $\longeltt{\beta}{j}$, we divide into the following two cases.
\smallskip

{\it Case 1: $\min(\tJ{\beta}{\balalpp{j}}) > \ell(\alpha)$.}
For $\sigma \preceq_L  \longeltt{\beta}{j} = \paralong{\htJ{\beta}{\balalpp{j}}}$,
it holds that
\begin{align}\label{eq: kappa j'}
\begin{aligned}
& \sigmatabsh{\sigma}^{1+\delta_{j,m}}_{j} = |\tS'_{k_{-1}}|+1, \\
& \sigmatabsh{\sigma}^1_{m+k_{j}-1}=|\tS'_{k_{-1}}|+2 \text{ and } \\
& \sigmatabsh{\sigma}^{1+\delta_{j',m}}_{j'} > \sigmatabsh{\sigma}^1_{m+k_{j'}-1} \qquad \text{ if }1 \leq j' \leq m \text{ and } j'\neq j.
\end{aligned}
\end{align}
Moreover, the definition of $\sigmatabsh{\sigma}$ says that
\begin{align}\label{eq: P(A_kappa j')}
& \setA{\sigmatabsh{\sigma}}{j} = \left\{A^1:=\left[|\tS'_{k_{-1}}|+2,|\tS'_{k_{-1}}|+|\tS'_{k_0}|\right]\right\}.
\end{align}
Putting these together, we can derive the following equalities:
\begin{equation}
\begin{aligned}\label{eq: image of kappa under the partial1}
\opone(\sigmatabsh{\sigma} + \injhull(\calV_\alpha))
&= \sum_{1 \leq r \leq m}\bstau{\sigmatabsh{\sigma}}{r}\\
&=\bstau{\sigmatabsh{\sigma}}{j} 
&\text{(by \eqref{eq: kappa j'})}\;\\
&=\tau_{\sigmatabsh{\sigma};j;A^1} 
&\text{(by \eqref{eq: P(A_kappa j')})}.
\end{aligned}
\end{equation}
Since $\tau_{\sigmatabsh{\sigma};j;A^1} = \pi_\sigma \cdot \tau_{\sigmatabsh{\id};j;A^1}$ and $\tau_{\sigmatabsh{\id};j;A^1} = \TLba{\beta}{\balalpp{j}}$,
we see that
\begin{align}\label{eq: tau kappa sigma}
\opone(\sigmatabsh{\sigma} + \injhull(\calV_\alpha))
 = \pi_\sigma \cdot \TLba{\beta}{\balalpp{j}}.
\end{align}
Finally, putting \eqref{Eq: 616} and \eqref{eq: tau kappa sigma} together yields that
\begin{align*}
\opi_{\longeltt{\beta}{j}} \cdot \TLba{\beta}{\balalpp{j}} =
\sum_{\sigma \preceq_L \longeltt{\beta}{j}}
(-1)^{\ell(\longeltt{\beta}{j}) - \ell(\sigma)}
\, \opone(\sigmatabsh{\sigma} + \injhull(\calV_\alpha)),
\end{align*}
which verifies the assertion.
\smallskip

{\it Case 2: $\min(\tJ{\beta}{\balalpp{j}}) \le \ell(\alpha)$.}
Let $\sigma \preceq_L \longeltt{\beta}{j}$.
Since
\[\longeltt{\beta}{j} = \pqlong{[\ell(\alpha)]}{|\tS'_{k_0}|}
\cdot
\paralong{\htJ{\beta}{\balalpp{j}}} \text{ and }\min(\htJ{\beta}{\balalpp{j}}) > \ell(\alpha) + 1,\]
we can write $\sigma$ as $\sigma' \sigma''$ for some
$\sigma' \in \SGG{n}{\htJ{\beta}{\balalpp{j}}}$ and $\sigma'' \in \SGGG{n}{[\ell(\alpha)]}{(|\tS'_{k_0}|)}$.
Therefore, the right hand side of~\eqref{Eq: 616} can be rewritten as
\begin{align}\label{Eq: sum over preceq w0}
\sum_{\sigma \preceq_L \longeltt{\beta}{j}}
(-1)^{ \ell(
\paralong{\htJ{\beta}{\balalpp{j}}}) + \ell(
\pqlong{[\ell(\alpha)]}{|\tS'_{k_0}|}) -(\ell(\sigma') + \ell(\sigma''))} \pi_{\sigma'} \pi_{\sigma''} \cdot \TLba{\beta}{\balalpp{j}}.
\end{align}
Since $\{\sigma \in \SG_n \mid \sigma \preceq_L \longeltt{\beta}{j}\}$ can be decomposed into
\[
\bigsqcup_{\sigma' \in \SGG{n}{\htJ{\beta}{\balalpp{j}}}}
\bigsqcup_{\sigma'' \in \SGGG{n}{[\ell(\alpha)]}{(|\tS'_{k_0}|)}} \{\sigma' \sigma''\},
\]
\eqref{Eq: sum over preceq w0} can also be rewritten as
\begin{align}\label{eq: sum of pi sig' pi sig''}
\sum_{\sigma' \in \SGG{n}{\htJ{\beta}{\balalpp{j}}}}
(-1)^{\calN(\sigma')}
\pi_{\sigma'}
\underbrace{\sum_{\sigma'' \in \SGGG{n}{[\ell(\alpha)]}{(|\tS'_{k_0}|)}}
(-1)^{\calM(\sigma'')} \pi_{\sigma''} \cdot \TLba{\beta}{\balalpp{j}}}_{(\mathsf{P})}.
\end{align}
Here we are using the notation
\begin{align*}
\calN(\sigma') := \ell(\paralong{\htJ{\beta}{\balalpp{j}}}) - \ell(\sigma')
\quad \text{ and } \quad
\calM(\sigma'') := \ell(\pqlong{[\ell(\alpha)]}{|\tS'_{k_0}|}) - \ell(\sigma'').
\end{align*}
Note that
$\ell(\alpha)-|\tS'_{k_0}|+2 = |\tS'_{k_{-1}}|+1$
since $\ell(\alpha)+1 = |\tS'_{k_0}|+|\tS'_{k_{-1}}|$.
In view of \eqref{Eq: Bij PBQ} and \eqref{Eq: Delta leq Omega}, we see that the summation $(\mathsf{P})$
in \eqref{eq: sum of pi sig' pi sig''} equals
\begin{align*}
\sum_{1 \le u \le |\tS'_{k_{-1}}|+1}
\sum_{\zeta \in \PBQ{n}{u+1,\ell(\alpha)}{|\tS'_{k_0}|+u-1}}
(-1)^{\calM(\zeta \deltau)}
\pi_{\zeta} \pi_{\deltau}
\cdot \TLba{\beta}{\balalpp{j}}.
\end{align*}
For each $1 \le u \le |\tS'_{k_{-1}}|+1$, we claim that
\begin{align*}
\sum_{\zeta \in \PBQ{n}{u+1,\ell(\alpha)}{|\tS'_{k_0}|+u-1}}
(-1)^{\calM(\zeta \deltau)} \pi_{\zeta \deltau} \cdot \TLba{\beta}{\balalpp{j}}
= (-1)^{|\tS'_{k_{-1}}|-u} \partial^1(\sigmatabsh{\deltau}),
\end{align*}
which will give rise to
\begin{align*}
\opi_{\longeltt{\beta}{j}} \cdot \TLba{\beta}{\balalpp{j}} \in \rmIm(\partial^1).
\end{align*}

The last of the proof will be devoted to the verification of this claim.
We fix $u \in [1,|\tS'_{k_{-1}}|+1]$ and observe that
\begin{align*}
&\sigmatabsh{\deltau}(\tS_{k_0})
=[\ell(\alpha) + 1]\setminus\{u\} \text{ and}\\
&\min
\left(\sigmatabsh{\deltau}(\tS_{k_{j'}})\right) > \ell(\alpha) + 1 \quad \text{if $1 \leq j' \leq m$ and $j' \neq j$}.
\end{align*}
This implies that $\sigmatabsh{\deltau}^{1+\delta_{j',m}}_{j'} > \sigmatabsh{\deltau}^1_{m+k_{j'}-1}$, and therefore
\begin{align}\label{eq: expansion of partial1}
\partial^1(\sigmatabsh{\deltau})
= \bstau{\sigmatabsh{\deltau}}{j}
= \sum_{A \in \setA{\sigmatabsh{\deltau}}{j}} \sgn(A) \tau_{\sigmatabsh{\deltau};j;A}.
\end{align}
Combining Lemma~\ref{Lem: tausigmabetaj} with~\eqref{Eq: bstau neq 0} shows that the summation
given in the last term is non-zero.
In what follows, we transform this summation
into a form suitable for proving our claim.
For this purpose, we need to analyze $\setA{\sigmatabsh{\deltau}}{j}$.
Since
$\tA_{\sigmatabsh{\deltau};j}  = [u+1,\ell(\alpha)+1]$ and $\ell(\alpha)-k_j+1 = |\tS'_{k_0}| - 1$,
it follows that
\begin{align*}
\setA{\sigmatabsh{\deltau}}{j}=\binom{[u+1,\ell(\alpha)+1]}{|\tS'_{k_0}| - 1}.
\end{align*}
Thus we have the natural bijection
\[
\psi: \setA{\sigmatabsh{\deltau}}{j} \ra (\SG_n)^{(|\tS'_{k_0}|-1)}_{[\ell(\alpha)-u]},
\quad A = \{a_1 < a_2 < \cdots < a_{|\tS'_{k_0}|-1}\} \mapsto \psi(A),
\]
where $\psi(A)$ denotes the permutation in $(\SG_n)^{(|\tS'_{k_0}|-1)}_{[\ell(\alpha)-u]}$ such that
$\psi(A)(i) = a_i - u$ for $1 \leq i \leq |\tS'_{k_0}|-1$.
Recall that there is a natural right $\SG_{|\tA_{\sigmatabsh{\deltau};j}|}$-action on $\tA_{\sigmatabsh{\deltau};j}$ given by~\eqref{eq: SG action on ATj}.
Put
$$
\AZero:=[u+1,u+|\tS'_{k_0}|-1].
$$
Since $|\tA_{\sigmatabsh{\deltau};j}| = \ell(\alpha) - u + 1$, we may identify $\SG_{|\tA_{\sigmatabsh{\deltau};j}|}$ with $(\SG_n)_{[\ell(\alpha)-u]}$.
Note that $\psi(A)$ is the unique permutation in $(\SG_n)^{(|\tS'_{k_0}| - 1)}_{[\ell(\alpha)-u]}$ that gives $A^0$ when acting on $A$, that is, $A \cdot \psi(A) = A^0$.
Since
\[
\AZero \cdot \psi(A)^{-1}
= \left(A^1_{\sigmatabsh{\deltau};j}\cdot
\pqlong{[\ell(\alpha)-u]}{|\tS'_{k_0}|-1}
^{-1}\right) \cdot \psi(A)^{-1},
\]
we have that
\begin{align*}
\sgn(A) = (-1)^{\ell(\pqlong{[\ell(\alpha)-u]}{|\tS'_{k_0}|-1}) - \ell(\psi(A))}.
\end{align*}
Applying this identity to~\eqref{eq: expansion of partial1} yields that
\begin{align}\label{eq: partial 1 expansion to sigmatabsh}
\partial^1(\sigmatabsh{\deltau})
& =
\sum_{A \in \setA{\sigmatabsh{\deltau}}{j}}
(-1)^{\ell(\pqlong{[\ell(\alpha)-u]}{|\tS'_{k_0}|-1}) -\ \ell(\psi(A))}
\tau_{\sigmatabsh{\deltau};j;A}.
\end{align}
Consider the bijection
$$
\shiftu: (\SG_n)^{(|\tS'_{k_0}| - 1)}_{[\ell(\alpha) - u]} \ra
(\SG_n)^{(|\tS'_{k_0}|-1+u)}_{[u+1,\ell(\alpha)]},
\quad
s_i \mapsto s_{i + u}.
$$
From the constructions of $\sigmatabsh{\deltau}$ and $\tau_{\sigmatabsh{\deltau};j;\AZero}$ we can derive the identities:
\begin{align}\label{eq: tausigmatabsh A expansion}
\begin{aligned}
\tau_{\sigmatabsh{\deltau};j;A}
 = \tau_{\sigmatabsh{\deltau};j;(\AZero\cdot \psi(A)^{-1})}
 = \pi_{\shiftu(\psi(A))} \cdot \tau_{\sigmatabsh{\deltau};j;\AZero}  = \pi_{\shiftu(\psi(A))} \pi_{\deltau} \cdot \TLba{\beta}{\balalpp{j}}.
\end{aligned}
\end{align}
As a consequence,
\begin{align*}
\partial^1(\sigmatabsh{\deltau})
&
\overset{\eqref{eq: partial 1 expansion to sigmatabsh}}{=}
\sum_{A \in \setA{\sigmatabsh{\deltau}}{j}}
(-1)^{\ell(
\pqlong{[\ell(\alpha)-u]}{|\tS'_{k_0}|-1}
) -\ \ell(\psi(A))}
\tau_{\sigmatabsh{\deltau};j;A} \\
&
\overset{\eqref{eq: tausigmatabsh A expansion}}{=}
\sum_{A \in \setA{\sigmatabsh{\deltau}}{j}}
(-1)^{\ell(
\pqlong{[\ell(\alpha)-u]}{|\tS'_{k_0}|-1}
)-\ell(\shiftu(\psi(A)))}
\pi_{\shiftu(\psi(A))}\pi_{\deltau} \cdot \TLba{\beta}{\balalpp{j}}.
\end{align*}
Making use of the bijection
$\shiftu \circ \psi: \setA{\sigmatabsh{\deltau}}{j} \ra
(\SG_n)^{(|\tS'_{k_0}|-1+u)}_{[u+1,\ell(\alpha)]}$, we can rewrite the second summation as
\begin{align}\label{eq: G circ tphi transform}
\begin{aligned}
\sum_{\xi \in (\SG_n)^{(|\tS'_{k_0}|-1+u)}_{[u+1,\ell(\alpha)]}}
(-1)^{\ell(
\pqlong{[\ell(\alpha)-u]}{|\tS'_{k_0}|-1}
)-\ell(\zeta)}
\pi_{\zeta} \pi_{\deltau} \cdot \TLba{\beta}{\balalpp{j}}.
\end{aligned}
\end{align}
Note that
\begin{align*}
\ell\left(
\pqlong{[\ell(\alpha)-u]}{|\tS'_{k_0}|-1}
\right) -\ell(\zeta)
&=(|\tS'_{k_0}|-1)(\ell(\alpha)-u-|\tS'_{k_0}|+1) - \ell(\zeta)\\
&=(|\tS'_{k_0}|-1)(|\tS'_{k_{-1}}|-u) - \ell(\zeta)\\
&= \calM(\zeta\deltau) - |\tS'_{k_{-1}}|+u.
\end{align*}
By substituting $\calM(\zeta\deltau) - |\tS'_{k_{-1}}|+u$ for $\ell\left(
\pqlong{[\ell(\alpha)-u]}{|\tS'_{k_0}|-1}\right) -\ell(\zeta)$
in~\eqref{eq: G circ tphi transform},
we finally obtain that
\[
\partial^1(\sigmatabsh{\deltau}) =
(-1)^{|\tS'_{k_{-1}}| - u}
\sum_{\zeta \in (\SG_n)^{(|\tS'_{k_0}|-1+u)}_{[u+1,\ell(\alpha)]}}
(-1)^{\calM(\zeta\deltau)}
\pi_{\zeta} \pi_{\deltau} \cdot \TLba{\beta}{\balalpp{j}},
\]
as required.

(d)
It is well known that
\[\Ext_{H_n(0)}^1(\bfF_{\beta},\calV_\alpha)={\rm Hom}_{H_n(0)}(\bfF_{\beta}, \Omega^{-1}(\calV_\alpha))\]
(see \cite[Corollary 2.5.4]{91Benson}).
This immediately yields that
\[
\dim\Ext_{H_n(0)}^1(\bfF_{\beta},\calV_\alpha)=
[{\rm soc}(\Omega^{-1}(\calV_\alpha)):\bfF_{\beta}].
\]
By (c), one sees that ${\rm soc}(\Omega^{-1}(\calV_\alpha))$
equals the socle of $\bigoplus_{1 \leq j \leq m} \bfP_{\balalpp{j}}$.
So we are done.
\end{proof}

\section{Further avenues}

(a) For each $\alpha \models n$, let
\begin{equation}\label{eq: proj pres for F}
\begin{tikzcd}
P_1 \arrow[r, "\partial_1"] & \bfP_\alpha \arrow[r, "\epsilon"] & \bfF_\alpha \arrow[r] & 0
\end{tikzcd}
\end{equation}
be a minimal projective presentation of $\bfF_{\alpha}$.
From~\cite[Corollary 2.5.4]{91Benson} we know that $\dim\Ext_{H_n(0)}^1(\bfF_\alpha,\bfF_{\beta})$ counts the multiplicity of $\bfP_{\beta}$ in
the decomposition of $P_1$ into indecomposable modules, equivalently,
\begin{equation*}
P_1 \cong \bigoplus_{\beta \models n} \bfP_{\beta}^{\dim \Ext^1_{H_n(0)}(\bfF_{\alpha},\bfF_{\beta})}.
\end{equation*}
This dimension has been computed in~\cite[Section 4]{02DHT} and \cite[Theorem 5.1]{05Fayers}.
However, to the best of the authors' knowledge, no description for $\partial_1$ has not been available yet.
It would be nice to find an explicit description of $\partial_1$, especially in a combinatorial manner.
If this is done successfully, by taking an anti-automorphism twist introduced in \cite[Section 3.4]{21JKLO} to~\eqref{eq: proj pres for F}, we can also derive a minimal injective presentation for  $\bfF_{\alpha}$.

(b)
Besides dual immaculate functions, the problem of constructing $H_n(0)$-modules has been considered for the following quasisymmetric functions:
the \emph{quasisymmetric Schur functions} in ~\cite{15TW,19TW},
the \emph{extended Schur functions} in ~\cite{19Searles},
the \emph{Young row-strict quasisymmetric
Schur functions} in ~\cite{22BS},
the \emph{Young quasisymmetric Schur functions} in \cite{20CKNO2}, and the images of all these quasisymmetric functions under certain involutions on $\Qsym$ in \cite{21JKLO}.
Although these modules are built in a very similar way, their homological properties have not been well studied.
The study of their projective and injective presentations will be pursued in the near future with appropriate modifications to the method used in this paper.

(c) By virtue of Lemma~\ref{dimension of Hom} and Lemma~\ref{F-expansion od dual immaculate}, we have a combinatorial description for $\dim\Hom_{H_n(0)}(\bf P_\alpha, \calV_\beta)$.
However, no similar one is known for $\dim\Hom_{H_n(0)}(\calV_\alpha, \calV_\beta)$ except when $\beta \le_l \alpha$. It would be interesting to find such a description that holds for all $\alpha, \beta \models n$.


\vspace*{10mm}




\begin{thebibliography}{99}

\bibitem{95ARS}
M.~Auslander, I.~Reiten, and S.~Smal\o.
\newblock {\em Representation theory of {A}rtin algebras}, volume~36 of {\em
  Cambridge Studies in Advanced Mathematics}.
\newblock Cambridge University Press, Cambridge, 1995.

\bibitem{22BS}
J.~Bardwell and D.~Searles.
\newblock 0-{H}ecke modules for {Y}oung row-strict quasisymmetric {S}chur
  functions.
\newblock {\em European J. Combin.}, 102, 2022.

\bibitem{91Benson}
D.~J. Benson.
\newblock {\em Representations and cohomology. {I}}, volume~30 of {\em
  Cambridge Studies in Advanced Mathematics}.
\newblock Cambridge University Press, Cambridge, 1991.
\newblock Basic representation theory of finite groups and associative
  algebras.

\bibitem{14BBSSZ}
C.~Berg, N.~Bergeron, F.~Saliola, L.~Serrano, and M.~Zabrocki.
\newblock A lift of the {S}chur and {H}all-{L}ittlewood bases to
  non-commutative symmetric functions.
\newblock {\em Canad. J. Math.}, 66(3):525--565, 2014.

\bibitem{15BBSSZ}
C.~Berg, N.~Bergeron, F.~Saliola, L.~Serrano, and M.~Zabrocki.
\newblock Indecomposable modules for the dual immaculate basis of
  quasi-symmetric functions.
\newblock {\em Proc. Amer. Math. Soc.}, 143(3):991--1000, 2015.

\bibitem{17BBSSZ}
C.~Berg, N.~Bergeron, F.~Saliola, L.~Serrano, and M.~Zabrocki.
\newblock Multiplicative structures of the immaculate basis of non-commutative
  symmetric functions.
\newblock {\em J. Combin. Theory Ser. A}, 152:10--44, 2017.

\bibitem{16BSZ}
N.~Bergeron, J.~S\'{a}nchez-Ortega, and M.~Zabrocki.
\newblock The {P}ieri rule for dual immaculate quasi-symmetric functions.
\newblock {\em Ann. Comb.}, 20(2):283--300, 2016.

\bibitem{05BB}
A.~Bj\"{o}rner and F.~Brenti.
\newblock {\em Combinatorics of {C}oxeter groups}, volume 231 of {\em Graduate Texts in Mathematics}.
\newblock Springer, New York, 2005.

\bibitem{89Cab}
M.~Cabanes.
\newblock Extension groups for modular {H}ecke algebras.
\newblock {\em J. Fac. Sci. Univ. Tokyo Sect. IA Math.}, 36:347--362, 1989.

\bibitem{16Camp}
J.~M. Campbell.
\newblock Bipieri tableaux.
\newblock {\em Australas. J. Combin.}, 66:66--103, 2016.

\bibitem{17Camp}
J.~M. Campbell.
\newblock The expansion of immaculate functions in the ribbon basis.
\newblock {\em Discrete Math.}, 340(7):1716--1726, 2017.

\bibitem{20CKNO2}
S.-I. Choi, Y.-H. Kim, S.-Y. Nam, and Y.-T. Oh.
\newblock The projective cover of tableau-cyclic indecomposable ${H}_n(0)$-modules.
\newblock {\em Trans. Amer. Math. Soc.}, 375(11):7747--7782, 2022.

\bibitem{11Denton}
T.~Denton.
\newblock A combinatorial formula for orthogonal idempotents in the 0-{H}ecke
  algebra of the symmetric group.
\newblock {\em Electron. J. Combin.}, 18(1):Paper 28, 20, 2011.

\bibitem{02DHT}
G.~Duchamp, F.~Hivert, and J.-Y. Thibon.
\newblock Noncommutative symmetric functions. {VI}. {F}ree quasi-symmetric
  functions and related algebras.
\newblock {\em Internat. J. Algebra Comput.}, 12(5):671--717, 2002.

\bibitem{96DKLT}
G.~Duchamp, D.~Krob, B.~Leclerc, and J.-Y. Thibon.
\newblock Fonctions quasi-sym\'{e}triques, fonctions sym\'{e}triques non
  commutatives et alg\`ebres de {H}ecke \`a {$q=0$}.
\newblock {\em C. R. Acad. Sci. Paris S\'{e}r. I Math.}, 322(2):107--112, 1996.

\bibitem{05Fayers}
M.~Fayers.
\newblock {$0$}-{H}ecke algebras of finite {C}oxeter groups.
\newblock {\em J. Pure Appl. Algebra}, 199(1-3):27--41, 2005.

\bibitem{16GY}
A.~L.~L. Gao and A.~L.~B. Yang.
\newblock A bijective proof of the hook-length formula for standard immaculate
  tableaux.
\newblock {\em Proc. Amer. Math. Soc.}, 144(3):989--998, 2016.

\bibitem{17Grin}
D.~Grinberg.
\newblock Dual creation operators and a dendriform algebra structure on the
  quasisymmetric functions.
\newblock {\em Canad. J. Math.}, 69(1):21--53, 2017.

\bibitem{06HNT}
F.~Hivert, J.-C. Novelli, and J.-Y. Thibon.
\newblock Yang-{B}axter bases of {$0$}-{H}ecke algebras and representation theory of {$0$}-{A}riki--{K}oike--{S}hoji algebras.
\newblock {\em Adv. Math.}, 205(2):504--548, 2006.

\bibitem{16Huang}
J.~Huang.
\newblock A tableau approach to the representation theory of {$0$}-{H}ecke algebras.
\newblock {\em Ann. Comb.}, 20(4):831--868, 2016.

\bibitem{21JKLO}
W.-S. Jung, Y.-H. Kim, S.-Y. Lee, and Y.-T. Oh.
\newblock Weak {B}ruhat interval modules of the 0-{H}ecke algebra.
\newblock {\em Math. Z.}, 301:3755--3786, 2022.

\bibitem{97KT}
D.~Krob and J.-Y. Thibon.
\newblock Noncommutative symmetric functions. {IV}. {Q}uantum linear groups and {H}ecke algebras at {$q=0$}.
\newblock {\em J. Algebraic Combin.}, 6(4):339--376, 1997.

\bibitem{99Lam}
T.~Y. Lam.
\newblock {\em Lectures on modules and rings}, volume 189 of {\em Graduate Texts in Mathematics}.
\newblock Springer-Verlag, New York, 1999.

\bibitem{21MS}
S.~Mason and D.~Searles.
\newblock Lifting the dual immaculate functions.
\newblock {\em J. Combin. Theory Ser. A}, 184:Paper No. 105511, 52, 2021.

\bibitem{79Norton}
P.~N. Norton.
\newblock {$0$}-{H}ecke algebras.
\newblock {\em J. Austral. Math. Soc. Ser. A}, 27(3):337--357, 1979.

\bibitem{19Searles}
D.~Searles.
\newblock Indecomposable {$0$}-{H}ecke modules for extended {S}chur functions.
\newblock {\em Proc. Amer. Math. Soc.}, 148(5):1933--1943, 2020.

\bibitem{15TW}
V.~Tewari and S.~van Willigenburg.
\newblock Modules of the {$0$}-{H}ecke algebra and quasisymmetric {S}chur functions.
\newblock {\em Adv. Math.}, 285:1025--1065, 2015.

\bibitem{19TW}
V.~Tewari and S.~van Willigenburg.
\newblock Permuted composition tableaux, {$0$}-{H}ecke algebra and labeled binary trees.
\newblock {\em J. Combin. Theory Ser. A}, 161:420--452, 2019.
\end{thebibliography}
\end{document}